\documentclass[12pt]{article}

\usepackage{amsmath,amsthm}
\usepackage{amssymb}
\usepackage{amsfonts}
\usepackage{amscd}

 at 9truept
\font \smallrm=cmr10 at 9truept
\font \smallbf=cmbx10 at 9truept
\font \smallsl=cmsl10 at 9truept
 at 9truept

\newtheorem{theorem}{Theorem}[section]
\newtheorem{proposition}[theorem]{Proposition}
\newtheorem{definition}[theorem]{Definition}
\newtheorem{corollary}[theorem]{Corollary}
\newtheorem{lemma}[theorem]{Lemma}
\newtheorem{free text}[theorem]{}

\theoremstyle{definition}
\newtheorem{remark}[theorem]{Remark}

\newcommand \gerg {\mathfrak g}
\newcommand \gerh {\mathfrak h}
\newcommand \gerk {\mathfrak k}
\newcommand \gerl {\mathfrak l}
\newcommand \gert {\mathfrak t}
\newcommand \gerf {\mathfrak f}

\newcommand \gerI {{\mathfrak I}}
\newcommand \gerC {{\mathfrak C}}
\newcommand \calI {{\mathcal I}}
\newcommand \calC {{\mathcal C}}
\newcommand \calP {{\mathcal P}}
\newcommand \calH{{\mathcal H}}
\newcommand \gersl {\mathfrak{sl}}

\newcommand \HA{{\mathcal {HA}}}

\newcommand \U {\mathbb U}
\newcommand \F {\mathbb F}
\newcommand \N {\mathbb N}

\newcommand \C {\mathbb C}
\newcommand \Cq {\C (q)}
\newcommand \Cqqm {{\C\big[q, q^{-1}\big]}}
\newcommand \id {\text{\rm id}}
\newcommand \Ker {\text{\sl Ker}}

 \newcommand \ideal{\trianglelefteq}
 \newcommand\coideal {\dot{\trianglelefteq}}

\newcommand\Hh{\stackrel{\circ}{\text{H}}}
\newcommand\Kk{\stackrel{\circ}{\text{K}}}
\newcommand\gerho{\stackrel{\circ}{\gerh}}

 \numberwithin{equation}{section}

\begin{document}

{\ }

\vskip-33pt

 \centerline{\smallrm Documenta Mathematica  {\smallbf 19}  (2014), 333--380.   
 }
%
 \vskip1pt
   \centerline{\smallrm {\smallsl The original publication is available at\/}
\  http://www.math.uni-bielefeld.de/documenta/vol-19/vol-19.html}
 \phantom{\centerline{preprint  {\sl arXiv:1210.1597 [math.QA]\/}  (2012)}}

 \vskip45pt   {\ }

\centerline{\Large \bf A GLOBAL QUANTUM DUALITY}
 \vskip9pt
\centerline{\Large \bf PRINCIPLE FOR SUBGROUPS}
 \vskip9pt
\centerline{\Large \bf AND HOMOGENEOUS SPACES}

\vskip39pt

\centerline{ Nicola CICCOLI${}^\flat$, Fabio GAVARINI${}^\#$ }

\vskip9pt

\centerline{\it ${}^\flat$ Dipartimento di Matematica e Informatica, Universit\`a di Perugia }
\centerline{\it via Vanvitelli 1  --- I-06123 Perugia, Italy}
\centerline{{\footnotesize e-mail: ciccoli@dipmat.unipg.it}}

\vskip5pt

\centerline{\it ${}^\#$ Dipartimento di Matematica, Universit\`a di Roma ``Tor Vergata'' }
\centerline{\it via della ricerca scientifica 1  --- I-00133 Roma, Italy}
\centerline{{\footnotesize e-mail: gavarini@mat.uniroma2.it}}

\vskip57pt

\begin{abstract}
 {\footnotesize
   For a complex or real algebraic group  $ G $,  with  $ \, \gerg := \mathrm{Lie}(G) \, $,  quantizations of {\sl global\/}  type are suitable Hopf algebras  $ F_q[G] $  or  $ U_q(\gerg) $  over  $ \Cqqm \, $.  Any such quantization yields a structure of Poisson group on  $ G $,  and one of Lie bialgebra on  $ \gerg \, $:  correspondingly, one has dual Poisson groups  $ G^* $  and a dual Lie bialgebra  $ \gerg^* \, $.  In this context, we introduce suitable notions of  {\sl quantum subgroup\/}  and, correspondingly, of  {\sl quantum homogeneous space}, in three versions:  {\sl weak},  {\sl proper\/}  and  {\sl strict\/}  (also  called  {\sl flat\/}  in the literature).  The last two notions only apply to those subgroups which are coisotropic, and those homogeneous spaces which are Poisson quotients; the first one instead has no restrictions whatsoever.
                                                                \par
   The global quantum duality principle (GQDP), as developed in  \cite{Ga3},  associates with any global quantization of $ G \, $,  or of  $ \gerg \, $,  a global quantization of  $ \gerg^* $,  or of  $ G^* $.  In this paper we present a similar GQDP for quantum subgroups or quantum homogeneous spaces.  Roughly speaking, this associates with every quantum subgroup, resp.~quantum homogeneous space, of  $ G \, $,  a quantum homogeneous space, resp.~a quantum subgroup, of  $ G^* \, $.  The construction is tailored after four parallel paths --- according to the different ways one has to algebraically describe a subgroup or a homogeneous space ---   and is ``functorial'', in a natural sense.
                                                                            \par
   Remarkably enough, the output of the constructions are always quantizations of  {\sl proper\/}  type.  More precisely, the output is related to the input as follows: the former is the  {\it coisotropic dual\/}  of the coisotropic interior of the latter   --- a fact that extends the occurrence of Poisson duality in the original GQDP for quantum groups.  Finally, when the input
 is a strict quantization then the output is strict as well   --- so the special r{\^o}le of strict quantizations is respected.
                                                                            \par
   We end the paper with some explicit examples of application of our recipes.
%
%
 }
\end{abstract}

%
\vfill
 \eject
%
%

\section{Introduction}

   In this paper we work with quantizations of (algebraic) complex and real groups, their subgroups and homogeneous spaces, and a special symmetry among such quantum objects which we refer to as the ``Global Quantum Duality Principle''.  This is just a last step in a process, which is worth recalling in short.

\vskip5pt

   In any possible sense, quantum groups are suitable deformations of some algebraic objects attached with algebraic groups, or Lie groups.  Once and for all, we adopt the point of view of algebraic groups: nevertheless, all our analysis and results can be easily converted in the language of Lie groups.

\vskip5pt

   The first step to deal with is describing an algebraic group  $ G $  via suitable algebraic object(s).  This can be done following two main approaches, a  {\sl global\/}  one or a  {\sl local\/}  one.

\vskip5pt

   In the  {\sl global geometry\/}  approach,  one considers  $ U(\gerg) $   --- the universal enveloping algebra of the tangent Lie algebra  $ \, \gerg := \mathrm{Lie}(G) \, $  ---   and  $ F[G] $   --- the algebra of regular functions on  $ G \, $.  Both these are Hopf algebras, and there exists a non-degenerate pairing among them so that they are dual to each other.  Clearly,  $ U(\gerg) $  only accounts for the local data of  $ G $  encoded in  $ \gerg \, $,  whereas  $ F[G] $  instead totally describes  $ G \, $:  thus  $ F[G] $  yields a global description of  $ G \, $,  which is why we speak of ``global geometry'' approach.
                                                          \par
   In this context, one describes (globally) a subgroup  $ K $  of  $ G $   --- always assumed to be Zariski closed ---   via the ideal in  $ F[G] $  of functions vanishing on it; alternatively, an infinitesimal description is given taking in  $ U(\gerg) $  the subalgebra  $ U(\gerk) \, $,  where  $ \, \gerk := \mathrm{Lie}(K) \, $.
                                                          \par
   For a homogeneous  $ G $--space,  say  $ M $,  one describes it in the form  $ \, M \cong G\big/\!K \, $   --- which amounts to fixing some point in  $ M $  and its stabilizer subgroup  $ K $  in  $ G \, $.  After this, a local description of  $ \, M \cong G\big/\!K \, $  is given by representing its left-invariant differential operators as  $ \, U(\gerg) \big/ U(\gerg)\,\gerk \, $:  therefore, we can select  $ \, U(\gerg)\,\gerk \, $   --- a left ideal, left coideal in  $ U(\gerg) $  ---   as algebraic object to encode  $ \, M \cong G\big/\!K \, $,  at least infinitesimally.  For a global description instead, obstructions might occur.  Indeed, we would like to describe  $ \, M \cong G\big/\!K \, $  via some algebra  $ \, F[M] \cong F\big[G\big/\!K\big] \, $  strictly related with $ F[G] \, $.  This varies after the nature of  $ \, M \cong G\big/\!K \, $   --- hence of  $ K $  ---   and in general might be problematic.  Indeed, there exists a most natural candidate for this job, namely the set  $ \, {F[G]}^K \, $  of  $ K $--invariants  of  $ F[G] \, $,  which is a subalgebra and left coideal.  The problem is that  $ {F[G]}^K $  permits to recover exactly  $ G\big/\!K $  if and only if  $ \, M \cong G\big/\!K \, $  is a quasi-affine variety (which is not always the case).  This yields a genuine obstruction, in the sense that this way of (globally) encoding the space  $ \, M \cong G\big/\!K \, $  only works with quasi-affine  $ G $--spaces;  for the other cases, we just drop this approach   --- however, for a complete treatment of the case of  {\sl projective}  $ G $--spaces  see  \cite{CFG}.

\vskip5pt

   In contrast, the approach of  {\sl formal geometry\/}  is a looser one: one replaces  $ F[G] $  with a topological algebra  $ \, F[[G]] = F\big[\!\big[G_{\!f}\big]\!\big] \, $   --- the algebra of ``regular functions on the formal group  $ G_{\!f} $''  associated with  $ G $  ---   which can be realized either as the suitable completion of the local ring of  $ G $  at its identity or as the (full) linear dual of  $ U(\gerg) \, $.  In any case, both algebraic objects taken into account now only encode the local information of  $ G \, $.
                                                          \par
   In this formal geometry context, the description of (formal) subgroups and (formal) homogeneous spaces goes essentially the same.  However, in this case no problem occurs with (formal) homogeneous space, as any one of them can be described via a suitably defined subalgebra of invariants  $ \, {F\big[\!\big[G_{\!f}\big]\!\big]}^{K_{\!f}} $:  in a sense, ``all formal homogeneous spaces are quasi-affine''.  As a consequence, the overall description one eventually achieves is entirely symmetric.

\vskip5pt

   When dealing with quantizations, Poisson structures arise (as semiclassical limits) on groups and Lie algebras, so that we have to do with Poisson groups and Lie bialgebras.  In turn, there exist distinguished subgroups and homogeneous spaces   --- and their infinitesimal counterparts ---   which are ``well-behaving'' with respect to these extra structures: these are  {\sl coisotropic subgroups\/}  and  {\sl Poisson quotients}.  Moreover, the well-known Poisson duality   --- among Poisson groups  $ G $  and  $ G^* $  and among Lie bialgebras  $ \gerg $  and  $ \gerg^* $  ---   extends to similar dualities among coisotropic subgroups (of  $ G $  and  $ G^* $)  and among Poisson quotients (of  $ G $  and  $ G^* $  again).  It is also useful to notice that each subgroup contains a maximal coisotropic subgroup (its ``coisotropic interior''), and accordingly each homogeneous space has a naturally associated Poisson quotient.
                                                            \par
   As to the algebraic description, all properties concerning Poisson (or Lie bialgebra) structures on groups, Lie algebras, subgroups and homogeneous spaces have unique characterizations in terms of the algebraic codification one adopts for these geometrical objects.  Details change a bit according to whether one deals with global or formal geometry, but everything goes in parallel in either context.

\vskip5pt

   By (complex) ``quantum group'' of  {\sl formal type\/}  we mean any topological Hopf algebra  $ H_\hbar $  over the ring  $ \C[[\hbar]] $  whose semiclassical limit at  $ \, \hbar = 0 \, $   --- i.e.,  $ \, H_\hbar \big/ \hbar\,H_\hbar \, $  ---  is of the form  $ F\big[\!\big[G_{\!f}\big]\!\big] $  or  $ U(\gerg) $  for some formal group  $ G_{\!f} $  or Lie algebra  $ \gerg \, $.  Accordingly, one writes  $ \, H_\hbar := F_\hbar\big[\!\big[G_{\!f}\big]\!\big] \, $  or  $ \, H_\hbar := U_\hbar(\gerg) \, $,  calling the former a QFSHA and the latter a QUEA.  If such a quantization (of either type) exists, the formal group  $ G_{\!f} $  is Poisson and  $ \gerg $  is a Lie bialgebra; accordingly, a dual formal Poisson group  $ G_{\!f}^{\,*} $  and a dual Lie bialgebra  $ \gerg^* $  exist too.
                                                            \par
   In this context, as formal quantizations of subgroups or homogeneous spaces one typically considers suitable subobjects of either  $ F_\hbar\big[\!\big[G_{\!f}\big]\!\big] $  or  $ U_\hbar(\gerg) $  such that:  {\it (1)\/}  with respect to the containing formal Hopf algebra, they have the same relation as a in the ``classical'' setting   --- such as being a one-sided ideal, a subcoalgebra, etc.;  {\it (2)\/}  taking their specialization at  $ \, \hbar = 0 \, $  is the same as restricting to them the specialization of the containing algebra (this is typically mentioned as a ``flatness'' property).  This second requirement has a key consequence, i.e.~the semiclassical limit object is necessarily ``good'' w.r.~to the Poisson structure: namely, if we are quantizing a subgroup, then the latter is necessarily coisotropic, while if we are quantizing a homogeneous space then it is indeed a Poisson quotient.

\vskip5pt

  In the spirit of global geometry, by (complex) ``quantum group'' of  {\sl global type\/}  we mean any Hopf algebra  $ H_q $  over the ring  $ \Cqqm $  whose semiclassical limit at  $ \, q = 1 \, $   --- i.e.,  $ \, H_q \big/ (\,q\!-\!1)\,H_q \, $  ---  is of the form  $ F[G] $  or  $ U(\gerg) $  for some algebraic group  $ G $  or Lie algebra  $ \gerg \, $.  Then one writes  $ \, H_q := F_q[G] \, $  or  $ \, H_q := U_\hbar(\gerg) \, $,  calling the former a QFA and the latter a QUEA.  Again, if such a quantization (of either type) exists the group  $ G $  is Poisson and  $ \gerg $  is a Lie bialgebra, so that dual formal Poisson groups  $ G^* $  and a dual Lie bialgebra  $ \gerg^* $  exist too.
                                                            \par
   As to subgroups and homogeneous spaces, global quantizations can be defined via a sheer reformulation of the same notions in the formal context: we refer to such quantizations as  {\sl strict}.  In this paper, we introduce two more versions of quantizations, namely  {\sl proper\/}  and  {\sl weak\/}  ones, ordered by increasing generality, namely  $ \, \{\text{\sl strict\/}\} \subsetneq \{\text{\sl proper\/}\} \subsetneq \{\text{\sl weak\/}\} \, $.  This is achieved by suitably weakening the condition  {\it (2)\/}  above which characterizes a quantum subgroup or quantum homogeneous space.  Remarkably enough, one finds that now the existence of a  {\sl proper\/}  quantization is already enough to force a subgroup to be coisotropic, or a homogeneous space to be a Poisson quotient.

\vskip5pt

   The  {\sl Quantum Duality Principle\/}  (=QDP) was first developed by Drinfeld  (cf.~\cite{Dr}, \S 7) for formal quantum groups (see  \cite{Ga1}  for details).  It provides two functorial recipes, inverse to each other, acting as follows: one takes as input a QFSHA for  $ G_{\!f} $  and yields as output a QUEA for  $ \gerg^* \, $;  the other one as input a QUEA for  $ \gerg $  and yields as output a QFSHA for  $ G_{\!f}^{\,*} \, $.
                                                                       \par
   The  {\sl Global Quantum Duality Principle\/}  (=GQDP) is a version of the QDP tailored for global quantum groups  (see \cite{Ga2,Ga3}):  now one functorial recipe takes as input a QFA for  $ G $  and yields a QUEA for  $ \gerg^* \, $,  while the other takes a QUEA for  $ \gerg $  and provides a QFA for  $ G^* \, $.

\vskip5pt

   An appropriate version of the QDP for formal subgroups and formal homogeneous spaces was devised in  \cite{CiGa}.  Quite in short, the outcome there was an explicit recipe which taking as input a formal quantum subgroup, or a formal quantum homogeneous space, respectively, of  $ G_{\!f} $  provides as output a quantum formal homogeneous space, or a formal quantum subgroup, respectively, of  $ G_{\!f}^{\,*} \, $.  In short, these recipes come out as direct ``restriction'' (to formal quantum subgroups or formal quantum homogeneous spaces) of those in the QDP for formal quantum groups.  This four-fold construction is fully symmetric, in particular all duality or orthogonality relations possibly holding among different quantum objects are preserved.  Finally, Poisson duality is still involved, in that the semiclassical limit of the output quantum object is always the coisotropic dual of the semiclassical limit of the input quantum object.

\vskip7pt

   The main purpose of the present work is to provide a suitable version of the GQDP for global quantum subgroups and global quantum homogeneous spaces   --- extending the GQDP for global quantum groups ---   as much general as possible.  The inspiring idea, again, is to ``adapt'' (by restriction, in a sense) to these more general quantum objects the functorial recipes available from the GQDP for global quantum groups.  Remarkably enough, this approach is fully successful: indeed, it does work properly not only with  {\sl strict\/}  quantizations (which should sound natural) but also for  {\sl proper\/}  and for  {\sl weak\/}  ones.  Even more, the output objects always are global quantizations (of subgroups or homogeneous spaces) of  {\sl proper\/}  type   --- which gives an independent motivation to introduce the notion of proper quantization.
                                                                       \par
   Also in this setup, Poisson duality, in a generalized sense, shows up again as the link between the input and the output of the GQDP recipes: namely, the semiclassical limit of the output quantum object is always the coisotropic dual of the coisotropic interior of the semiclassical limit of the input quantum object.
                                                                       \par
   Besides the wider generality this GQDP applies to (in particular, involving also non-coisotropic subgroups, or homogeneous spaces which are not Poisson quotients), we pay a drawback in some lack of symmetry for the final result   --- compared to what one has in the formal quantization context.  Nevertheless, such a symmetry is almost entirely recovered if one restricts to dealing with  {\sl strict\/}  quantizations, or to dealing with ``double quantizations''   --- involving simultaneously a QFA and a QUEA in perfect (i.e.\ non-degenerate) pairing.

\vskip9pt

   At the end of the paper  (Section \ref{examples})  we present some applications of our GQDP: this is to show how it effectively works, and in particular that it does provide explicit examples of global quantum subgroups and global quantum homogeneous spaces.  Among these, we also provide an example of a quantization which is  {\sl proper\/}  but is  {\sl not strict}   --- which shows that the former notion is a non-trivial generalization of the latter.

\bigskip

\section{General Theory}

The main purpose of the present section is to collect some classical material about Poisson geometry for groups and homogeneous spaces.  Everything is standard, we just need to fix the main notions and notations we shall deal with.

\medskip

\subsection{Subgroups and homogeneous spaces}  \label{subgrps-homspaces}

Let $G$ be a complex affine algebraic group and let $\mathfrak g$ be its tangent Lie algebra.  Let us denote by $F[G]$ its algebra of regular functions and by $U(\gerg)$  its universal enveloping algebra. Both such algebras are Hopf algebras, and there exists a natural pairing of Hopf algebras between them, given by evaluation of differential operators onto functions. This pairing is perfect if and only if $G$ is connected, which we will always assume in what follows.
                                                     \par
   A \emph{real form} of either $G$ or $\mathfrak g$ is given once a Hopf $*$--algebra structure is fixed on either $F[G]$ or $U(\mathfrak g) $   --- and in case one take such a structure on both sides, the two of them must be dual to each other. Thus by \emph{real algebraic group} we will always mean a complex algebraic group endowed with a suitable $*$--structure.
                                                     \par
   A subgroup $K$ of $G$ will always be considered as Zariski--closed and algebraic. For any such subgroup, the quotient $G\big/K$ is an algebraic left homogeneous $G$--space, which is quasi-projective as an algebraic variety. Given  an algebraic left homogeneous $G$--space $M$ and choosing $m\in M$, the stabilizer subgroup $K_m$ will be a closed algebraic subgroup of $G$ such that  $ \, G\big/K_m \simeq M \, $;  changing point will change the stabilizer within a single conjugacy class.
                                                     \par
   We shall describe the subgroup $K\,$, or the homogeneous space $G\big/K$, through either an algebraic subset of $F[G]$  --- to which we will refer as a \emph{global} coding ---   or an algebraic subset of $U(\gerg)$   --- to which we will refer as a \emph{local} coding.  The complete picture is the following:
\begin{description}
 \item[--- subgroup] $K\;$:
  \begin{itemize}
      \item[{\it (local)}] \, letting $\gerk=\mathrm{Lie}(K)$ we can consider its enveloping algebra $U(\gerk)$ which is a Hopf subalgebra of $U(\gerg)\,$;  we then set $\;\gerC\equiv\gerC(K):=U(\gerk)\;$;
      \item[{\it (global)}] \, functions which are $0$ on $K$ form a Hopf ideal  $ \, \calI \equiv \calI(K) \, $ inside $F[G]\,$, such that  $ \, F[K] \simeq F[G]\big/\calI \; $.
    \end{itemize}
 \item[--- homogeneous space] $G\big/K\;$:
   \begin{itemize}
           \item[{\it (local)}] \, let  $ \; \gerI \equiv \gerI(K) = U(\gerg) \cdot \gerk \; $:  this is a left ideal and two-sided coideal in $U(\gerg)\,$, and  $ \, U(\gerg) \big/ \gerI \, $  is the set of left--invariant differential operators on  $ G \big/ K \, $.
           \item[{\it (global)}] \, regular functions on the homogeneous space $G\big/K$ may be identified with $K$--invariant regular functions on $G\,$. We will let  $ \, \calC = \calC(K) = {F[G]}^K \, $;  this is a subalgebra and left coideal in $F[G]\,$.
                                        \par
         \quad  {\sl  $ \underline{\text{Warning}} \, $:\/}  this needs clarification!  The point is: can one recover the homogeneous space  $ G\big/K $  from  $ \, \calC(K) = {F[G]}^K \, $? The answer depends on geometric properties of  $ G\big/K $  itself   --- or (equivalently) of  $ K $  --- which we explain later on.
   \end{itemize}
\end{description}

\vskip7pt

   For any Hopf algebra $\calH$ we introduce the following notations: $\le^1$ will stand for ``unital subalgebra'', $\trianglelefteq$ for ``two-sided ideal'', $\trianglelefteq_l$ for ``left ideal'' and similarly $\dot{\le}$ will stand for ``subcoalgebra'',  $\coideal$  for ``two-sided coideal'' and $\coideal_\ell$ for ``left coideal''. When the same symbols will be decorated by a subindex referring to a specific algebraic structure their meaning should be modified accordingly,  e.g.~$\trianglelefteq_\calH$ \hbox{will stand for ``Hopf ideal'' and $\le_\calH$ for ``Hopf subalgebra''.}
                                                              \par
   With such notations, with any subgroup $K$ of $G$ there is associated one of the following algebraic objects:
\begin{equation}\label{subgroup}
  (a) \;\; \calI\trianglelefteq_{\cal H}F[G]  \qquad\;\;
  (b) \;\; \calC \le^1\coideal_\ell\, F[G]  \qquad\;\;
  (c) \;\; \gerI\trianglelefteq_l\,\coideal\, U(\gerg)  \qquad\;\;
  (d) \;\; \gerC\le_{\cal H}U(\gerg)  \;\;\;
\end{equation}

\noindent In the real case, one has to consider, together with (\ref{subgroup}),  additional requirements involving the $*$ structure and the antipode $S\,$,  namely
\begin{equation}\label{realsubgroup}
  (a) \;\;\; \calI^*=\calI  \qquad\;\;
  (b) \;\;\; S(\calC)^*=\calC  \qquad\;\;
  (c) \;\;\; S(\gerI)^*=\gerI  \qquad\;\;
  (d) \;\;\; \gerC^*=\gerC
\end{equation}

In the connected case algebraic objects of type $\calI\,$, $\gerI$ and $\gerC$ in (\ref{subgroup}) are enough to reconstruct either $K$ or $G\big/K\,$:
  $$  K \, = \, \mathrm{Spec}\big( F[G]/\calI\big) \, = \, \mathrm{exp}\big(\mathrm{Prim}(\gerC)\big)
\, = \, \mathrm{exp}\big(\mathrm{Prim}(\gerI)\big)  $$
where $\mathrm{Prim}(X)$ denotes the set of primitive elements of a bialgebra $X\,$.
                                           \par
   In contrast,  $ \, \calC(K) = {F[G]}^K \, $  might be not enough to reconstruct $K\,$, due to lack of enough global algebraic functions; this happens, for example, when $G\big/K$ is projective and therefore $\,\calC(K)=\mathbb{C}\,$.  Any group  $ K $  which can be reconstructed from its associated  $\calC$  is called  {\it observable\/}:  we shall now make this notion more precise.
 \vskip5pt
   Let us call $\tau$ the map that to any subgroup $K$ associates the algebra of invariant functions $F[G]^K$ and let us call $\sigma$ the map that to any subalgebra $A$ of $F[G]$ associates its stabilizer  $ \, \sigma(A)=\left\{\, g \in G\,\big|\, g\cdot f=f\,\,\, \forall \, f \in A \,\right\} \, $. These two maps are obviously inclusion--reversing. Furthermore they establish what is also known as a \emph{simple Galois correspondence\/}:  namely, for any subgroup $K$ and any subalgebra $A$  one has
  $$  (\sigma\circ\tau)(K) \supseteq K \; ,  \qquad (\tau\circ\sigma)(A) \supseteq A  $$
so that  $\left(\tau\circ\sigma\circ\tau\right)(K)=\tau(K)$, $\left(\sigma\circ\tau\circ\sigma\right)(A)=\sigma(A)$. A subgroup $K$ of $G$ such that $\left(\sigma\circ\tau\right)(K)=K$ is said to be \emph{observable\/}:  this means exactly that such a subgroup can be fully recovered from its algebra of invariant functions $\tau(K)$. If $K$ is any subgroup, then $\,\widehat{K}:=\left(\sigma\circ\tau\right)(K)\,$  is the smallest observable subgroup containing $K$; we will call it the \emph{observable hull } of $K$. Remark then that $\,\calC(K)=\calC\big(\widehat{K}\big)\,$.
                                          \par
   The following fact (together with many properties of observable subgroups), which gives a characterization of observable subgroups in purely geometrical terms, may be found in \cite{Gro}:

\vskip11pt

\noindent{\bf Fact}:  {\sl a subgroup $K$ of $G$ is observable if and only if  $ \, G\big/K \, $  is quasi--affine}.

\vskip11pt

   Let us now clarify how to pass from algebraic objects directly associated with subgroups to those corresponding to homogeneous spaces. Let $H$ be a Hopf algebra, with counit  $ \, \varepsilon \, $  and coproduct  $ \Delta\, $.  For any submodule $M\subseteq H$ define
\begin{equation}
\label{coinvariants}
M^+:=M\cap \Ker(\varepsilon) \;\; ,  \qquad
H^{{\mathrm co} M}:=\left\{\, y\in H\,\big|\, \left(\Delta(y)-y\otimes 1\right)\in H\otimes M \,\right\}
\end{equation}

Let $C$ be a (unital) subalgebra and left coideal of $H$ and define $\Psi(C)=H\cdot C^+$. Then $\Psi(C)$ is a left ideal and two-sided coideal in $H$. Conversely, let $I$ be a left ideal and two-sided coideal in $H$ and define $\, \Phi(I) := H^{{\mathrm co}I} \, $. Then $\Phi(I)$ is a unital subalgebra and left coideal in $H$. Also, this pair of maps $(\Phi,\Psi)$ defines a simple Galois correspondence, that is to say
\begin{itemize}
\item[{\it (a)}] \qquad  $ \Psi $ and $\Phi$ are inclusion-preserving;
\item[{\it (b)}] \qquad  $\left(\Phi\circ\Psi\right)(C)\supseteq C \; $,  \quad  $\left(\Psi\circ\Phi\right)(I)\subseteq I \; $;
\item[{\it (c)}] \qquad  $\Phi\circ\Psi\circ\Phi=\Phi\;$, \quad $\Psi\circ\Phi\circ\Psi=\Psi \; $.
\end{itemize}
(where the third property follows from the previous ones; see \cite{Mo, Sch, Ta} for further details).

\vskip7pt

   Let now $K$ be a subgroup of $G$ and let $\calI$, $\cal C$, $\gerI$, $\gerC$ the corresponding algebraic objects as described in (\ref{subgroup}).
%
%
 We can thus establish the following relations among them:

\begin{description}
  \item[subgroup vs. homogeneous space:]  objects directly related to the subgroup (namely,  $\calI$ and $\gerC$) and objects directly related to the homogeneous space (namely,  $\calC$ and $\gerI$) are linked by  $ \Psi $  and  $ \Phi $  as follows:
\begin{equation}
\label{correK}
  \gerI \, = \, \Psi(\gerC) \; ,  \qquad  \gerC \, = \, \Phi(\gerI) \; ,  \qquad
\calI \, \supseteq \, \Psi(\calC) \; ,  \qquad  \calC \, = \, \Phi(\calI)
\end{equation}

In particular, $K$ is observable if and only if $\,\calI=\Psi(\calC)\,$;  on the other hand, we have in general $\,\Psi(\calC(K))=\calI(\widehat{K})\;$.
  \item[orthogonality] with respect to the natural pairing between $F[G]$ and $U(\gerg)\,$:  this is expressed by the relations
\begin{equation}
\label{orthoK}
\calI \, = \, \gerC^\bot \; ,  \qquad  \gerC \, = \, \calI^\bot \; ,  \qquad  \calC \, = \, \gerI^\bot \; ,  \qquad  \gerI \, \subseteq \, \calC^\bot
\end{equation}

In particular, $K$ is observable if and only if $\,\gerI=\calC^\bot\,$;  on the other hand, we have in general $\,\calC(K)^\bot=\gerI\big(\widehat{K}\big)\;$.
                                                        \par
   Let us also remark that orthogonality intertwines the local and global description.
\end{description}

\medskip

\noindent
 {\bf The ``formal'' vs.~``global'' geometry approach.}  In the present approach we are dealing with geometrical objects   --- groups, subgroups and homogeneous spaces ---   which we describe via suitably chosen algebraic objects.  When doing that, universal enveloping algebras or subsets of them only provide a  {\sl local\/}  description   --- around a distinguished point: the unit element in a (sub)group, or its image in a coset (homogeneous) space.  Instead, function algebras yield a  {\sl global\/}  description, i.e.~they do carry information on the whole geometrical object; for this reason, we refer to the present approach as the ``global'' one.
                                                        \par
   The ``formal geometry'' approach instead only aims to describe a group by a topological Hopf algebra, which can be realized as an algebra of formal power series; in short, this is summarized by saying that we are dealing with a ``formal group''.  Subgroups and homogeneous spaces then are described by suitable subsets in such a formal series algebra (or in the universal enveloping algebra, as above): this again yields only a local description   --- in a formal neighborhood of a distinguished point ---   rather than a global one.
 \vskip4pt
   Now, the analysis above shows that an asymmetry occurs when we adopt the global approach.  Indeed, we might have problems when describing a homogeneous space by means of (a suitably chosen subalgebra of invariant) functions: technically speaking, this shows up as the occurrence of  {\sl inclusions}   --- rather than identities! ---   in formulas  \ref{correK}  and  \ref{orthoK}.  This is a specific, unavoidable feature of the problem, due to the fact that homogeneous spaces (for a given group) do not necessarily share the same geometrical nature   --- beyond being all quasi-projective ---   in particular they are not necessarily quasi-affine.
                                                        \par
   The case of those homogeneous spaces which are  {\sl projective\/}  is treated in  \cite{CFG},  where their quantizations are studied; in particular, there a suitable method to solve the problematic  ``$ \calC $--side''  of the QDP in that case is worked out, still in terms of ``global geometry'' but with a different tool (semi-invariant functions, rather than invariant ones).
 \vskip4pt
   In contrast, in the formal geometry approach such a lack of symmetry does not occur: in other words, it happens that  {\it every  {\sl formal}  (closed) subgroup is observable},  or  {\it every  {\sl formal}  homogeneous space is quasi-affine}.  This means that there is no need of worrying about observability, and the full picture   --- for describing a subgroup or homogeneous space, in four different ways ---   is entirely symmetric.  This was the point of view adopted in  \cite{CiGa},  where this complete symmetry of the formal approach is exploited to its full extent.

\medskip

\subsection{Poisson subgroups and Poisson quotients}  \label{Pois-sgrs_Pois-quots}

Let us now assume that $G$ is endowed with a complex Poisson group structure  corresponding to a Lie bialgebra structure on $\mathfrak g\,$,  whose Lie cobracket is denoted  $ \, \delta : \gerg \longrightarrow \gerg \wedge \gerg \, $. At the Hopf algebra level this means that $F[G]$ is a Poisson--Hopf algebra and $U(\mathfrak g)$ a co-Poisson Hopf algebra, in such a way that the duality pairing is compatible with these additional  structures (see \cite{CP} for basic definitions). Let us recall that the linear dual  $ \gerg^*$  inherits a Lie algebra structure;
 on the other hand, it has a natural Lie coalgebra structure, whose cobracket  $ \, \delta : \gerg^* \longrightarrow \gerg^* \wedge \gerg^* \, $  is the dual map to the Lie bracket of  $ \gerg \, $.  Altogether, this makes  $ \gerg^* $  into a Lie bialgebra, which said to be  {\sl dual\/}  to  $ \gerg \, $.
 Therefore, there exist Poisson groups whose tangent Lie bialgebra is $\mathfrak g^*\,$; we will assume one such connected group is fixed, we will denote it with $G^*$ and call it the dual Poisson group of $G$. In the real case the involution in $F[G]$ is a Poisson algebra antimorphism and the one in $U(\gerg)$ is a co-Poisson algebra antimorphism.

A closed subgroup $K$ of $G$ is called \emph{coisotropic} if its defining ideal $\calI(K)$ is a Poisson subalgebra, while it is called a \emph{Poisson subgroup} if   $\calI(K)$ is a Poisson ideal, the latter condition being equivalent to $K\hookrightarrow G$ being a Poisson map. Connected coisotropic subgroups can be characterized, at an infinitesimal level, by one of the following conditions on $\,\gerk\subseteq\gerg\,$:
  $$  \displaylines{
   \qquad   \text{\it (C-i)}  \qquad  \delta(\gerk) \, \subseteq \, \gerk \wedge \gerg \; ,  \quad  \text{that is  $ \gerk $  is a Lie coideal in  $ \gerg \; $},   \hfill  \cr
   \qquad   \text{\it (C-ii)}  \qquad  \text{$ \gerk^\bot $  is a Lie subalgebra of $\gerg^*\;$,}   \hfill  }  $$
while analogous characterizations of Poisson subgroups correspond to $\gerk$ being a Lie subcoalgebra or $\gerk^\bot$ being a Lie ideal.

The most important features of coisotropic subgroups, in this setting, is the fact that $G\big/K$ naturally inherits a Poisson structure from that of $G$. Actually, a Poisson manifold $(M\,,\omega_M)$ is called a {\it Poisson homogeneous $G$--space\/}  if there exists a smooth, homogeneous $G$--action  $ \, \phi : G \times M \longrightarrow M \, $  which is a Poisson map (w.r.~to the product Poisson structure on the domain).  In particular, we will say that $(M,\omega_M)$ is a  {\it Poisson quotient\/}  if it verifies one of the following equivalent conditions (cf.~\cite{Zak}):
  $$  \displaylines{
   \;\quad   \text{\it (P-i)}  \quad  \text{ there exists $x_0\in M$ whose stabilizer $G_{x_0}$ is coisotropic in $G$ };   \hfill  \cr
   \;\quad   \text{\it (P-ii)}  \quad  \text{ there exists $x_0\in M$ such that $\; \phi_{x_0}:G\to M \, $, $\, \phi(x_0,g)=\phi(x,g) \, $, }   \hfill  \cr
   \;\quad   \phantom{\text{\it (P-ii)}}  \quad  \text{ is a Poisson map };   \hfill  \cr
   \;\quad   \text{\it (P-iii)}  \quad  \text{ there exists $x_0\in M$ such that $\,\omega_M(x_0)=0\,$ }.   \hfill  }  $$
It is important to remark here that inside the same conjugacy class of subgroups of $G$ there may be subgroups which are Poisson, coisotropic, or non coisotropic.  Therefore, on the same homogeneous space there may exist many Poisson homogeneous structures, some of which make it into a Poisson quotient while some others do not.
                                                         \par
   For a fixed connected subgroup $K$ of a Poisson group $G$, with Lie algebra $\gerk$, one can consider the following descriptions in terms of the Poisson Hopf algebra $F[G]$ or of the co-Poisson Hopf algebra $U(\gerg)\,$:
\begin{align}
  \calI\le_{\calP}F[G] \; ,  &  \qquad  \calC\le_{\calP}F[G]\\
  \gerI \, \coideal_{\cal P} \; U(\gerg) \; ,  &  \qquad  \gerC \, \coideal_{\calP} \; U(\gerg)
\end{align}
where on first line we have global conditions and on second line local ones. Conversely each one of these conditions imply coisotropy of $G$ with the exception of the condition on $\calC$, which implies only that the observable hull $\widehat{K}$ is coisotropic. Therefore a connected, observable, coisotropic subgroup of $G$ is identified by one of the following algebraic objects:
\begin{align}
  \calI\trianglelefteq_{\cal H}\,\le_{\calP}F[G] \; ,  &  \qquad \calC\le^1 \coideal_\ell \le_{\calP} F[G]\\
  \gerI\trianglelefteq_l \coideal \, \coideal_{\cal P}\; U(\gerg) \; ,  &  \qquad  \gerC  \le_{\calH} \coideal_{\cal P} \; U(\gerg)
\end{align}

\noindent
 (still with the usual, overall restriction on the use of $\calC$, which in general only describes the observable hull $\widehat{K}\,$).
                                                         \par
   Thanks to self-duality in the notion of Lie bialgebra, with any Poisson group there is associated a natural \emph{Poisson dual}, which is fundamental in the QDP; note that a priori many such dual groups are available, but when dealing with the QDP such an (apparent) ambiguity will be solved.  As we aim to extend the QDP to coisotropic subgroups, we need to introduce a suitable notion of (Poisson) duality for coisotropic subgroups as well.

\smallskip

\begin{definition}\label{coisotropic duality}
Let $G$ be a Poisson group and $\,G^*$ a fixed Poisson dual.
\begin{enumerate}
\item If $\,K$ is coisotropic in $G$ we call  {\sl complementary dual}  of  $ K $  the unique connected subgroup $K^\bot$ in $G^*$ such that $\,\mathrm{Lie}(K^\bot)=\gerk^\bot\,$.
\item If $M$ is a Poisson quotient and $\,M\simeq G\big/K_M\,$ we call  {\sl complementary dual}  of $M$ the Poisson $G^*$--quotient $\,M^\bot:=G^*\!\big/K_M^\bot\,$.
\item For any subgroup $H$ of $\,G$ we call  {\sl coisotropic interior}  of $H$ the unique maximal, closed, connected, coisotropic subgroup  $\Hh$  of $G$ contained in $H\,$.
\end{enumerate}
\end{definition}

 \eject

\noindent{\bf Remarks:}
\begin{enumerate}
\item The complementary dual of a coisotropic subgroup is, trivially, a coisotropic subgroup whose complementary dual is the connected component of the one we started with:. Similarly, the complementary dual of a Poisson quotient is a Poisson quotient, and if we start with a Poisson quotient whose coisotropy subgroup (w.r.~to any point) is connected then taking twice the complementary dual brings back to the original Poisson quotient.
\item The coisotropic interior may be characterized, at an algebraic level, as the unique closed subgroup whose Lie algebra is maximal between Lie subalgebras of $\gerh$ which are Lie coideals in $\gerg\,$.
\end{enumerate}

\medskip

\begin{proposition}\label{bot-pro}
Let $K$ be any subgroup of $G$ and let $K^{\langle\bot\rangle}:=\big\langle \mathrm{exp}(\gerk^\bot)\big\rangle$  be the closed, connected, subgroup of $G^*$ generated by $\mathrm{exp}(\gerk^\bot)\,$. Then:
\begin{itemize}
\item[(a)] \, the Lie algebra $\gerk^{\langle\bot\rangle}$ of $K^{\langle\bot\rangle}$ is the Lie subalgebra of $\gerg^*$ generated by $\gerk^\bot$;
\item[(b)] \; $\gerk^{\langle\bot\rangle}$ is a Lie coideal of $\gerg^*$, hence $K^{\langle\bot\rangle}$ is a coisotropic subgroup of $\,G^*$;
\item[(c)] $\;K^{\langle\bot\rangle}=(\Kk\,)^\bot\,$; in particular if $K$is coisotropic then $\,K^{\langle\bot\rangle}=K^\bot\,$;
\item[(d)]  $ \; (K^{\langle\bot\rangle})^{\langle\bot\rangle} = \, \Kk \, $  and  $ K $  is coisotropic if and only if  $ \, (K^{\langle\bot\rangle})^{\langle\bot\rangle} = K \, $.
\end{itemize}
\end{proposition}

\begin{proof}
 Part {\it (a)} is trivial. As for {\it (b)}, since $\gerk=(\gerk^\bot)^\bot$ is a Lie subalgebra of $\gerg\,$, we have that $\gerk^\bot$ is a Lie coideal in $\gerg^*\,$: therefore, due to the identity
  $$  \delta\big([x,y]\big)  \, = \,  {\textstyle \sum\limits_{[y]}} \big( [x,y_{[1]}] \otimes y_{[2]} + y_{[1]} \otimes [x,y_{[2]}] \big) + {\textstyle \sum\limits_{[x]}} \big( [x_{[1]},y] \otimes x_{[2]} + x_{[1]} \otimes [x_{[2]},y] \big)  $$
(where $\delta(z)=\sum_{[z]}z_{[1]}\otimes z_{[2]}$ for $z\in\gerg^*$), the Lie subalgebra $\langle\gerk^\bot\rangle$ of $\gerg^*$ generated by $\gerk^\bot$ is a Lie coideal too.
It follows then by claim {\it (a)} that $K^{\langle\bot\rangle}$ is coisotropic. Thus {\it (b)} is proved.

As for part {\it (c)} we have
  $$  \left(\gerk^{\langle\bot\rangle}\right)^\bot  \; = \;  \langle\gerk^\bot\rangle^\bot  \; = \;
\Bigg(\, {\textstyle \bigcap\limits_{\substack{\gerh\le_{\cal L}\,\gerg^*\\ \gerh\supseteq \gerk^\bot}}} \gerh \Bigg)^{\!\!\bot}
\; = \;  {\textstyle \sum\limits_{\substack{\gerh\le_{\cal L}\,\gerg^*\\ \gerh\supseteq \gerk^\bot}}} \gerh  \; = \;
{\textstyle \sum\limits_{\substack{\gerf\coideal_{\cal L}\gerg \\ \gerf\supseteq \gerk}}} \gerf
\; = \;  \stackrel{\circ}{\gerk}  $$
(with $\le_{\cal L}$ meaning ``Lie subalgebra'' and $\coideal_{\cal L}$ meaning ``Lie coideal'') where
$\stackrel{\circ}{\gerk}$ is exactly the maximal Lie subalgebra and Lie coideal of $\gerg$ contained in $\gerk\,$. To be precise, this last statement follows from the above formula for  $ \delta\big([x,y]\big) \, $,  since that formula implies that the Lie subalgebra generated by a family of Lie coideals is still a Lie coideal.

Now $\;\stackrel{\circ}{\gerk}\,=\mathrm{Lie}(\Kk)\,$, so $\,\mathrm{Lie}(K^{\langle\bot\rangle})=
\gerk^{\langle\bot\rangle}=\left(\left(\gerk^{\langle\bot\rangle}\right)^\bot\right)^\bot\!=
\big(\stackrel{\circ}{\gerk}\big)^\bot\!=\mathrm{Lie}(\Kk)^\bot\,$ implies $\,K^{\langle\bot\rangle}=
(\stackrel{\circ}{\gerk})^\bot\,$ as we wished to prove.
If, in addition, $K$ is coisotropic then, obviously, $\,K^{\langle\bot\rangle}=K\,$. All other statements follow easily.
\end{proof}

\bigskip

\section{Strict, proper, weak quantizations}  \label{quantizations}

The purpose of this section is to fix some terminology concerning the meaning of the word ``quantization'' and to describe some possible ways of quantizing a (closed) subgroup, or a homogeneous space.  We set the algebraic machinery needed for talking of ``quantization'' and ``specialization'':
these notions must be carefully specified before approaching the
construction of Drinfeld's functors.
%

\smallskip

   Let  $ q $  be an indeterminate,  $ \, \Cqqm \, $  the ring of
complex-valued Laurent polynomials in  $ q \, $,  and  $ \, \C(q)
\, $  the field of complex-valued rational functions in  $ q \, $.
   Denote by  $ \HA $  the category of all Hopf algebras over
$ \Cqqm $  which are torsion-free as  $ \Cqqm $--modules.
                                             \par
   Given  a Hopf algebra $ H $  over the field  $\C(q) \, $,  a subset  $ \overline{H} \subseteq H$
 is called  {\it a  $ \Cqqm $--integral  form}  (or
simply  {\it a\/  $ \Cqqm $--form})   if it is
a  $ \Cqqm $--Hopf  subalgebra of  $ \, H \, $  and  $ \; H_F
:= \C(q) \otimes_{\C[q,q^{-1}]} \overline{H} = H \, $.  Then  $ \, \overline{H} \, $  is torsion-free as a
$ \Cqqm $--module,  hence  $ \, \overline{H} \in \HA \, $.
                                               \par
   For any  $ \Cqqm $--module  $ M $,  we set  \hbox{$ \, M_1 := M \big/
(q-1) M = \C \otimes_{\C[q,q^{-1}]} M \, $:}  this is a  $ \C $--module
(via  $ \, \Cqqm \rightarrow \Cqqm \Big/ (q-1) = \C \, $),  called
{\sl specialization of  $ M $  at  $ \, q = 1 \, $}.
                                                     \par
   Given two  $ \C(q) $--modules  $ A $  and  $ B $  and a
$ \C(q) $--bilinear  pairing  \hbox{$ \, A \times B \longrightarrow F\, $,}  for any  $ \Cqqm $--submodule  $ \, A_\times \subseteq A \, $
we set:
%
\begin{equation}  \label{dual}
    {A_\times\phantom{\big)}}^{\hskip-8pt\bullet}  \; := \;
\Big\{\, b \in B \,\Big\vert\; \big\langle A_\times, \, b \big\rangle
\subseteq \Cqqm \Big\}
\end{equation}

\noindent
 In such a setting, we call  $ {A_\times\phantom{\big)}}^{\hskip-8pt
\bullet} $  {\sl the  $ \Cqqm $--dual  of}  $ A_\times \, $.

\medskip

   We will call \emph{quantized universal enveloping algebra} (or, in short, QUEA) any $U_q\in\HA$ such that $U_1:=(U_q)_1$ is isomorphic to $U(\gerg)$ for some Lie algebra $\gerg\,$, and we will call \emph{quantized function algebra} (or, in short, QFA) any $\,F_q\in\HA\,$ such that  $ \; F_1:={(F_q)}_1 \; $  is isomorphic to $F[G]$ for some connected algebraic group $G$  and, in addition, the following technical condition holds:
  $$  {\textstyle \bigcap\limits_{n \geq 0}}{(q-1)}^n F_q  \,\; = \;\,  {\textstyle \bigcap\limits_{n \geq 0}} {\big( (q-1) F_q \, + \, \Ker(\epsilon_{{}_{F_q}}) \big)}^n  $$
   \indent   We will add the specification that such quantum algebras are \emph{real} whenever the starting object is a $*$--Hopf algebra. As a matter of notation, we write
  $$  \U_q  \; := \;  \C(q) \otimes_{\C[q,q^{-1}]} U_q  \quad ,  \qquad  \F_q  \; := \;  \C(q)\otimes_{\C[q,q^{-1}]} F_q \quad  .  $$

When $U_q$ is a (real) QUEA, its specialization $U_1$ is a (real) co--Poisson Hopf algebra so that $\gerg$ is in fact a (real) Lie bialgebra. Similarly, for any (real) QFA $F_q$ the specialization $F_1$ is a (real) Poisson-Hopf algebra and therefore $G$ is a (real) Poisson group (see \cite{CP} for details).
                                                                \par
   On occasions it is useful to consider simultaneous quantizations of both the universal enveloping algebra and the function algebra, or, in a larger generality, of a pair of dual Hopf algebra. Let  $ \, H \, , K \in \HA \, $  and assume that there exists a pairing of Hopf algebras  $ \, \langle\,\ ,\ \rangle : H \times K \longrightarrow \Cqqm \, $.  If the pairing is such that
\begin{itemize}
  \item[{\it (a)}]   $ \; H = K^\bullet \, $,  $ \, K = H^\bullet \, $
(notation of (\ref{dual}))  w.r.t.~the pairing  $ \, \mathbb{H} \times \mathbb{K}
\rightarrow \Cq \, $,  for  $ \, \mathbb{H} := \C(q) \otimes_{\C[q,q^{-1}]} H \, $,  $ \, \mathbb{K} := \C(q) \otimes_{\C[q,q^{-1}]} K \, $,  induced from  $ \, H \times K \rightarrow \Cq \, $
  \item[{\it (b)}]  \; the Hopf pairing  $ H_1 \times K_1 \rightarrow
\C$  given by specialization at  $ q = 1$  is perfect (i.e.\ non-degenerate)
\end{itemize}
then we will say that  {\it $H$ and $K$ are dual to each other}.  Note that all these assumptions imply that the initial pairing between  $ H $  and  $ K $  is perfect.  When  $ \, H = U_q(\mathfrak g) \, $  is a QUEA and  $ \, K = F_q[G] \, $  is a QFA, if the specialized pairing at $1$ is the natural pairing between $U(\mathfrak g)$ and $F[G]$ we will say that  {\it the pair $(U_q(\mathfrak g)\,,F_q[G])$ is a double quantization of $\,(G,\mathfrak g)\,$}.

\bigskip

   Let us now move to the case in which $G$ is a Poisson group and $K$ a subgroup. We want to define a reasonable notion of ``quantization'' of $K$ and of the corresponding homogeneous space $G\big/K\,$. There is a standard way to implement this, which actually implies   --- cf.~Lemma \ref{strictisproper}  and  Proposition \ref{coisotropic creed}  later on ---   the additional constraint that $K$ be  {\sl coisotropic}.

\medskip

\begin{definition}  \label{def_strict-quant}
 Let  $ F_q[G] $  and  $ U_q(\gerg) $  be a QFA and a QUEA for  $ G $  and  $ \, \gerg $ and let
  $$  \displaylines{
   \pi_{F_q} \colon \, F_q[G] \relbar\joinrel\twoheadrightarrow F_q[G] \Big/ \! (q\!-\!1) F_q[G] \, \cong \, F[G]  \cr
   \pi_{U_q} \colon U_q(\gerg) \relbar\joinrel\twoheadrightarrow U_q(\gerg)\! \Big/ \! (q-\!1) \, U_q(\gerg) \cong U(\gerg)  }  $$
be the specialization maps.  Let  $ \calI $,  $ \calC $,  $ \gerI $  and
$ \gerC $  be the algebraic objects associated with the subgroup  $ K $
of  $ G $  (see \ref{subgroup}).  We call  {\sl ``strict quantization''\/}
(and sometimes we shall drop the adjective  {\sl ``strict''\/}) of each
of them any object  $ \calI_q \, $,  $ \calC_q \, $,  $ \gerI_q $  or
$ \gerC_q $  respectively, such that
\begin{equation} \label{strictquant}
  \begin{array}{ccc}
     {\it (a)} \qquad  \calI_q \ideal_\ell \coideal \; F_q[G] \;\; ,
&  \quad  \pi_{F_q}(\calI_q) \, = \, \calI \;\; ,
&  \quad  \pi_{F_q}(\calI_q) \, \cong \, \calI_q \big/ (q\!-\!1) \, \calI_q  \\
     {\it (b)} \qquad    \calC_q \leq^1 \coideal_\ell \, F_q[G] \;\; ,
&  \quad  \pi_{F_q}(\calC_q) \, = \, \calC \;\; ,
&  \quad  \pi_{F_q}(\calC_q) \, \cong \, \calC_q \big/ (q\!-\!1) \, \calC_q  \\
     {\it (c)} \qquad   \gerI_q \ideal_\ell \coideal \;\; U_q(\gerg) \;\; ,
&  \quad  \pi_{U_q}(\gerI_q) \, = \, \gerI \;\; ,
&  \quad  \pi_{U_q}(\gerI_q) \, \cong \, \gerI_q \big/ (q\!-\!1) \, \gerI_q  \\
     {\it (d)} \qquad    \gerC_q \leq^1 \coideal_\ell \, U_q(\gerg) \;\; ,
&  \quad  \pi_{U_q}(\gerC_q) \, = \, \gerC \;\; ,
&  \quad  \pi_{U_q}(\gerC_q) \, \cong \, \gerC_q \big/ (q\!-\!1) \, \gerC_q
  \end{array}
\end{equation}
\end{definition}

\medskip

   In order to explain this definition let us start by considering the first two conditions in each line of (\ref{strictquant}).
\begin{itemize}
\item[a)] A left ideal and two-sided coideal in a QFA quantizes the Hopf ideal of functions which are zero on a (closed) \emph{subgroup};
\item[b)] a left coideal subalgebra in a QFA quantizes the algebra of invariant functions on a \emph{homogeneous space};
\item[c)]  a left ideal and two-sided coideal in a QUEA quantizes the  infinitesimal algebra on a \emph{homogeneous space};
\item[d)] a  left coideal subalgebra in a QUEA quantizes the universal enveloping subalgebra of a \emph{subgroup}.
\end{itemize}

Once again, we must stress the fact that $\calC_q\,$, as was explained in Proposition \ref{correK}, has to be seen as a quantization of the observable hull $\widehat{K}$ rather than of $K$ itself.

\bigskip

   Let us now be more precise about the last condition in the previous definition. By asking  $ \; \calI_q \big/ (q-1) \calI_q \cong \pi_{F_q}(\calI_q) = \calI \; $  we mean the following: the specialization map sends $\calI_q$ inside $F[G]$. This map factors through  $ \, \calI_q \big/ (q-1) \calI_q \, $;  in addition, we require that the induced map $ \, \calI_q \big/ (q-1) \calI_q \longrightarrow F[G] \, $  be a bijection on  $ \cal I \, $.  Of course this bijection will respect the whole Hopf structure, since $\pi_{F_q}$ does. Now, since
  $$  \pi_{F_q}(\calI_q)  \, = \,  \calI_q \Big/ \big( \calI_q \cap (q-1)F_q[G] \big)  $$
this property may be equivalently rephrased by saying that  $ \; \calI_q \cap(q\!-\!1)\,F_q[G] = (q\!-\!1) \, \calI_q \; $  as well.  The previous discussions may be repeated unaltered for all four algebraic objects under consideration.  {\sl An equivalent definition of strict quantizations is therefore the following:}
 \begin{equation} \label{strictquant2}
\begin{array}{ccc}
   \hskip-11pt   (a) \qquad  \calI_q \ideal_\ell \coideal \; F_q[G] \;\; ,
&  \pi_{F_q}(\calI_q) \, = \, \calI \;\; ,
&  \calI_q\cap(q-1)F_q[G] \, = \, (q-1) \, \calI_q  \hskip5pt  \\
   \hskip-11pt   (b) \qquad  \calC_q \leq^1 \coideal_\ell \, F_q[G] \;\; ,
&  \pi_{F_q}(\calC_q) \, = \, \calC \;\; ,
&  \calC_q\cap(q-1)F_q[G] \, = \, (q-1) \, \calC_q  \hskip5pt  \\
   \hskip-11pt   (c) \qquad  \gerI_q \ideal_\ell \coideal \;\; U_q(\gerg) \;\; ,
&  \pi_{U_q}(\gerI_q) \, = \, \gerI \;\; ,
&  \gerI_q\cap(q-1)U_q(\gerg) \, = \, (q-1) \, \gerI_q  \hskip5pt  \\
   \hskip-11pt   (d) \qquad  \gerC_q \leq^1 \coideal_\ell \, U_q(\gerg) \;\;  ,
&  \pi_{U_q}(\gerC_q) \, = \, \gerC \;\; ,
&  \gerC_q\cap(q-1)U_q(\gerg) \, = \, (q-1) \, \gerC_q  \hskip5pt
\end{array}
\end{equation}

   The purpose of the last condition   --- which is often mentioned by saying that  $ \gerC_q $  is a  {\sl flat\/}  quantization (typically, in the literature on deformation quantization) ---   should be clear: indeed, removing it means losing any control on what is contained, in quantization, inside the kernel of the specialization map.

\smallskip

   Although the just mentioned notion of quantization appears to be, in many respect, the ``correct'' one   --- and indeed is typically the one considered in literature ---   another notion of quantization naturally appears when one has to deal with quantum duality principle.

\smallskip

\begin{definition}\label{properq}
 Let  $ F_q[G] $  and  $ U_q(\gerg) $  be a QFA and a QUEA for $G$  and  $\gerg$ and let
  $$  \displaylines{
   \pi_{F_q} \colon \, F_q[G] \relbar\joinrel\twoheadrightarrow F_q[G] \Big/ \! (q\!-\!1) F_q[G] \; \cong \; F[G]  \cr
   \pi_{U_q} \colon U_q(\gerg) \relbar\joinrel\twoheadrightarrow U_q(\gerg)\! \Big/ \! (q-\!1) \, U_q(\gerg) \; \cong \; U(\gerg)  }  $$
be the specialization maps. Let  $\nabla := \Delta - \Delta^{{\mathrm op}}$. Let  $ \calI $,  $ \calC $,  $ \gerI $  and $\gerC$  be the algebraic objects associated with the subgroup  $ K $
of  $G$  (see \ref{subgroup}).  We call  {\sl ``proper quantization''\/}
 of each of them any object  $ \calI_q \, $,  $ \calC_q \, $,  $ \gerI_q $  or
$ \gerC_q $  respectively, such that
 \begin{equation}\label{properquant}
   \begin{array}{ccc}
 (a) \qquad  \calI_q \ideal_\ell \coideal \; F_q[G] \;\; ,
&  \pi_{F_q}(\calI_q) \, = \, \calI \;\; ,
&  \big[ \calI_q \, , \calI_q \big] \, \subseteq \, (q-1) \, \calI_q  \hskip5pt  \\
   (b) \qquad  \calC_q \leq^1 \coideal_\ell \, F_q[G] \;\; ,
&  \pi_{F_q}(\calC_q) \, = \, \calC \;\; ,
&  \big[ \calC_q \, , \calC_q \big] \, \subseteq \, (q-1) \, \calC_q^{\phantom{|}} \hskip5pt  \\
   (c) \qquad  \gerI_q \ideal_\ell \coideal \;\; U_q(\gerg) \;\; ,
&  \pi_{U_q}(\gerI_q) \, = \, \gerI \;\; ,
&  \; \nabla(\,\gerI_q) \, \subseteq \, (q\!-\!1) \, U_q(\gerg) \wedge \gerI_q  \hskip5pt  \\
   (d) \qquad  \gerC_q \leq^1 \!\coideal_\ell \; U_q(\gerg) \;\; ,
&  \pi_{U_q}(\gerC_q) \, = \, \gerC \;\; ,
&  \; \nabla(\gerC_q) \, \subseteq \, (q\!-\!1) \, U_q(\gerg) \wedge \gerC_q  \hskip5pt
               \end{array}
  \end{equation}
\end{definition}

\vskip9pt

   The link between these two notions of quantization is the following:

\vskip15pt

\begin{lemma}\label{strictisproper}
Any strict quantization is a proper quantization.
\end{lemma}

\begin{proof}
 This is an easy consequence of definitions.  Indeed, let  $K$ be a subgroup of $G\,$.  If  $ \calI_q := \calI\big(\widehat{K}\big)$  is any strict quantization of $ \calI(K)$,  we have
  $$  \calI_q \cap (q-1) \, F_q \, = \, (q-1) \, \calI_q  $$
 by assumption, and moreover  $ \; \big[ F_q \, , F_q \big] \, \subseteq \, (q-1) \, F_q \; $.  Then
  $$  \big[\calI_q \, , \calI_q \big]  \; \subseteq \;  \calI_q \cap \big[ F_q \, , F_q \big]  \; \subseteq \; \calI_q \cap (q-1) \, F_q  \; = \;  (q-1) \, \calI_q  $$
thus  $ \big[\calI_q \, , \calI_q \big] \subseteq (q-1) \, \calI_q \, $,  \; i.e.~$ \calI_q $  is proper. A similar argument works for quantizations of type  $ \calC_q(K) \, $.  Also, if  $ \gerI_q(K) $  is any strict quantization of  $ \gerI(K) \, $,  then we have
$ \; \gerI_q \cap (q-1) \, U_q \, = \, (q-1) \, \gerI_q \; $  by assumption, and moreover
$\nabla(U_q) \, \subseteq \, (q-1) \, U_q^{\,\wedge 2} \, $.  Then
  $$  \nabla(\,\gerI_q)  \; \subseteq \;  \big( U_q \wedge \gerI_q \big) \cap \nabla(U_q)  \; \subseteq \; \big( U_q \wedge \gerI_q \big) \cap (q-1) \, U_q^{\,\wedge 2}  \; \subseteq \;  (q-1) \, U_q \wedge \gerI_q  $$
so that $ \gerI_q $  is proper.  A similar argument works for quantizations of type  $\gerC_q(K)\,$.
\end{proof}

\smallskip

\begin{remark}
 {\it The converse to  Lemma \ref{strictisproper}  here above is false}.
 \vskip4pt
  Indeed, there exist quantizations (of subgroups / homogeneous spaces) which are proper but  {\sl not\/}  strict: we present an explicit example   --- of type  $ \calC_q $  ---   in  Subsection \ref{non-coiso}  later on.
                                                                            \par
   This means that giving two different versions of ``quantization'' does make sense, in that they actually capture two  {\sl inequivalent\/}  notions   --- hierarchically related via  Lemma \ref{strictisproper}.
\end{remark}

\medskip

   The following statement clarifies why such definitions actually apply only to the (restricted) case of coisotropic subgroups (this result can be traced back to \cite{Lu}, where it is mentioned as \emph{coisotropic creed}).

\medskip

\begin{proposition}\label{coisotropic creed}
Let $K$ be a subgroup of $G$ and assume a proper quantization of it exists. Then $K$ is coisotropic or, in case the quantization is $\calC_q\,$, its observable hull $\widehat{K}$ is coisotropic.
\end{proposition}

\begin{proof}
 Assume  $ \, \calI_q $  exists.  Let  $ \, f, g \in \calI\, $,  and let  $ \, \varphi, \gamma \in \calI_q \, $  with  $ \,
\pi_{F_q}(\varphi) = f \, $,  $ \, \pi_{F_q}(\gamma) = g \, $.  Then by
definition  $ \, \{f,g\} = \pi_{F_q}\big( (q-1)^{-1} [\varphi,\gamma] \big)
\, $.  But
  $$  [\varphi,\gamma] \in \big[\calI_q \, , \calI_q \big] \subseteq(q-1) \, \calI_q  $$
\noindent  by assumption, hence  $ \, (q-1)^{-1} [\varphi,\gamma] \in
\calI_q \, $,  thus  $ \, \{f,g\} = \pi_{F_q} \big( (q-1)^{-1} [\varphi,\gamma]
\big) \in \pi_{F_q}(\calI_q) = \calI \, $,  which means that  $ \, \calI \, $
is closed for the Poisson bracket.  Thus (see  (2.6))  $ K$  is coisotropic.
                                          \par
   Similar arguments work when dealing with  $ \calC_q \, $,  $ \gerI_q $
or  $ \gerC_q \, $.  We shall only remark that working with  $ \calC_q $
we end up with  $ \, \calC\big(\widehat{K}\big) = \calC(K) \leq_{\cal P}
F[G] \, $, whence  $\widehat{K}$  is coisotropic.
\end{proof}

\vskip5pt

Since we would like to show also what happens in the non coisotropic case, we will consider, also, the weakest possible   --- na{\"\i}ve ---   version of quantization.

\medskip

\begin{definition}\label{weakq}
 Let  $ F_q[G] $  and  $ U_q(\gerg) $  be a QFA and a QUEA for
$ G $  and  $ \gerg \, $  and let
  $$  \displaylines{
   \pi_{F_q} \colon \, F_q[G] \relbar\joinrel\twoheadrightarrow F_q[G] \Big/ \! (q\!-\!1) F_q[G] \; \cong \; F[G]  \cr
   \pi_{U_q} \colon U_q(\gerg) \relbar\joinrel\twoheadrightarrow U_q(\gerg)\! \Big/ \! (q-\!1) \, U_q(\gerg) \; \cong \; U(\gerg)  }  $$
be the specialization maps.  Let  $ \calI $,  $ \calC $,  $ \gerI $  and
$ \gerC $  be the algebraic objects associated with the subgroup  $ K $
of  $ G $  (see \ref{subgroup}).  We call  {\sl ``weak quantization''\/}  of each
of them any object  $\calI_q\,$,  $\calC_q\,$,  $\gerI_q$  or
$\gerC_q$  respectively, such that
 \begin{equation} \label{weakquant}
      \begin{array}{ccc}
   (a) \qquad  \calI_q \ideal_\ell \coideal \; F_q[G] \;\; ,
&   \;\;  \pi_{F_q}(\calI_q) \, = \, \calI  \\
   (b) \qquad  \calC_q \leq^1 \coideal_\ell \, F_q[G] \;\; ,
&  \;\;  \pi_{F_q}(\calC_q) \, = \, \calC  \\
   (c) \qquad  \gerI_q \ideal_\ell \coideal \;\; U_q(\gerg) \;\; ,
&  \;\;  \pi_{U_q}(\gerI_q) \, = \, \gerI  \\
   (d) \qquad  \gerC_q \leq^1 \coideal_\ell \, U_q(\gerg) \;\; ,
&  \;\;  \pi_{U_q}(\gerC_q) \, = \, \gerC
      \end{array}
 \end{equation}
\end{definition}

\smallskip

It is obvious that strict or proper quantizations are weak. Let us remark that every subgroup of  $ G $  is quantizable in the weak sense, since we may just consider  e.g.~$ \, \calI_q := \pi_{F_q}^{-1}(\calI) \, $  to be a quantization of  $ \calI \, $.  As na{\"\i}f as it may seem, this remark will play a r{\^o}le in what follows.

\smallskip

Let us lastly remark how the real case should be treated.

\medskip

\begin{definition}\label{realq}
 Let  $( F_q[G]\,,*) $  and  $ (U_q(\gerg)\,,*) $  be a real QFA and a real QUEA for
$ G $  and  $ \gerg \, $.  Let $\calI_q\,$, $\calC_q\,$, $\gerI_q$  and  $\gerC_q$ be subgroup quantizations (either strict, proper or weak). Then such quantizations are called  {\sl real}  if
\begin{equation}\label{realquant}
  {\big( S(\calI_q) \big)}^\star \, = \, \calI_q \; ,  \qquad  \calC_q^\star \, = \, \calC_q \; ,  \qquad
{\big( S(\gerI_q) \big)}^\star \, = \, \gerI_q \; ,  \qquad  \gerC_q^\star \, = \, \gerC_q
\end{equation}
\end{definition}

\medskip

\begin{free text}
 {\bf The formal quantization approach.}  {\rm In the present work we are dealing with global quantizations. In  \cite{CiGa}  instead we treated  {\sl formal quantizations\/}:  these are topological Hopf  $ \C[[h]] $--algebras  which for  $ \, h = 0 \, $  yield back the (formal) Hopf algebras associated with a (formal) group.  In this case, such objects as  $ \calI_q \, $,  $ \calC_q \, $,  $ \gerI_q $  and  $ \gerC_q $  are defined in the parallel way.  However, in  \cite{CiGa}  we did  {\sl not\/}  consider the notions of  {\sl proper\/}  nor  {\sl weak\/}  quantizations but only dealt with strict quantizations.  Actually, one can consider the notions of proper or weak quantizations in the formal quantization setup as well; then the relation between these and strict quantizations will be again the same as we showed here above.
                                                                      \par
   We point out also that the semiclassical limits of formal quantizations are just formal Poisson groups, or their universal enveloping algebras, or subgroups, homogeneous spaces, etc.  In any case, this means   --- see the end of  Subsection \ref{subgrps-homspaces}  ---   that no restrictions on subgroups apply (all are ``observable'') nor on homogeneous spaces (all are ``quasi-affine'').}
\end{free text}

\bigskip

\section{Quantum duality principle}

 Drinfeld's quantum duality principle (cf.~\cite{Dr}, \S 7; see  also \cite{Ga1}  for a proof) has a stronger version (see \cite{Ga3}) best suited for  {\sl our\/}  quantum groups   --- in the sense of  Section \ref{quantizations}.

\medskip

   Let  $ H $  be any Hopf algebra in  $\mathcal{HA}$ and  let
\begin{equation}\label{kernelI}
  I  \; := \;  \Ker\, \Big( H \, {\buildrel \epsilon \over {\relbar\joinrel\twoheadrightarrow}}
\; \Cqqm \,{\buildrel {ev}_1 \over {\relbar\joinrel\relbar\joinrel\twoheadrightarrow}}\,
 \C \,\Big)  \; = \;  \Ker\,\Big( H \,{\buildrel {ev}_1 \over
{\relbar\joinrel\relbar\joinrel\twoheadrightarrow}}\, H \big/ (q\!-\!1) \, H \,{\buildrel
\bar{\epsilon} \over {\relbar\joinrel\twoheadrightarrow}}\; \C \,\Big)
\end{equation}
Then $I$ is a Hopf ideal of  $ H $.  We define
\begin{equation}\label{check}
  H^\vee  \; := \;  {\textstyle \sum_{n \geq 0}}
\, {(q-1)}^{-n} I^n  \; = \;
 {\textstyle \bigcup_{n \geq 0}} \, {\big(
{(q-1)}^{-1} I \, \big)}^n  \;\; \Big( \subseteq  \C(q) \otimes_{\C[q,q^{-1}]} H \Big)
\end{equation}
Notice that, setting  $ \, J := \Ker\, \Big( H \,{\buildrel
\epsilon \over {\relbar\joinrel\twoheadrightarrow}}\; \Cqqm \Big) \, $,  one has  $ \, I = (q-1) \cdot 1_{\scriptscriptstyle H} + J \, $,  so that
\begin{equation}\label{check2}
 H^\vee  \; = \;  {\textstyle \sum_{n \geq 0}} \,
{(q-1)}^{-n} J^n \; = \; {\textstyle \sum_{n \geq 0}}
\, {\big( {(q-1)}^{-1} J \, \big)}^n
\end{equation}
Consider, now, for every  $n \in \N$ the iterated coproduct
$\Delta^n \colon H \rightarrow H^{\otimes n} \; $  where
\[
\Delta^0:= \epsilon \qquad \Delta^1 :={ \mathrm id}_H \qquad
\Delta^n := \big( \Delta \otimes \id_H^{\otimes (n-2)} \big) \circ \Delta^{n-1}\quad {\mathrm if}\,\, n \ge 2 \; .
\]
For any ordered subset  $ \, \Sigma = \{i_1, \dots, i_k\} \subseteq
\{1, \dots, n\} \, $  with  $ \, i_1 < \dots < i_k \, $,  \, define
the morphism  $ \; j_{\mathrm \Sigma} : H^{\otimes k}
\longrightarrow H^{\otimes n} \; $  by
\[
 j_{\mathrm\Sigma} (a_1 \otimes \cdots \otimes a_k) :=
b_1 \otimes \cdots\otimes b_n\,\,\mathrm{where}\,\, \left\{\begin{array}{ccc} b_i := 1 &\mathrm{ if} &\, i \notin \Sigma \\
b_{i_m} := a_m &\mathrm{if} & 1 \leq m \leq k
\end{array}\right.
\]
then set  $ \, \Delta_\Sigma := j_{\mathrm \Sigma} \circ \Delta^k \, $,  $ \, \Delta_\emptyset := \Delta^0 \, $,  and  $ \, \delta_\Sigma := \sum_{\Sigma' \subset \Sigma} {(-1)}^{n- \left| \Sigma' \right|} \Delta_{\Sigma'} \, $,  $ \, \delta_\emptyset := \epsilon \, $.  By the inclusion-exclusion principle, the inverse formula
$\Delta_\Sigma = \sum_{\Psi\subseteq \Sigma} \delta_\Psi$  holds.  We shall use notation
$ \, \delta_0 := \delta_\emptyset \, $,  $ \, \delta_n :=
\delta_{\{1, 2, \dots, n\}} \, $,  and the key identity  $ \;
\delta_n = {(\id_{\mathrm H} - \epsilon)}^{\otimes n}
\circ \Delta^n \, $,  \, for all  $ \, n \in \N_+ \, $.  Given
$H \in {\mathcal H}$,  we define
\begin{equation}\label{prime}
 H' := \big\{\, a \in H \,\big\vert\, \delta_n(a) \in
{(q-1)}^n H^{\otimes n} , \; \forall \,\, n \in \N \, \big\}
\quad  \big( \! \subseteq H \, \big) \, .
\end{equation}

\smallskip

\begin{theorem}[Global Quantum Duality Principle]\label{GQDP}
(cf.~\cite{Ga3}) For any  $ \, H \in \HA \,$ one has:
\begin{itemize}
\item[(a)] $H^\vee$ is a QUEA and $H'$ is a QFA. Moreover the following inclusions hold:
\begin{equation}\label{GQDP-incl}
   H \subseteq {\big( H^\vee \big)}'  \,\; ,  \quad  H \supseteq {\big( H' \big)}^{\!\vee}  \; ,  \quad  H^\vee \! = \! \big( \! \big(H^\vee\big)' \,\big)^{\!\vee}  \; ,  \quad  H' \! = \! \big( \! \big(H'\big)^{\!\vee} \big)'
\end{equation}
\item[(b)] $H=\big(H^\vee\big)'\iff H$ is a QFA,  \;  and  \;  $ H = \big(H'\big)^\vee\iff H$ is a QUEA;
\item[(c)] If $\,G$ is a Poisson group with Lie bialgebra  $ \gerg \, $,  then
\[
F_q[G]^\vee\Big/(q-1)F_q[G]^\vee=U(\mathfrak g^*)\,\qquad U_q(\mathfrak g)^\prime\big/ (q-1)U_q(\mathfrak g)^\prime= F[G^*]
\]
where $G^*$ is some connected Poisson group dual to  $ G $;
\item[(d)] Let $F_q[G]$ and $U_q(\gerg)$  be dual to each other w.r.~to some perfect Hopf pairing. Then ${F_q[G]}^\vee$ and ${U_q(\gerg)}^\prime$ are dual to each other w.r.~to the same pairing.
\end{itemize}
\end{theorem}

\bigskip

A number of remarks are due, at this point:

\begin{enumerate}
\item The Poisson group $G^*$ dual to $G$ appearing in  {\it (c)\/}  of Theorem \ref{GQDP} does depend on $U_q(\gerg)$  which is given as a data. Different choices of $U_q(\gerg)$, though associated with the same Lie bialgebra $\gerg$ may give rise to a different connected Poisson dual group $G^*$.
\item For all Hopf  $ \Cq $--algebra  $\mathbb H$ the existence of a $\Cqqm$-integral form $H_f$ which is a QUEA at $q=1$ is equivalent to the existence of a $\Cqqm$--integer form $H_u$ which is a QFA at $q=1\,$.
\item All claims above have obvious analogues in the real case.
\item If $H$ is a Hopf algebra and $\Phi\subseteq\mathbb N$ is a finite subset, then (\cite{KT}, Lemma 3.2)
\begin{equation}
\label{deltaone}
\delta_\Phi(ab) \; = {\textstyle \sum\limits_{\Lambda\cup Y=\Phi}} \delta_\Lambda(a) \, \delta_Y(b) \qquad \forall \;\; a, b \in H
\end{equation}
furthermore, if  $ \, \Phi \ne \emptyset \, $  we have
\begin{equation}
\label{deltatwo}
\delta_\Phi(ab - ba) \; =
{\textstyle \sum\limits_{\substack{\Lambda \cup Y = \Phi  \\ \Lambda \cap Y \not= \emptyset}}}
\big( \delta_\Lambda(a) \, \delta_Y(b) - \delta_Y(b) \,
\delta_\Lambda(a) \big) \qquad  \forall \;\; a, b \in H
\end{equation}
The above formulas will be used frequently in what follows
\end{enumerate}

Having clarified the exact statement of quantum duality principle that we have in mind, let us extend it to objects of subgroup type as in Definition \ref{weakq}, i.e.~to left coideal subalgebras and to left ideals and two-sided coideals   --- either in $F_q[G]$ or in $U_q(\gerg)\,$.  This was already done in \cite{CiGa} where we only considered  {\sl local\/}  (i.e.~over $\mathbb C[[h]]$) quantizations. Let us remark that the quantum duality principle we have in mind not only exchanges the r{\^o}le of algebras of functions with that of universal enveloping algebras, but also exchanges the r{\^o}le of subgroups with that of homogeneous spaces. At the semiclassical level, the pair of dual objects is given by a coisotropic subgroup $H$ and a Poisson quotient $G^*\!\big/H^\bot\,$. When $H$ is a Poisson subgroup, its orthogonal $H^\bot$ turns out to be normal in $G^*$ and $G^*\!\big/H^\bot \cong H^*$ as a Poisson group, thus recovering the usual quantum duality principle. In particular, we will consider a process moving along the following draft:
  $$  \begin{array}{cccccccc}
 \hbox{\it (a)} \quad  &  \calI \; \big( \! \subseteq \! F[G] \, \big)  &
\hskip-3pt {\buildrel (1) \over \longrightarrow} \hskip-3pt  &  \calI_q \; \big( \! \subseteq \!
F_q[G] \, \big)  &  \hskip-3pt {\buildrel (2) \over \longrightarrow} \hskip-3pt  &
{\calI_q}^{\!\curlyvee} \; \big( \! \subseteq \! {F_q[G]}^\vee \, \big)  &
\hskip-3pt {\buildrel (3) \over \longrightarrow} \hskip-3pt  &  {\calI_1}^{\!\curlyvee}
\; \big( \! \subseteq \! U(\gerg^*) \big)
\\
   \hbox{\it (b)} \quad  &  \calC \; \big( \! \subseteq F[G] \, \big)  &
\hskip-3pt {\buildrel (1) \over \longrightarrow} \hskip-3pt  &  \calC_q \; \big( \!
\subseteq \! F_q[G] \,\big)  &  \hskip-3pt {\buildrel (2) \over \longrightarrow} \hskip-3pt
&  {\calC_q}^{\!\!\triangledown} \; \big( \! \subseteq \! {F_q[G]}^\vee \,\big)
&  \hskip-3pt {\buildrel (3) \over \longrightarrow} \hskip-3pt  &
{\calC_1}^{\!\!\triangledown} \; \big( \! \subseteq \!
U(\gerg^*) \big)  \\
   \hbox{\it (c)} \quad  &  \gerI \; \big( \! \subseteq \! U(\gerg) \,\big)  &
\hskip-3pt {\buildrel (1) \over \longrightarrow} \hskip-3pt  &  \gerI_q \; \big( \! \subseteq
\! U_q(\gerg) \, \big)  &  \hskip-3pt {\buildrel (2) \over \longrightarrow} \hskip-3pt  &
{\gerI_q}^{\! !} \; \big( \! \subseteq \! {U_q(\gerg)}' \,\big)  &
\hskip-3pt {\buildrel (3) \over \longrightarrow} \hskip-3pt  &  {\gerI_1}^{\! !} \;
\big( \! \subseteq \! F[G^*] \big)  \\
   \hbox{\it (d)} \quad  &  \gerC \; \big( \! \subseteq \! U(\gerg) \,\big)  &
\hskip-3pt {\buildrel (1) \over \longrightarrow} \hskip-3pt  &  \gerC_q \; \big( \!
\subseteq \! U_q(\gerg) \,\big)  &  \hskip-3pt {\buildrel (2) \over \longrightarrow} \hskip-3pt
&  {\gerC_q}^{\!\!\Lsh} \; \big( \! \subseteq \! {U_q(\gerg)}' \,\big)  &
\hskip-3pt {\buildrel (3) \over \longrightarrow} \hskip-3pt  &  {\gerC_1}^{\!\!\Lsh} \;
\big( \! \subseteq \! F[G^*] \big)
\end{array}  $$
\noindent
where arrows $(1)$ are quantizations, arrows $(3)$ are specializations at $q=1$ and the definition of arrows $(2)$ will be the core of what follows. It will turn out that:
\begin{enumerate}
\item  each one of the
right-hand-side objects above is one of the four algebraic
objects which describe a closed connected subgroup of  $ G^* \, $:
namely, the correspondence is
 \vskip1pt
   \centerline{ $
\text{\it (a)} \, =\joinrel\Longrightarrow \, \text{\it (c)}
\; ,  \hskip17pt
\text{\it (b)} \, =\joinrel\Longrightarrow \, \text{\it (d)}
\; ,  \hskip17pt
\text{\it (c)} \, =\joinrel\Longrightarrow \, \text{\it (a)}
\; ,  \hskip17pt
\text{\it (d)} \, =\joinrel\Longrightarrow \, \text{\it (b)}
\; . $ }
\item  the four quantizations of subgroups of  $ G^* $
so obtained are always  {\sl proper}   --- hence the subgroups of  $ G^* $
associated with them are  {\sl coisotropic}.
\item   if we begin with  {\sl
strict\/}  quantizations, and we start from  a subgroup $K$,  then the quantization of the unique
coisotropic closed connected subgroup of  $G^* $  mentioned above
is  {\sl strict\/}  as well, and the subgroup itself is  $ K^\perp $
(cf.~Definition \ref{coisotropic duality}), with some care in case  {\it (b)},  i.e.~if we
start from  $ \calC(K)$.  This will partially generalize to  {\sl weak\/}
quantizations, for which, starting from a subgroup  $K$  of $G$,  the unique
coisotropic closed connected subgroup of  $G^*$  obtained above is
$ K^{\langle \perp \rangle} $  (cf.~Proposition ~\ref{bot-pro}).
\end{enumerate}

\medskip

Let us fix, in what follows, quantizations $U_q(\mathfrak g)$ and $F_q[G]$ as in  Section \ref{quantizations}.  Unless explicitly mentioned we will not assume that this is a double quantization. To simplify notations, let us set
  $$  \displaylines{
   \U_q \, := \, \U_q(\gerg)  \quad ,  \qquad  U_q \, := \, U_q(\gerg)  \quad ,
\qquad  {U_q}' \, := \, {U_q(\gerg)}'  \cr
   \F_q \, := \, \F_q[G]  \quad ,  \qquad  F_q \, := \, F_q[G]  \quad ,  \qquad  {F_q}^\vee \, := \, {F_q[G]}^\vee  }  $$
As mentioned in the first remark after Theorem \ref{GQDP}, this implies that a specific connected Poisson dual $G^*$ of $G$ is selected (it depends on the choice of  $ \, U_q := U_Q(\gerg) \, $,  not only on  $ \gerg $  itself). Let us consider quantum subgroups $ \; \calI_q \, $,  $ \, \calC_q\, $,  $ \, \gerI_q \, $  and  $ \, \gerC_q \, $  as defined in \ref{weakq}.

 \eject

\begin{definition}\label{maps}
Using notations as in (\ref{kernelI}) we define:
\begin{itemize}
  \item[(a)] \hskip7pt  $ \displaystyle{ {\calI_q}^{\!\curlyvee}
\; := \;  {\textstyle \sum\limits_{n=1}^\infty} \, {(q-1)}^{-n} \cdot
I^{\,n-1} \cdot \calI_q  \; = \;  {\textstyle \sum\limits_{n=1}^\infty}
\, {(q-1)}^{-n} \cdot J^{\,n-1} \cdot \calI_q } $
  \item[(b)] \hskip7pt  $ \displaystyle{ {\calC_q}^{\!\!
\triangledown}  \; := \;  {\textstyle \sum\limits_{n=0}^\infty}
{(q-1)}^{-n} \cdot {\big( \calC_q \cap I \,\big)}^n \, = \;
{\textstyle \sum\limits_{n=0}^\infty} \, {(q-1)}^{-n} \cdot
{\big( \calC_q \cap J \big)}^n } $
  \item[(c)] \hskip7pt  $ \displaystyle{ {\gerI_q}^{\! !}  \; :=
\;  \Big\{\, x \in \gerI_q \;\Big\vert\; \delta_n(x) \in {(q-1)}^n \,
{\textstyle \sum_{s=1}^n} \, {U_q}^{\otimes (s-1)} \! \otimes \gerI_q
\otimes {U_q}^{\otimes (n-s)} \! , \; \forall\; n \in \N_+ \Big\} } $
  \item[(d)] \hskip7pt  $ \displaystyle{ {\gerC_q}^{\!\!\Lsh}
\; := \;  \Big\{\, x \in \gerC_q \;\Big\vert\; \delta_n(x) \in {(q-1)}^n
\, {U_q}^{\otimes (n-1)} \! \otimes \gerC_q \, , \; \forall\; n \in \N_+
\,\Big\} } $
\end{itemize}
\end{definition}

\smallskip

Let us remark that the following inclusions hold directly by definitions:
\begin{equation}
\label{easy-incl}
 (i) \quad  {\calI_q}^{\!\curlyvee} \supseteq \calI_q \; ,  \!\!\qquad
(ii) \quad {\calC_q}^{\!\!\triangledown} \supseteq \calC_q \; ,
\!\!\qquad  (iii) \quad  {\gerI_q}^{\! !} \subseteq \gerI_q \; ,
\!\!\qquad  (iv) \quad  {\gerC_q}^{\!\!\Lsh} \subseteq \gerC_q \; .
\end{equation}

\bigskip

\section{Duality maps}

In the present section we will prove properties of the four Drinfeld--type maps defined in the previous section, namely the maps  $ \, \calI_q \mapsto {\calI_q}^{\!\curlyvee} \, $,  $ \, \calC_q \mapsto {\calC_q}^{\!\!\triangledown} \, $,
$ \, \gerI_q \mapsto {\gerI_q}^{\! !} \, $  and  $ \, \gerC_q \mapsto {\gerC_q}^{\!\!\Lsh} \, $.
Let us recall that such maps do not  change, as we will see, the algebraic properties of subobjects, but interchanges quantized function algebra with quantum enveloping algebra and therefore quantizations of coisotropic subgroups will be sent to quantizations of (embeddable) homogeneous spaces   --- of the dual quantum group ---   and viceversa.
                                                                           \par
   Let us start by considering the map  $ \, \calI_q \mapsto {\calI_q}^{\!\curlyvee} \, $.

\smallskip

\begin{proposition}\label{curlyvee} Let  $ \, \calI_q = \calI_q(K) \, $  be
a left ideal and two-sided coideal in $F_q[G]\,$,  that is a weak quantization (of type  $ \calI \, $)  of some subgroup $ K $  of  $ G \, $.  Then
\begin{enumerate}
  \item  \;  $ {\calI_q}^{\!\curlyvee}$ is a left ideal and two-sided coideal in  ${F_q[G]}^\vee$;
  \item  \;  if  $ \, \calI_q $  is  {\sl strict},  then  $ {\calI_q}^{\!\curlyvee} $  is strict too, i.e.~$ \;\; {\calI_q}^{\!\curlyvee} \! \bigcap \, (q-1) \,
{F_q[G]}^\vee \, = \; (q-1) \, {\calI_q}^{\!\curlyvee} \;\, $;
  \item  \;  there exists a coisotropic subgroup $L$ of $G^*$ such that  $ {\calI_q(K)}^{\!\curlyvee} \! = \gerI_q(L) \; $:  namely,  $ {\calI_q(K)}^{\!\curlyvee} $  is a  {\sl proper}  quantization, of type  $ \gerI $,  of some coisotropic subgroup  $ L $  of  $ \, G^* \, $;
  \item  \;  in the  {\sl real case},  i.e.~if the quantization  $ \calI_q $  is a real one,  $ {\calI_q}^{\!\curlyvee} $  is real too, i.e.~$ \, {\big( S \big(
{\calI_q}^{\!\curlyvee} \big) \big)}^* \! = {\calI_q}^{\!\curlyvee}\,$.  Therefore claims (1--3) still hold in the framework of  {\sl real}  quantum subgroups.
\end{enumerate}
\end{proposition}

\begin{proof}
 {\it (1)} \,  Consider that   $ \, {\calI_q}^{\!\curlyvee} \, $
is  the left ideal of  $ {F_q}^{\!\vee} $
generated by  $ \, {(q-1)}^{-1} \, \calI_q \; $;   therefore, in
order to prove  $ \, {\calI_q}^{\!\curlyvee} \coideal\, {F_q}^{\!\vee}
\, $  it is enough to show that  $ \, \Delta \big( {(q\!-\!1)}^{-1} \,
\calI_q \big) \, \subseteq \, {F_q}^{\!\vee} \!\otimes {\calI_q}^{\!
\curlyvee} + \, {\calI_q}^{\!\curlyvee} \!\otimes {F_q}^{\!\vee} \, $.  Since  $ \calI_q $  is a coideal of  $ F_q \, $, we have
\begin{equation}\label{coideal}
  \Delta \big(\! {(q\!-\!1)}^{-1} \, \calI_q \,\big) \, \subseteq \,
F_q \otimes {(q\!-\!1)}^{-1} \, \calI_q + {(q\!-\!1)}^{-1} \, \calI_q
\otimes F_q \, \subseteq \, {F_q}^{\!\vee} \!\otimes {\calI_q}^{\!
\curlyvee} + {\calI_q}^{\!\curlyvee} \!\otimes {F_q}^{\!\vee}
\end{equation}
whence  $ \, {\calI_q}^{\!\curlyvee} \coideal\, {F_q}^{\!\vee} \, $
follows, and the first  claim is proved.
 \vskip5pt
   {\it (2)} \,  Assume $ \calI_q $ to  be a  {\sl strict}  quantization, so that  $ \, \calI_q \bigcap \, (q-1) \, F_q \, = \,(q-1) \, \calI_q \; $.
                                          \par
   Let  $\, J := \Ker\, \big(\epsilon \, \colon \, F_q
\longrightarrow \Cqqm \,\big) \; $.  Then
\[
J \! \mod (q \!- \! 1) F_q = \Ker\,(\epsilon){\big|}_{F[G]}  \; = \;  \mathfrak{m}_e
\]
 and  $\mathfrak{m}_e\Big/ {\mathfrak{m}_e}^{\!2} = \gerg^*$,
 the cotangent Lie bialgebra of  $ G \, $.  Let  $ \, \{y_1,\dots,y_n\}$  be a subset of
$ \mathfrak{m}_e $  whose image in the local ring of  $ G $  at the identity $ e$
is a local system of parameters, and pull it back to a subset  $\{j_1,\dots,j_n\}$
of  $J$.  Let  $ \widehat{F}_q $  be the  $ J $--adic  completion
of  $ F_q$.  From \cite{Ga3}, Lemma 4.1, we know that the set of ordered
monomials  $\big\{ j^{\,\underline{e}} \;\big|\; \underline{e} \in
\N^{\,n} \big\}$  (where hereafter  $ \, j^{\,\underline{e}} :=
\prod_{s=1}^n j_s^{\,\underline{e}(i)} \, $,  \, for all  $ \,
\underline{e} \in \N^{\,n} \, $)  is a  $ \Cqqm $--pseudobasis
 of  $ \widehat{F}_q \, $,
which means that each element of  $ \widehat{F}_q $  has a unique
expansion as a formal infinite linear combination of the  $ j^{\,
\underline{e}} $'s.  In a similar way, the  $ (q-1) $--adic
completion of  $ {F_q}^{\!\vee} $  admits  $ \, \big\{ {(q -
1)}^{-|\underline{e}|} j^{\,\underline{e}} \;\big|\; \underline{e}
\in \N^{\,n} \big\} \, $  as a  $ \Cqqm $--pseudobasis,  where
$ \, |\underline{e}| := \sum_{i=1}^n \underline{e}(i) \, $.
                                                \par
   For our purposes we need a special choice of the set  $ \{j_1,\dots, j_n\}$ adapted to the smooth subvariety
$K$  of  $ G $.  By general theory we can choose  $ \{y_1, \dots, y_n\} $  so that  $ \; y_1 $,  $ \dots $,
$ y_k \in \mathfrak{m}_e \; $  and  $ \; y_{k+1} , \dots , y_n \in
\calI(K) \, $,   where  $ \, k = \text{\it dim}\,(K) \, $.
We can also choose the lift  $ \{j_1, \dots, j_n\} $  of  $ \{y_1,
\dots, y_n\} $  inside  $ J $  so that  $ j_s $  is a lift of
$ y_s $,  for all  $ \, s = 1, \dots, k \, $,  \, and  $ \,
j_{k+1}, \dots, j_n \in \calI_q \, $.  With these assumptions,
 it's easy to see that
  $$  \varphi\in {\calI_q}^{\!\curlyvee} \,{\textstyle \bigcap}\, (q-1) \, {F_q}^{\!\vee}  \;\;\Longrightarrow\;\;
{(q-1)}^n \, \varphi \in \big( J^{n-1} \cdot \calI_q\big) \,{\textstyle \bigcap}\, (q-1) \, J^n  $$
\noindent  for some  $n \in \N $, which in turn yields  $ \; {(q-1)}^n \varphi \in J^{n-1} \cdot
\big( \calI_q \, \bigcap \, (q-1) \, J \big) \, $. Since
  $$  \calI_q \,{\textstyle \bigcap}\, (q-1) \, J \; \subseteq \; \calI_q \,{\textstyle \bigcap}\, (q-1) \, F_q  \; = \; (q-1) \, \calI_q  $$
we conclude that  $ \; {(q-1)}^n \, \varphi \in (q-1) \, J^{n-1} \cdot \calI_q \, $,
\, whence  $ \; \varphi \in (q-1) \, {\calI_q}^{\!\curlyvee} \, $.
The converse inclusion  $ \, {\calI_q}^{\!\curlyvee} \bigcap \,
(q-1) \, {F_q}^\vee \, \supseteq \, (q-1) \, {\calI_q}^{\!\curlyvee} \, $  is obvious, hence claim  {\it (2)\/}  is proved.
 \vskip5pt
   {\it (3)} \,  It is an obvious statement that  $ {\calI_q}^{\!\curlyvee} $  is a weak quantization of its image  $ \, \pi_{{F_q}^{\!\vee}} \big( {\calI_q}^{\!\curlyvee} \big) \, $:  in particular,
$ \pi_{{F_q}^{\!\vee}} \big( {\calI_q}^{\!\curlyvee} \big) \ideal_\ell \! \coideal \; \pi_{{F_q}^{\!\!\vee}} \big( {F_q}^{\!\vee} \big) = U\big(\gerg^*\big) \, $  implies that
$\pi_{{F_q}^{\!\vee}} \big( {\calI_q}^{\!\curlyvee} \big) = \gerI(L) \, $  for some subgroup  $L$  of  $G^* $.  Thus  $ {\calI_q}^{\!\curlyvee} $  is a weak quantization, to be called  $ \gerI_q(L) $, of  $ \gerI(L) \, $, and it is even strict if  $ \calI_q $  itself is strict, as we've just seen. Now we show that such quantization  $ \gerI_q(L) $  turns out to be always  {\sl proper}.
                                                 \par
   In fact, (\ref{coideal})  implies $\nabla \big( {(q-1)}^{-1} \, \calI_q \big) \, \subseteq \, {(q-1)}^{-1} \, \big( F_q \wedge \calI_q \big) \, $.  On the other hand  $F_q \wedge \calI_q \, \subseteq \, J \wedge \calI_q \, \subseteq \,
{(q-1)}^2 \, {F_q}^\vee \!\wedge {\calI_q}^{\!\curlyvee}$,   thus, finally,
$\nabla\big({\calI_q}^{\!\curlyvee}\big) \in (q-1) \, {F_q}^\vee \!\wedge {\calI_q}^{\!\curlyvee}$,  which means that  $ {\calI_q}^{\!\curlyvee} $  is proper and  {\it (3)\/}  holds.
 \vskip5pt
   {\it (4)} \,  This is an obvious consequence of definitions.
\end{proof}

\smallskip

\begin{remark}
 In functorial language we may say that the map
$\calI_q \mapsto {\calI_q}^{\!\curlyvee}$ establishes a functor between quantizations of coisotropic subgroups of $G$ and quantizations of (embeddable) homogeneous spaces of $G^*$, moving from a global to a local description,  sending each type of quantization in a proper one and preserving strictness.  Indeed, we should make precise what are the ``arrows'' in our categories of ``quantum subgroups'' or ``quantum homogeneous spaces'', and how the functor acts on these: we leave these details to the interested reader.
\end{remark}

\medskip

   Let us move on to properties of the map  $ \; \calC_q \mapsto
{\calC_q}^{\!\!\triangledown} \; $.

\smallskip

\begin{proposition}\label{triangledown} Let  $ \, \calC_q = \calC_q(K) \, $  be a left coideal subalgebra in $F_q[G]$.
  Then
\begin{enumerate}
  \item  \;  $ {\calC_q}^{\!\!\triangledown}$ is a left coideal subalgebra in ${F_q[G]}^\vee$;
  \item  \;  if  $ \, \calC_q $  is  {\sl strict},  then  $ {\calC_q}^{\!\!\triangledown} $  is strict too,  i.e.~$ \,\; {\calC_q}^{\!\!\triangledown} \bigcap \, (q-1) \, {F_q[G]}^\vee \, = \; (q-1) \, {\calC_q}^{\!\!\triangledown} \;\, $.
  \item  \;  there exists a coisotropic subgroup $L$ of $G^*$ such that $ {\calC_q(K)}^{\!\triangledown} \! = \gerC_q(L) \, $:  namely,  $ {\calC_q(K)}^{\!\triangledown} $  is a  {\sl proper}  quantization, of type $ \gerC \, $, of some coisotropic subgroup  $ L $  of  $ \, G^* \, $;
  \item  \;  in the  {\sl real case},  i.e.~if the quantization  $ \calC_q $  is a real one,  $ {\calC_q(K)}^{\!\triangledown} $  is real too,  i.e.~$ \, {\big( {\calC_q}^{\!\!\triangledown} \big)}^* \! = {\calC_q}^{\!\!
\triangledown} \, $.  Therefore claims (1--3) still hold in the framework of  {\sl real} quantum subgroups.
\end{enumerate}
\end{proposition}

\begin{proof}
 The proof uses essentially the same arguments as the previous one.
 \vskip5pt
   {\it (1)} \,  By the very definitions  $ \; {\calC_q}^{\!\!\triangledown}
\! \leq^1 \! {F_q}^{\!\vee} := {F_q[G]}^\vee \, $.  More precisely,  $ \,
{\calC_q}^{\!\!\triangledown} \, $  is (by construction) the unital
$ \Cqqm $--subalgebra  of  $ {F_q}^{\!\vee} $  generated by  $ \,
{(q-1)}^{-1} \, {(\calC_q)}^+ \, $,  \, where  $ \, {(\calC_q)}^+
:= \calC_q \bigcap\, J \, $.  So to get  $ \, {\calC_q}^{\!\!
\triangledown} \coideal_\ell \, {F_q}^{\!\vee} \, $  we must
only prove  $ \, \Delta \big( {(q\!-\!1)}^{-1} {(\calC_q)}^+
\big) \! \subseteq {F_q}^{\!\vee} \!\otimes {\calC_q}^{\!\!\triangledown}
\, $.  But  $ \, \calC_q \, \coideal_\ell \, F_q \, $,  so:
\begin{equation}\label{coideal2}
\Delta \big( {(q-1)}^{-1} {(\calC_q)}^+ \big) \, \subseteq \, F_q
\otimes {(q-1)}^{-1} {(\calC_q)}^+ \, \subseteq \, {F_q}^{\!\vee}
\!\otimes {\calC_q}^{\!\!\triangledown}
\end{equation}
therefore  $ \,{\calC_q}^{\!\!\triangledown} \, \coideal_\ell \, {F_q}^{\!\vee}$, and claim  {\it (1)\/}  is proved.
 \vskip5pt
   {\it (2)} \,  Now suppose $ \calC_q $ to  be a  {\sl strict\/}
quantization, i.e.  $ \; \calC_q \bigcap \, (q-1) \, F_q \, =
\, (q-1) \, \calC_q \; $.  We need an explicit description of  $ {F_q}^{\!\vee} $
and of  $ {\calC_q}^{\!\!\triangledown} \, $.  This goes along the
same lines followed to describe  $ {\calI_q}^{\!\curlyvee} $  in the
proof of Proposition \ref{curlyvee}: but now the choice of the subset  $ \{ j_1,
\dots, j_n\} $  of  $ J $  is different.
                                               \par
   First, since  $ \, \calC(K) = \calC\big(\widehat{K}\big) \, $
we can assume that  $ \, K = \widehat{K} \, $,  \, i.e.~$ K $  is
observable.  Then we can choose  $ \big\{ j_1, \dots, j_n \big\} $
so that  $ \, j_{k+1}, \dots, j_n \in J \, \bigcap \, \calC_q =
{\calC_q}^{\!+} \, $  (where again  $ \, k = \text{\it dim}\,(K)
\, $)  and, letting  $ \; y_s := j_s \! \mod (q-1) \, F_q \; $,
 the set  $ \big\{ y_1, \dots, y_n \big\} $  yields a local system
of parameters at  $ \, e \in G \, $  (in the localized ring), as
before; now in addition we have  $ \, y_{k+1}, \dots, y_n \in
\mathfrak{m}_e \, \bigcap \, \calC(K) =: {\calC(K)}^+ \, $.  With
these assumptions, the  $ (q-1) $--adic  completion of
$ {F_q}^{\!\vee} $  admits  $ \, \big\{ {(q-1)}^{-|\underline{e}|}
j^{\,\underline{e}} \;\big|\; \underline{e} \in \N^{\,n} \big\} \, $
as a  $ \Cqqm $--pseudobasis,  like before, but in addition the
same analysis can be done for the  $ (q-1) $--adic  completion
of  $ {\calC_q}^{\!\!\triangledown} $  (just because  $ \calC_q $
is  {\sl strict\/}),  which then has  $ \Cqqm $--pseudobasis  $ \,
\big\{ \prod_{s=k+1}^n j_s^{\,e_s} \,\big|\, (e_{k+1},\dots,e_n) \in
\N^{\,n-k} \big\} \, $.  From these description of the completions,
and comparing the former with  $ {F_q}^{\!\vee} $  and  $ \calC_q
\, $,  we easily see that  $ \; {\calC_q}^{\!\!\triangledown}
\, \bigcap \, (q-1) \, {F_q}^\vee \, \subseteq \, (q-1) \,
{\calC_q}^{\!\!\triangledown} \, $.  The converse is trivial,
hence claim  {\it (1)\/}  is proved.
 \vskip5pt
   {\it (3)} \, It follows directly from  {\it (1)\/}  that
$ {\calC_q}^{\!\!\triangledown} $  is a weak quantization
of its image  $ \, \pi_{{F_q}^{\!\vee}} \big( {\calC_q}^{\!\!\triangledown}
\big) \, $:  \, in particular,  $ \, \pi_{{F_q}^{\!\vee}} \big( {\calC_q}^{\!\!
\triangledown} \big) \leq^1 \! \coideal_\ell \; \pi_{{F_q}^{\!\!\vee}} \big(
{F_q}^\vee \big) = U \big( \gerg^* \big) \, $  means that  $ \, \pi_{{F_q}^{\!\vee}}
\big( {\calC_q}^{\!\!\triangledown} \big) = \gerC(L) \, $  for some subgroup
$ L$ of $G^*$.  Thus  $ {\calC_q}^{\!\!\triangledown} $  is a weak quantization   --- to be called  $ \gerC_q(L) $  ---   of  $ \gerC(L)$,  and it is even strict if  $ \calC_q $  itself is strict, by claim  {\it (1)}.  Now in addition we show that, in any case, such a quantization  $ \gerC_q(L) $  is always  {\sl proper}.
                                                 \par
From (\ref{coideal2}) we have
  $$  \nabla \big( {(q-1)}^{-1} \, {(\calC_q)}^+ \big) \, \subseteq
\, {(q-1)}^{-1} J \wedge {(\calC_q)}^+ \subseteq \, {(q-1)}^{-1+2}
\, {F_q}^\vee \!\wedge {\calC_q}^{\!\!\triangledown} \, = \, (q-1)
\, {F_q}^\vee \!\wedge {\calC_q}^{\!\!\triangledown}  $$
which implies exactly that  $ {\calC_q}^{\!\!\triangledown} $   --- which by definition is the  unital subalgebra generated by  ${(q-1)}^{-1} \, {(\calC_q)}^+$  ---   is proper.
 \vskip3pt
   {\it (4)} \, This follows directly from definitions and from
${\calC_q}^{\!*} = \, \calC_q$,  which holds
by assumption.
\end{proof}

\smallskip

\begin{remark}
 In functorial language we may say that the map $\calC_q\mapsto {\calC_q}^{\!\!\triangledown}$ establishes a functor between quantized homogeneous spaces of $G$ and quantizations of coisotropic subgroups of $G^*$, moving from a global to a local description,  sending each type of quantization in a proper one and preserving strictness. Again, to be precise, several details need to be fixed, and are left to the reader.
\end{remark}

\medskip

   The third step copes with the map  $ \, \gerI_q \mapsto
{\gerI_q}^{\! !} \; $.

\smallskip

\begin{proposition}\label{esclamativo}
 Let  $ \, \gerI_q = \gerI_q(K) \, $  be a left ideal and two-sided coideal in $U_q(\gerg)\,$,  weak quantization (of type  $ \gerI $)  of some coisotropic subgroup  $ K $ of  $G\,$.  Then:
\begin{enumerate}
\item  \;  $ \, {\gerI_q}^{\! !}$ is a left ideal and two-sided coideal in ${U_q(\gerg)}' $;
\item  \;  if  $ \, \gerI_q $  is  {\sl strict}, then  $ {\gerI_q}^{\! !} $  is strict too,  i.e.~$ \,\; {\gerI_q}^{\! !} \bigcap \, (q-1) \, {U_q(\gerg)}'
\, = \; (q-1) \, {\gerI_q}^{\! !} \;\, $;
\item  \;  there exists a coisotropic subgroup $L$ in $G^*$ such that   $ {\gerI_q(K)}^! \! = \calI_q(L) \, $:  namely,  $ {\gerI_q(K)}^! $  is a  {\sl proper}  quantization, of type  $ \calI \, $,  of some coisotropic subgroup  $ L $ of  $ \, G^* \, $;
\item  \;  in the  {\sl real case},  i.e.~if the quantization  $ \gerI_q $  is a real one,  $ {\gerI_q}^{\! !} $  is real too,  i.e.~$ \, {\big(S \big( {\gerI_q}^{\! !} \big) \big)}^* \! = {\gerI_q}^{\! !} $.  Therefore claims (1--3) still hold in the framework of  {\sl real}  quantum subgroups.
\end{enumerate}
\end{proposition}

\begin{proof}
 {\it (1)} \,  Let  $ a \in {U_q}' $  and  $b \in{\gerI_q}^{\! !}\, $:  by definition of  ${\gerI_q}^{\! !}$,
from  $ \, \gerI_q \ideal_\ell U_q \, $  and from (\ref{deltaone}) we get
  $$  \delta_n(a b)  \, \in \,  {(q-1)}^n {\textstyle \sum\limits_{s=1}^n} {U_q}^{\otimes (s-1)} \!
\otimes \gerI_q \otimes {U_q}^{\otimes (n-s)}  $$
so  $ \; a \, b \in {\gerI_q}^{\! !} \, $,  thus  $ \, {\gerI_q}^{\! !} \ideal_\ell {U_q}' \, $.
                                              \par
   As to the coideal property, it is proven resorting to  $ (q-1) $--adic
completions, arguing as in the proof of Proposition 3.5 in \cite{Ga3}, and basing on the fact that  $ \gerI_q
\coideal \; U_q \, $.  Details are left to the reader.
 \vskip5pt
   {\it (2)} \,  Assume now  $ \gerI_q $  to be {\sl strict}.  The
inclusion
  $$  {\gerI_q}^{\! !} \,{\textstyle \bigcap}\, (q-1) \, {U_q(\gerg)}' \; \supseteq \; (q-1) \, {\gerI_q}^{\! !}  $$
\noindent is trivially true, and we must prove the converse.  Let  $ \, \eta \in {\gerI_q}^{\! !} \, \bigcap \, (q-1) \, {U_q(\gerg)}' \, $.  We have
  $$  \delta_n(\eta)  \; \in \;  {(q-1)}^n \left( \left( {\textstyle
\sum_{s=1}^n} {U_q}^{\otimes (s-1)} \otimes \gerI_q \otimes
{U_q}^{\otimes (n-s)} \right) {\textstyle \bigcap} \, (q-1)
\, {U_q}^{\otimes n} \right)  $$
\noindent for all  $ \, n \in \N_+ \, $.  But then our assumption gives
  $$  \displaylines{
   \left(\, {\textstyle \sum\limits_{s=1}^n} \; {U_q}^{\otimes (s-1)}
\otimes \gerI_q \otimes {U_q}^{\otimes (n-s)} \right) \, {\textstyle
\bigcap} \; (q-1) \, {U_q}^{\otimes n}  =   \hfill  \cr
   \hfill   =  {\textstyle \sum\limits_{s=1}^n} \; {U_q}^{\otimes
(s-1)} \otimes \Big( \gerI_q \, {\textstyle \bigcap} \, (q \! - \! 1)
\, U_q \Big) \otimes {U_q}^{\otimes (n-s)}  = \,  {(q \! - \! 1)}^{n+1}
{\textstyle \sum\limits_{s=1}^n} {U_q}^{\otimes (s-1)} \otimes \gerI_q
\otimes {U_q}^{\otimes (n-s)}  \cr }  $$
which, in turn, means  $ \, \eta \in (q-1) \, {\gerI_q}^{\! !} \, $.  Thus
$ \, {\gerI_q}^{\! !} \, \bigcap \, (q-1) \, {U_q(\gerg)}' \, \subseteq
(q-1) \, {\gerI_q}^{\! !} \, $, as expected.
 \vskip5pt
   {\it (3)}  Claim  {\it (1)\/}  implies that  $ {\gerI_q}^{\! !} $  is a weak quantization
of its image, therefore there exists a subgroup $L$ of $G^*$ such that
$ \; \pi_{{U_q}^{\!\prime}} \big( {\gerI_q}^{\! !} \,\big) = \calI(L) \, $ .  This  quantization  is even strict if  $ \gerI_q $  itself is strict, by the previous.  Now we show that this quantization  $ \calI_q(L) $  is always  {\sl proper\/}  --- hence the subgroup  $L$  is coisotropic, by Lemma \ref{coisotropic creed}.
 \vskip3pt
   Recall that, by definition,  $ \calI_q(L) $  is proper if and only if  $ \, [x,y] \in (q-1) \, {\gerI_q}^{\! !} \, $  for all  $ \, x, y \in {\gerI_q}^{\! !} \, $. From definitions we have
  $$  [x,y] \in (q \! - \! 1) \, {\gerI_q}^{\! !}  \; \Longleftrightarrow \;  \delta_n \big( [x,y] \big) \in
{(q \! - \! 1)}^{n+1} \, {\textstyle \sum_{s=1}^n} \, {U_q}^{\otimes (s-1)} \otimes \gerI_q \otimes
{U_q}^{\otimes (n-s)}  \quad  \forall\; n \! \in \! \N  $$
Then by formula (\ref{deltatwo}) we have (for all  $n \in \N$)
 \begin{equation}\label{quarantuno}
 \delta_n \big([x,y]\big)  \; = \;  {\textstyle \sum_{\substack{ \Lambda
\cup Y = \{1,\dots,n\}  \\   \Lambda \cap Y \not= \emptyset }}}
\hskip-1pt \big( \, \delta_\Lambda(x) \, \delta_Y(y) \, - \,
\delta_Y(y) \, \delta_\Lambda(x) \, \big)
\end{equation}
while (with notation of \S 4)
  $$  \displaylines{
   \delta_\Lambda(x) \in {(q-1)}^{\vert \Lambda \vert} \cdot
j_\Lambda \left( {\textstyle \sum_{s=1}^{\vert \Lambda \vert}}
\, {U_q}^{\otimes (s-1)} \otimes \gerI_q \otimes {U_q}^{\otimes
(\vert \Lambda \vert -s)} \right) \, ,  \cr
   \delta_Y(y) \in {(q-1)}^{\vert Y \vert} \cdot j_Y \left(
{\textstyle \sum_{s=1}^{\vert Y \vert}} \, {U_q}^{\otimes (s-1)}
\otimes \gerI_q \otimes {U_q}^{\otimes (\vert Y \vert -s)} \right)
\, ;  \cr }  $$
since  $\Lambda \cup Y = \{1, \dots, n\}$  and
$\Lambda \cap Y \not= \emptyset \, $  we have  $ \, \vert \Lambda
\vert + \vert Y \vert \geq n + 1 \, $;  moreover, for each index
$ \, i \in \{1, \dots, n \} \, $  we have  $ \, i \in \Lambda \, $
(and otherwise  $ \, \text{\sl Im}\,(j_\Lambda) \, $  has  $ 1 $  in
the  $ i $--th  spot) or  $ \, i \in Y$ (with the like remark
on  $\text{\it Im}\,(j_Y)$  if not).  As  $ \gerI_q $  is a
left ideal of  $ U_q$,  we conclude
  $$  \displaylines{
   \quad  \delta_\Lambda(x) \cdot \delta_Y(y) \, ,  \; \delta_Y(y)
\cdot \delta_\Lambda(x) \, \in \, {(q-1)}^{\vert \Lambda \vert +
\vert Y \vert} \, {\textstyle \sum_{s=1}^n} \, {U_q}^{\otimes (s-1)}
\otimes \gerI_q \otimes {U_q}^{\otimes (n-s)}   \hfill  \cr
   \hfill   \subseteq {(q-1)}^{n+1} \, {\textstyle \sum_{s=1}^n} \,
{U_q}^{\otimes (s-1)} \otimes \gerI_q \otimes {U_q}^{\otimes (n-s)}
\quad  \cr }  $$
so that (\ref{quarantuno}) gives  $\displaystyle{ \delta_n \big( [x,y] \big) \in
{(q\!-\!1)}^{n+1} \, {\textstyle \sum_{s=1}^n} \, {U_q}^{\otimes (s-1)}
\otimes \gerI_q \otimes {U_q}^{\otimes (n-s)} } $,  as expected.
 \vskip5pt
   {\it (4)} \, In the real case,  $ \, \big( S \big( {\gerI_q}^{\! !\,} \big) \big)^* \! = {\gerI_q}^{\! !} \, $  follows at once from definitions and from the identity  $ \, \big( S(\gerI_q) \big)^{\!*} \! = \gerI_q \; $.
\end{proof}

\smallskip

\begin{remark}
 In functorial language we may say that the map $ \gerI_q \mapsto
{\gerI_q}^{\! !} \, $ establishes a functor between quantized homogeneous spaces of $G$ and quantizations of coisotropic subgroups of $G^*$, moving from a local to a global description,  sending each type of quantization in a proper one and preserving strictness.  Once more, details are left to the interested reader.
\end{remark}

\medskip

   The fourth and last step is devoted to the map  $ \, \gerC_q\mapsto {\gerC_q}^{\!\!\Lsh} \, $.

\smallskip

\begin{proposition}\label{Lsh}
Let  $\gerC_q = \gerC_q(K)$  be a subalgebra and left coideal in $U_q(\gerg)\,$,  weak quantization (of type  $ \gerC $)  of some subgroup  $ K $  of  $ G \, $.  Then:
\begin{enumerate}
\item  \;  ${\gerC_q}^{\!\!\Lsh}$ is a subalgebra and left coideal in $U_q(\gerg)^\prime$;
\item  \;  if  $\gerC_q $  is  {\sl strict}, then  $ {\gerI_q}^{\! !} $  is strict too,  i.e.~$ \,\; {\gerC_q}^{\!\!\Lsh} \bigcap \, (q-1) \, {U_q(\gerg)}^\prime\, = \; (q-1) \, {\gerC_q}^{\!\!\Lsh} \;\, $;
 \item  \;  there exists a coisotropic subgroup $L$ in $G^*$ such that $ {\gerC_q(K)}^{\!\Lsh} \! = \calC_q(L) \, $:  namely,  $ {\gerC_q(K)}^{\!\Lsh} $  is a  {\sl proper}  quantization, of type  $ \calC \, $, of some coisotropic subgroup  $ L $  of  $ G^* \, $;
\item  \;  in the  {\sl real case},  i.e.~if the quantization  $ \gerC_q $  is a real one,  $ {\gerC_q(K)}^{\!\Lsh} $  is real too, i.e.~$ \, {\big( {\gerC_q}^{\!\!\Lsh} \,\big)}^* \! = {\gerC_q}^{\!\!\Lsh} \, $.  Therefore claims (1--3) still hold in the framework of  {\sl real}  quantum subgroups.
\end{enumerate}
\end{proposition}

\begin{proof}
 The whole proof is very similar to that of Proposition \ref{esclamativo}.
\vskip5pt
   {\it (1)} \,  By definitions,  $1 \in \gerC_q$  and  $\delta_n(1) = 0$  for all  $n \in \N$,  so  $1 \in {\gerC_q}^{\!\!\Lsh} $.  Let  $ \, x, y \in {\gerC_q}^{\!\!\Lsh}
\, $  and  $ \, n \in \N \, $;  by (\ref{deltaone}) we have  $ \delta_n(x y) =
\sum_{\Lambda \cup Y = \{1,\dots,n\}}  \delta_\Lambda(x) \, \delta_Y(y)
\, $.  Each of the factors  $ \, \delta_\Lambda(x) \, $  belongs to a
module  $ \, {(q-1)}^{\vert \Lambda \vert} \, {U_q}^{\otimes ( \vert
\Lambda \vert - 1 )} \! \otimes X \, $  where the last tensor factor
is either  $ \, X = \gerC_q \, $  (if  $ \, n \in \Lambda \, $)  or  $ \,
X = \{1\} \subset \gerC_q \, $  (if  $ \, n \not\in \Lambda \, $),  and
similarly for  $\delta_Y(y)$;  in addition  $ \, \Lambda \cup Y
= \{1,\dots,n\} \, $  implies  $ \, \vert \Lambda \vert + \vert Y \vert
\geq n \, $,  \, and summing up  $ \, \delta_n(x y) \in {(q \! - \! 1)}^n
{U_q}^{\otimes (n-1)} \! \otimes \gerC_q \, $,  whence  $ \, x \, y \in
{\gerC_q}^{\!\!\Lsh}$.  Thus  $ {\gerC_q}^{\!\!\Lsh}$ is a subalgebra of ${U_q}' $.
                                              \par
   In order to prove that  $ \, {\gerC_q}^{\!\!\Lsh}$ is a left coideal in  ${U_q}^{\!\prime}$,  one can again resort to  $ (q-1) $--adic
completions, with exactly the same arguments as in the proof of Proposition
3.5 in \cite{CiGa}, starting from the fact that  $\gerC_q \, \coideal_\ell\, U_q $.  Details are left to the reader.
\vskip5pt
   {\it (2)} \,  Assume, now, that  $ \gerC_q $  is a  {\sl strict\/}  quantization, i.e.
$ \gerC_q \bigcap \, (q-1) \, F_q \, = \, (q-1) \, \gerC_q$.  Then
clearly  ${\gerC_q}^{\!\!\Lsh} \, \bigcap \, (q-1) \, {U_q(\gerg)}'
\, \supseteq \, (q-1) \, {\gerC_q}^{\!\!\Lsh} \, $,  and we must prove
the converse inclusion.  Let  $ \; \kappa \in {\gerC_q}^{\!\!\Lsh} \, \bigcap \, (q-1) \, {U_q(\gerg)}' \, $.  Then:
  $$  \displaylines{
   \delta_n(\kappa) \, \in \, {(q-1)}^n \left( \left( {U_q}^{\otimes (n-1)} \otimes \gerC_q \right) {\textstyle \bigcap} \, (q-1) \, {U_q}^{\otimes n} \right) \; =   \hfill  \cr
   \hfill   = \;  {(q-1)}^n \Big( {U_q}^{\otimes (n-1)} \otimes \big( \gerC_q \, {\textstyle \bigcap} \, (q \! - \! 1) \, U_q \big) \Big)
\, = \; {(q-1)}^{n+1} \cdot {U_q}^{\otimes (n-1)} \otimes \gerC_q
\cr }  $$
which means  $ \, \kappa \in (q-1) \, {\gerC_q}^{\!\!\Lsh} \, $.  Therefore
$ \, {\gerC_q}^{\!\!\Lsh} \, \bigcap \, (q-1) \, {U_q(\gerg)}' \subseteq (q-1) \, {\gerC_q}^{\!\!\Lsh} $, as claimed.
 \vskip5pt
   {\it (3)} \,  The above algebraic properties show that  $ {\gerC_q}^{\!\!\Lsh} $  is a weak quantization of its image  $ \pi_{{U_q}^{\!\prime}} \big( {\gerC_q}^{\!\!\Lsh} \,\big)$; thus there exists a coisotropic subgroup $L$ of $G^*$ such that:  $ \pi_{{U_q}^{\!\prime}} \big( {\gerC_q}^{\!\!\Lsh} \,\big) = \calC(L) $.  Thus  $ {\gerI_q}^{\! !} $  is a weak quantization   --- to be called  $ \calI_q(L) $  ---   of  $ \calI(L) \, $,  and it is even strict if  $ \gerI_q $  itself is strict, by the previous.  Now we show first that this quantization  $ \calI_q(L) $  is always  {\sl proper}   --- hence the subgroup  $ L $  is coisotropic, by Lemma \ref{coisotropic creed}.
 \vskip2pt
  Proving that  $ \calI_q(L) $  is proper amounts to show that $ [x,y] \in (q-1) \, {\gerC_q}^{\!\!\Lsh}$  for all  $ x, y \in {\gerC_q}^{\!\!\Lsh}$.  By definition we have
  $$  [x,y] \, \in \, (q \! - \! 1) \, {\gerC_q}^{\!\!\Lsh}  \;\;
\Longleftrightarrow \;\;  \delta_n \big( [x,y] \big) \, \in
\, {(q \! - \! 1)}^{n+1} {U_q}^{\otimes (n-1)} \otimes
\gerC_q  \quad  \forall\; n \! \in \! \N  $$
and formula (\ref{deltatwo}) gives, for all  $ n \in \N$,
  \begin{equation}\label{quarantadue}
 \delta_n \big([x,y]\big) \, = \, {\textstyle \sum_{\substack{  \Lambda
\cup Y = \{1,\dots,n\}  \\   \Lambda \cap Y \not= \emptyset }}}
\hskip-1pt \big( \, \delta_\Lambda(x) \, \delta_Y(y) \, - \,
\delta_Y(y) \, \delta_\Lambda(x) \, \big)
  \end{equation}
while
  $$  \delta_\Lambda(x) \, \in \, {(q \! - \! 1)}^{\vert \Lambda \vert} \,
j_\Lambda \Big( {U_q}^{\otimes (|\Lambda|-1)} \otimes \gerC_q \Big) \; ,
\quad  \delta_Y(y) \, \in \, {(q \! - \! 1)}^{\vert Y \vert} \,
j_Y \Big( {U_q}^{\otimes (|Y|-1)} \otimes \gerC_q \Big) \; .  $$
Now,   $\Lambda \cup Y = \{1, \dots, n\} $  and  $\Lambda \cap Y \not= \emptyset$
 give  $\vert \Lambda\vert + \vert Y \vert \geq n + 1 $,  and since  $ \gerC_q $
is a subalgebra of  $ U_q $  we get
  $$  \delta_\Lambda(x) \, \delta_Y(y) \, ,  \; \delta_Y(y)
\, \delta_\Lambda(x) \, \in \, {(q \! - \! 1)}^{\vert \Lambda \vert
+ \vert Y \vert} \, {U_q}^{\otimes (n-1)} \otimes \gerC_q \subseteq
{(q \! - \! 1)}^{n+1} {U_q}^{\otimes (n-1)} \otimes \gerC_q  $$
so that (\ref{quarantadue}) yields
  $$  \delta_n \big([x,y] \big) \, \in \, {(q \! - \! 1)}^{n+1} {U_q}^{\otimes (n-1)}\otimes \gerC_q  $$
thus  $ \; [x,y] \in (q-1) \,{\gerC_q}^{\!\!\Lsh} \; $.
\vskip5pt
   {\it (4)} \,  In the real case  $ \, {(\gerC_q)}^* = \gerC_q \, $:  this and the very definitions imply the claim.
\end{proof}

\smallskip

\begin{remark}
 In functorial language we may say that the map $ \, \gerC_q
\mapsto {\gerC_q}^{\!\!\Lsh}$  establishes a functor between quantization of coisotropic subgroups of $G$ and quantizations of Poisson homogeneous spaces of $G^*$, moving from a local to a global description,  sending each type of quantization in a proper one and preserving strictness.  We leave to the interested reader all details which still need to be fixed.

\end{remark}

\vskip9pt

We now move to connectedness properties of the coisotropic subgroup $L$ identified in Propositions \ref{esclamativo} and \ref{Lsh}.

\smallskip

\begin{proposition}\label{connection}
{\ }
                                   \par
\begin{enumerate}
   \item  \;  Let $\,\gerI_q(K)$ \  be a  {\sl strict}  quantization (of type  $ \gerI $) of a (coisotropic) subgroup $K$ in $G\,$.  Then the subgroup  $ L$  of  $ G^* $  such that  $ \, {\gerI_q(K)}^! = \calI_q(L) \, $  is {\sl connected}.
   \item  \;  Let  $ \, \gerC_q(K) $  be a  {\sl strict}  quantization of type  $ \gerC$ of a (coisotropic) subgroup $K$ of $G\,$.  Then the subgroup  $ L $  of  $ G^* $  such that  $ \, {\gerC_q(K)}^! = \calC_q(L) \, $  is {\sl connected}.
\end{enumerate}
\end{proposition}

\begin{proof}
 {\it (1)} \,  Saying that the (closed) subgroup  $ L $  is connected is equivalent to saying that its function algebra  $ \, F[L] = F\big[G^*\big] \Big/ \calI(L) \, $  has no non-trivial idempotents.
 Note that, since  $ F\big[G^*\big] $  is the specialization of  ${U_q}^{\!\prime} $  at  $ q = 1$  and  $ \calI(L) $  is the similar specialization of  ${\gerI_q}^{\! !} \, $,  the quotient  $ F[L] = F \big[ G^* \big] \Big/ \calI(L)$  is canonically isomorphic to the specialization at  $ \, q = 1 \, $  of $ \, {U_q}^{\!\prime} \!\Big/ \, {\gerI_q}^{\! !} \, $.  Let $\overline{a}$ be an idempotent in $F[L]$:  if we take any lift of it  in  ${U_q}^{\!\prime} \! \Big/ \, {\gerI_q}^{\! !} \, $,  \, i.e.~any  $a \in {U_q}^{\!\prime} \! \Big/ \, {\gerI_q}^{\! !} $  such that  $\overline{a} = a \! \mod (q\!-\!1) \, {U_q}^{\!\prime} \! \Big/ \, {\gerI_q}^{\! !} \, $. We must prove:
\begin{equation}
\label{idempotent}
  a^2 \, \equiv \, a \mod (q\!-\!1) \, {U_q}^{\!\prime} \! \Big/ \, {\gerI_q}^{\! !}  \quad \Longrightarrow \quad  a \! \mod (q\!-\!1) \, {U_q}^{\!\prime} \! \Big/ \, {\gerI_q}^{\! !} \, \in \, \big\{ 0, 1 \big\}
\end{equation}
We can clearly reduce to the case when  $\epsilon (\overline{a}) = 0$:  in fact, if  $\overline{a}^{\,2} = \overline{a} $  then  $ \epsilon(\overline{a}) $  is necessarily  $ 0 $  or  $ 1 $  (for it is unipotent too), and in the latter case we then find that  $\overline{a}_0 := 1 - \overline{a}$  is idempotent and  $\epsilon(\overline{a}_0) = 0$.  Also the lift   $a \in {U_q}^{\!\prime} \! \Big/ {\gerI_q}^{\! !} $  can be chosen, in this case, such that:  $\epsilon(a) = 0$.
 To simplify notation, we set  $ H := U_q \! \Big/ \gerI_q $  and  $ H{\text{\bf '}} := {U_q}^{\!\prime} \! \Big/ {\gerI_q}^{\! !} $.  We shall prove that, if  $ a \in H{\text{\bf '}}$,
$\epsilon(a) = 0$  and  $ a^2 \equiv a \mod (q\!-\!1) \, H{\text{\bf '}}$, then  $a \equiv 0 \mod (q\!-\!1) \, H{\text{\bf '}}$,  i.e.~$ a \in (q\!-\!1) \, H{\text{\bf '}} \, $;   in fact, this will give (\ref{idempotent}).
                                            \par
   Having assumed that  $ \gerI_q $  to be strict,  $H{\text{\bf '}}$  identifies with
a  $ \Cqqm $--submodule  of  $H$ given in terms of the coalgebra structure of the latter: the embedding is the one canonically induced by the maps  $ {U_q}^{\!\prime}
\lhook\joinrel\longrightarrow \, U_q \relbar\joinrel\twoheadrightarrow
U_q \Big/ \, \gerI_q $.  In fact, the kernel of the latter map is
$ {U_q}^{\!\prime} \, \bigcap \, \gerI_q $ (by strictness assumption). It is easy to see from definitions that ${U_q}^{\!\prime} \, \bigcap \, \gerI_q = {\gerI_q}^{\! !}$.
Thus  $H{\text{\bf '}}$  does embed into  $H$:
\begin{equation}\label{quarantaquattro}
 H{\text{\bf '}}=
\bigg\{\, \eta \in H \,\;\bigg|\;\, \delta_n(\eta)
\in {(q\!-\!1)}^n \, H^{\! \otimes \,n}\, ,  \; \forall \; n \in \N \;\bigg\}  \, .
\end{equation}
   \indent   Now,  $a^2 \, \equiv \, a \mod (q\!-\!1)H{\text{\bf '}}$  means  $ a = a^2 + (q\!-\!1) \, c$  for some  $ \, c \in H{\text{\bf '}} $;  since  $\epsilon(a) =0$, we have  $\epsilon(c) = 0$  as well.  Applying $\delta_n $  to the identity  $ a = a^2 + (q\!-\!1) \, c$  and using formula (\ref{deltaone}) we get
  \[
\delta_n(a)  \; = \;  \delta_n\big(a^2\,\big) \, + \, (q\!-\!1) \,
\delta_n(c)  \; = \hskip-1pt  {\textstyle \sum\limits_{\Lambda \cup Y
= \{1,\dots,n\}}} \hskip-11pt  \delta_\Lambda(a) \; \delta_Y(a) \; +
\; (q-1) \; \delta_n(c)
  \]
for all  $n \in \N$, which   --- noting that  $\delta_0(a) := \epsilon(a) = 0$
yields:
\begin{equation}\label{quarantacinque}
 \delta_n(a)  \;\; =  {\textstyle \sum\limits_{\substack{\Lambda \cup Y = \{1,\dots,n\} \\  \Lambda, Y \not= \emptyset}}} \hskip-3pt  \delta_\Lambda(a) \; \delta_Y(a) \; + \; (q-1) \;
\delta_n(c)
\end{equation}
Since  $c \in H{\text{\bf '}} \, $, the last summand  $(q\!-\!1) \, \delta_n(c)  $  in right-hand side of (\ref{quarantacinque}) belongs to  $ \, (q\!-\!1)^{n+1} \, H^{\otimes n}$, thanks to (\ref{quarantaquattro}).
Similarly, since  $a \in H{\text{\bf '}}$  we have  $ \delta_k(a) \in (q\!-\!1)^k \, H^{\otimes k}$  for all  $k \in \N$,  by (\ref{quarantaquattro}) again: therefore each summand
$\delta_\Lambda(a) \; \delta_Y(a)$  in right-hand side of (\ref{quarantacinque}) belongs to
$ (q-1)^{n+1} \, H^{\otimes n} $  as well.  But then (\ref {quarantacinque}) yields
$ \delta_n(a) \, \in \, {(q-1)}^{n+1} H^{\otimes n}$  for all  $n \in \N$, which, again by (\ref{quarantaquattro}), means exactly that  $ \, a \in (q\!-\!1) \, H{\text{\bf '}}$. This ends the proof of the first claim.
 \vskip5pt
   {\it (2)} \,  We will use similar arguments to show this claim:  $ F[L] = F\big[G^*\big] \Big/ \calI(L)$  has no non-trivial idempotents.  Since  ${\gerC_q}^{\!\!\Lsh} = \calC_q(L)$  and  $\calC(L) = \calC \big(\widehat{L}\,\big)$, we can  assume  $L = \widehat{L}$,  i.e.~$ L $  is observable.  This implies  $\calI(L) = \Psi\big(\calC(L)\big)$, which is clearly the specialization at  $q = 1$  of
$\Psi\big(\calC(L)\big) = {U_q}^{\!\prime} \, {\gerC_q}^{\!\!\Lsh}$;  therefore,
$ F[L] = F\big[G^*\big] \Big/ \calI(L) \, $  is canonically isomorphic to the specialization at
$q = 1$  of  $ {U_q}^{\!\prime} \! \Big/ \, {U_q}^{\!\prime} \, {\gerC_q}^{\!\!\Lsh} \; $.
                                                      \par
   From now on, one can mimic step by step the proof of  part  {\it (1)}.
The only detail to modify is that one must take  $U_q \, {\gerC_q}^{\!+} =:
\Psi(\gerC_q)$  in place of  $ \gerI_q\, $,  \, and  ${U_q}^{\!\prime} \, \big({\gerC_q}^{\!\!\Lsh}\big)^{\!+} =: \Psi \big( {\gerC_q}^{\!\!\Lsh} \big)$ in place of  $ {\gerI_q}^{\! !} \, $.  Letting  $H := U_q \Big/ \Psi(\gerC_q) \, $, and
$ \, H{\text{\bf '}} := {U_q}^{\!\prime} \! \Big/ \Psi \big( {\gerC_q}^{\!\!\Lsh}
\big) \, $, the thesis amounts to prove that
\[
a \in H{\text{\bf '}}\, ,\quad a^2 \equiv a \mod (q\!-\!1) \, H{\text{\bf '}} \Rightarrow
 a \equiv 0 \mod (q\!-\!1) \, H{\text{\bf '}}
\]
(In fact  also $a \equiv 1 \mod (q\!-\!1) \, H{\text{\bf '}}$ would be ok, but, arguing as before, we'll restrict to the case $\epsilon(a) = 0$).
                                            \par
   As  $ \gerC_q $  is strict, it is easy to see from definitions that
${\gerC_q}^{\!\!\Lsh} = {U_q}^{\!\prime} \, \bigcap \, \gerC_q$, hence  $\Psi \big( {\gerC_q}^{\!\!\Lsh} \big) := {U_q}^{\!\prime} \, \big({\gerC_q}^{\!\!\Lsh}\big)^{\!+} = {U_q}^{\!\prime} \big( {U_q}^{\!\prime}
\cap \gerC_q \big)^{\!+} $: the latter is the kernel of the map
$ {U_q}^{\!\prime} \lhook\joinrel\longrightarrow U_q
\relbar\joinrel\twoheadrightarrow U_q \Big/ U_q \, \gerC_q^{\;+}$,  so
$H{\text{\bf '}}$  embeds as a  $ \Cqqm $--submodule  of  $ H$, namely
\[
 H{\text{\bf '}}= \bigg\{\, \eta \in H \;
\bigg| \, \delta_n(\eta) \in {(q\!-\!1)}^n \, H^{\! \otimes \,n} \, ,  \; \forall \; n \in
\N \,\bigg\}  \, .
\]
With this description at hand, computations are as in the proof of claim  {\it (1)}.
\end{proof}

\vskip7pt

 Our next results are about the behavior of quantum subgroups under composition of Drinfeld-like maps.

\vskip13pt

\begin{proposition}\label{reciprocity}
 Let  $\,\calI_q\,$,  $ \calC_q \, $,  $ \gerI_q \, $, $ \gerC_q \, $  be weak quantizations of a subgroup $K$ of $G\,$.  Then:
\begin{enumerate}
\item  \;\qquad  $ \calI_q \subseteq \big( {\calI_q}^{\!\curlyvee} \big)^! \;\; ,  \quad  \calC_q \subseteq \big( {\calC_q}^{\!\!\triangledown} \,\big)^{\!\Lsh} \;\; $;
\item  \;\qquad  $ \gerC_q \supseteq \big( {\gerC_q}^{\!\!\Lsh} \big)^{\!\triangledown} \;\; ,  \quad  \gerI_q \supseteq \big( {\gerI_q}^{\! !} \,\big)^{\! \curlyvee} \;\; $.
\end{enumerate}
 \eject
\end{proposition}

\begin{proof}
 {\it (1)} \,  By the very definitions, for any  $n \in \N$  we have
  \[  \displaylines{
   \delta_n\big(\calI_q\big) \, \subseteq \, {J_{\scriptscriptstyle F_q}}^{\!\!\otimes n} \,{\textstyle \bigcap}\, \Big( {\textstyle \sum_{s=0}^n} \, {F_q}^{\!\otimes s} \otimes \calI_q \otimes {F_q}^{\!\otimes (n-s-1)} \Big) = {\textstyle \sum_{s=0}^n} \, {J_{\scriptscriptstyle F_q}}^{\!\!\otimes s} \otimes \calI_q \otimes {J_{\scriptscriptstyle F_q}}^{\!\!\otimes (n-s-1)} \subseteq   \hfill  \cr
   \hfill   \subseteq \, {(q-1)}^n \cdot {\textstyle \sum_{s=0}^n} \, {\big( {F_q}^{\!\vee} \big)}^{\otimes s} \! \otimes {\calI_q}^{\!\curlyvee} \!\otimes {\big( {F_q}^{\!\vee} \big)}^{\otimes (n-s-1)}  \cr }
\]
which means exactly  $ \, \calI_q \subseteq {\big( {\calI_q}^{\!\curlyvee} \big)}^! \; $.  Similarly we can remark that:
 \[
  \delta_n\big(\calC_q\big) \, \subseteq \, {J_{\scriptscriptstyle F_q}}^{\!\!\otimes n} \,{\textstyle \bigcap}\, \big( {F_q}^{\!\otimes (n-1)} \otimes \calC_q \big) \, = \, {J_{\scriptscriptstyle F_q}}^{\!\!\otimes (n-1)} \otimes \big( \calC_q \,{\textstyle \bigcap}\, J_{\scriptscriptstyle F_q} \big) \, \subseteq \, {(q \! - \! 1)}^n {\big( {F_q}^{\!\vee} \big)}^{\!\otimes (n-1)} \!
\otimes {\calC_q}^{\!\!\triangledown}
 \]
which means  $\calC_q \subseteq \big( {\calC_q}^{\!\!\triangledown} \big)^{\!\Lsh}$.  Therefore claim {\it (1)\/}  is proved.
 \vskip5pt
   {\it (2)} \,  As  $ \big( {\gerC_q}^{\!\!\Lsh} \big)^{\!\triangledown} $     is generated   --- as an algebra ---   by  $ \, {(q \! - \! 1)}^{-1} {\gerC_q}^{\!\!\Lsh} \, {\textstyle \bigcap} \, J_{\scriptscriptstyle {U_q}^{\!\prime}}$,  it is enough to show that the latter space is contained in  $ \gerC_q \, $.  Let, then,  $x' \in {\gerC_q}^{\!\!\Lsh} \, {\textstyle \bigcap} \, J_{\scriptscriptstyle {U_q}^{\!\prime}} \, $. Surely  $\delta_1\big(x'\big) \in (q-1) \, \gerC_q$,  hence  $x' = \delta_1\big(x'\big) + \epsilon\big(x'\big) \in (q-1) \, \gerC_q$.  Therefore
$ {(q-1)}^{-1} x' \in \gerC_q \, $,  q.e.d.  Similarly,  $ \big( {\gerI_q}^{\! !} \big)^{\! \curlyvee}$  is the left ideal of  $ {U_q}^{\!\prime} $  generated by  ${(q \! - \! 1)}^{-1} \, {\gerI_q}^{\! !} \, {\textstyle \bigcap} \, J_{\scriptscriptstyle {U_q}^{\!\prime}}$, thus   --- since  $ \, {U_q}^{\!\prime} \subseteq U_q \, $  ---   we must only prove that  ${(q \! - \! 1)}^{-1} \, {\gerI_q}^{\! !} \, {\textstyle \bigcap} \, J_{\scriptscriptstyle {U_q}^{\!\prime}}$  is contained in  $ U_q$.  Again, if
$ y' \in {\gerI_q}^{\! !} \, {\textstyle \bigcap} \, J_{\scriptscriptstyle {U_q}^{\!\prime}} \, $  then  $ \, y'
= \delta_1\big(y'\big) + \epsilon\big(y'\big) \in (q-1) \, \gerI_q $.  Thus we get  $ \, {(q-1)}^{-1} y' \in \gerI_q \, $,  and  {\it (2)\/}  is proved.
\end{proof}

\vskip7pt

\noindent {\bf Remarks:}
\begin{itemize}
\item[(a)] \,  By repeated applications of the previous proposition it is easily proved that:
\[
 {\calI_q}^{\!\curlyvee} = \Big(\! \big( {\calI_q}^{\!\curlyvee} \big)^! \Big)^{\!\!\curlyvee} \;\; ,  \quad  {\calC_q}^{\!\!\triangledown} = \Big(\! \big( {\calC_q}^{\!\!\triangledown} \big)^{\!\Lsh} \Big)^{\!\!\triangledown} \;\; ,  \quad  {\gerC_q}^{\!\!\Lsh} = \Big(\! \big( {\gerC_q}^{\!\!\Lsh} \,\big)^{\!\triangledown} \Big)^{\!\!\Lsh} \;\; ,  \quad  {\gerI_q}^{\! !} = \Big(\! \big( {\gerI_q}^{\! !} \,\big)^{\! \curlyvee} \Big)^{\! !}
\]
\item[(b)] \,  Since  we proved  that Drinfeld-like maps always produce  {\sl proper\/}  quantizations, and that proper quantizations specialize to  {\sl coisotropic\/}  subgroups  (cf.~Proposition \ref{coisotropic creed}),  the following holds:
\begin{enumerate}
   \item  \;  if  $ \, \calI_q = \big( \calI_q^{\,\curlyvee} \big)^! \, $ then $\calI_q$ is a proper quantization (of type $\calI$) of a coisotropic subgroup of $G\,$;
   \item  \;  if $ \, \calC_q = {\big( \calC_q^{\,\triangledown} \big)}^{\!\Lsh} \, $  then $\calC_q$ is a proper quantization (of type $\calC$) of a coisotropic subgroup of $G\,$;
   \item  \;  if $ \, \gerI_q = {\big( \gerI_q^{\;!} \big)}^{\!\curlyvee} \, $  then $\gerI_q$ is a proper quantization (of type $\gerI$) of a coisotropic subgroup of $G\,$;
   \item  \;  if $ \, \gerC_q = \big( \gerC_q^{\,\Lsh\,} \big)^{\!\triangledown} \, $  then $\gerC_q$ is a proper quantization (of type $\gerC$) of a coisotropic subgroup of $G\,$.
\end{enumerate}
\item[(c)] \,  Since the whole construction is independent of the existence of real structures  all the above claims hold true in the  {\sl real}  framework as well.
\end{itemize}

\vskip7pt

   Next result reads as a converse of the previous one, holding for Drinfeld maps applied to  {\sl strict\/}  quantizations:

\vskip17pt

\begin{theorem}\label{converse}
 {\ }
                                         \par
\begin{itemize}
  \item[(a)] \;  if $\,\calI_q$ is a strict quantization of a coisotropic subgroup of $G$ then $ \, \calI_q = \big( \calI_q^{\,\curlyvee} \big)^! \, $;
  \item[(b)] \;  if $\,\calC_q$ is a strict quantization of a coisotropic subgroup of $G$ then $\,\calC_q = {\big( \calC_q^{\,\triangledown} \big)}^{\!\Lsh}\,$;
  \item[(c)] \;  if $\,\gerI_q$   is a strict quantization of a coisotropic subgroup of $G$  then $\,\gerI_q = {\big( \gerI_q^{\;!} \big)}^{\!\curlyvee}\,$;
  \item[(d)] \;  if $\,\gerC_q$  is a strict quantization of a coisotropic subgroup of $G$  then $\,\gerC_q = \big( \gerC_q^{\,\Lsh\,} \big)^{\!\triangledown}\,$;
  \item[(e)] \;  The above claims hold true in the  {\sl real}  framework as well.
\end{itemize}
\end{theorem}

\begin{proof}
 {\it (a)} \,  Let  $\calI_q$ be a strict quantization;   by  Proposition \ref{reciprocity}{\it (1)},  it is enough to prove  $ \, \calI_q \supseteq \big( \calI_q^{\,\curlyvee} \big)^! \, $.  For this we apply the argument used in \cite{Ga3}, Proposition 4.3, to prove that  $ \, F_q \supseteq {\big({F_q}^{\!\vee}\big)}' \, $.
                                                   \par
   We denote by  $L$  the closed, coisotropic, connected subgroup of  $ G^* $  such that
$ \calI_q^{\,\curlyvee} = \gerI_q(L) \, $, as in Proposition \ref{curlyvee}, and with $\gerl $  its Lie algebra.
                                                   \par
   Let  $ \, y' \in \big( \calI_q^{\,\curlyvee} \big)^! \, $.  Then there is  $ \, n \in \N \, $  and  $ \, y^\vee \in \calI_q^{\,\curlyvee} \setminus (q-1) \, \calI_q^{\,\curlyvee} \, $  such that  $ \, y' = {(q-1)}^n y^\vee \, $.
As we have seen strictness of $\calI_q$ implies strictness of $\calI_q^{\,\curlyvee}$ and therefore
$ \, y^\vee \not\in (q-1) \, {F_q}^{\!\vee} \, $,  and so for  $ \, \overline{y^\vee} := y^\vee \! \mod (q-1) \, {F_q}^{\!\vee}$  we have  $ \, \overline{y^\vee} \not= 0 \in {F_q}^{\!\vee}{\Big|}_{q=1} \!\! = \, U\big(\gerg^*\big) \, $.
                                                   \par
   As  $ {F_q}^{\!\vee} $  is a quantization of  $ U \big( \gerg^* \big)$, we can pick an ordered basis
$ \, {\{b_\lambda\}}_{\lambda \in \Lambda} \, $  of  $ \gerg^*$, and a subset  $ \, {\big\{ x^\vee_\lambda \big\}}_{\lambda \in \Lambda} \, $  of  $ \, {(q-1)}^{-1} J_{\scriptscriptstyle F_q} \, $  so that  $ \, x^\vee_\lambda  \hskip-1,4pt  \mod (q-1) \, {F_q}^{\!\vee} = b_\lambda \, $  for all  $\lambda \in \Lambda\,$;
therefore  $ \, x^\vee_\lambda = {(q-1)}^{-1} x_\lambda \, $  for some  $ \, x_\lambda \in J_{\scriptscriptstyle F_q} \, $,
for all  $ \lambda$  (like in the proof of  \cite{Ga3} Proposition 4.3).  In addition, we choose now the basis and its lift so that a subset  $ \, {\{b_\theta\}}_{\theta \in \Theta} \, $   (for some suitable  $ \, \Theta \subseteq \Lambda \, $)  is a basis of  $ \gerl \, $,  and, correspondingly,  ${\big\{ x^\vee_\theta \big\}}_{\theta \in \Theta} \subseteq \calI_q^{\,\curlyvee} \, $.  Since
$\overline{y^\vee} \not= 0 \in {F_q}^{\!\vee}{\big|}_{q=1} \!\! = \, U\big(\gerg^*\big) \, $,  by the Poincar\'e-Birkhoff-Witt theorem there is a non-zero polynomial  $ \, P \big( {\{ b_\theta \}}_{\theta \in \Theta} \big) \, $  in the  $ b_\theta $'s  such that  $ \; \overline{y^\vee} = P \big( {\{ b_\theta \}}_{\theta \in \Theta} \big) \, $,  hence
  \[
  y^\vee - P \big( {\big\{ x^\vee_\theta \big\}}_{\theta \in \Theta} \big)  \; \in \;  \calI_q^{\,\curlyvee} \,{\textstyle \bigcap}\; (q-1) \, {F_q}^{\!\vee}  \, = \,  (q-1) \; \calI_q^{\,\curlyvee}  \quad .
\]
This implies  $y^\vee = P \big( {\big\{ x^\vee_\theta \big\}}_{\theta \in \Theta} \big) + {(q-1)}^\nu \, y^\vee_1 \,$  for some  $\nu \in \N_+ $  where  $y^\vee_1 \in \calI_q^{\,\curlyvee} \setminus (q-1) \, \calI_q^{\,\curlyvee}$.
                                                   \par
   One can see, like in \cite{EK}, Lemma 4.12,   that the polynomial  $ P $  has degree not greater than  $ n \, $.  Thus
$y' = {(q \! - \! 1)}^n y^\vee = {(q \! - \! 1)}^n P \big( {\big\{ x^\vee_\theta \big\}}_{\theta \in \Theta} \big) + {(q \! - \! 1)}^{n+\nu} y^\vee_1 \, $,  and
\[
  {(q \! - \! 1)}^n P \big( \big\{ x^\vee_\theta \big\}_{\theta \in \Theta} \big) = {(q \! - \! 1)}^n P \big( {\big\{ {(q \! - \! 1)}^{-1} \, x_\theta \big\}}_{\theta \in \Theta} \big) \in \calI_q
\]
by a degree argument.  But now, Proposition \ref{reciprocity} gives  $ \, \calI_q \subseteq \big( \calI_q^{\,\curlyvee} \big)^!$.  Then
  \[
  y'_1 \, := \, y' - {(q \! - \! 1)}^n P \big( {\big\{ x^\vee_\mu \big\}}_{\theta \in \Theta} \big) \in \big( \calI_q^{\,\curlyvee} \big)^!  \;\quad  \text{and}  \quad\;  y'_1 \, = \, {(q \! - \! 1)}^{n+\nu} y^\vee_1 = {(q \! - \! 1)}^{n_1} y^\vee_1
\]
where  $n_1 := n + \nu > n$,  and  $y^\vee_1 \in \calI_q^{\,\curlyvee} \setminus (q-1) \, \calI_q^{\,\curlyvee}$.  We can then repeat the construction, with  $ y'_1 $  instead of  $ y' $,
 $ n_1 $  instead of  $ n$,  etc.: iterating, we find an increasing sequence of numbers ${\big\{ n_s \big\}}_{s \in \N}$  (with  $n_0 := n $)  and a sequence of polynomials  ${\big\{ P_s \big( {\{ X_\theta \}}_{\theta \in \Theta} \big) \big\}}_{s \in \N}$  (again  $P_0 := P$)  such that the degree of $ P_s \big( {\{ X_\theta \}}_{\theta \in \Theta} \big) $  is at most
$ n_s$,  and the formal identity  $ \, y'=\sum_{s \in \N} {(q-1)}^{n_s} P_s \big( {\big\{ x^\vee_\theta \big\}}_{\theta \in \Theta} \big) \, $  holds.
                                                   \par
   Now set  $I_n := \sum_{k=1}^n {(q-1)}^{n-k} \, {\calI_q}^{\!k}$  (for all  $ n \in \N$),  and let  $\widehat{\calI}_q$  be the topological completion of  $\calI_q$  with respect to
the filtration provided by the  $ I_n $'s.  Then, by construction,
${(q \! - \! 1)}^{n_s} P_s \big( {\big\{ x^\vee_\theta \big\}}_{\theta
\in \Theta} \big) \in I_n \, $  for all  $ s \! \in \! \N \, $. This yields
  $$  {\textstyle \sum\limits_{s \in \N}} {(q \! - \! 1)}^{n_s}
P_s \big( {\big\{ x^\vee_\theta \big\}}_{\theta \in \Theta} \big)
\, \in \, \widehat{\calI}_q  \qquad  \text{and}  \qquad
 y' = {\textstyle \sum\limits_{s \in \N}} {(q \! - \! 1)}^{n_s} P_s \big({\big\{ x^\vee_\theta \big\}}_{\theta \in \Theta} \big)  $$
where the last is an identity in  $ \widehat{\calI}_q \, $. Thus  $ \, y' \in \big(
\calI_q^{\,\curlyvee} \big)^! \bigcap \, \widehat{\calI}_q \, $.
Again with the same arguments as in \cite{Ga3}, we see that  $ \; \calI_q
\bigcap {(q \! - \! 1)}^\ell \, \widehat{\calI}_q = {(q \! - \! 1)}^\ell
\, \calI_q \; $  for any  $\ell \in \N\,$.  This together with
$ \; y' \in \big( \calI_q^{\,\curlyvee} \big)^! \bigcap \,
\widehat{\calI}_q \; $  give  $ \, y' = {(q-1)}^{-m} \eta \, $
for some  $ m \in \N$  and  $ \, \eta \in \calI_q \,$;
thus
\[
 \eta = {(q \! - \! 1)}^m \, y' \in \calI_q \, {\textstyle \bigcap} \,
{(q \! - \! 1)}^m \, \widehat{\calI}_q = {(q \! - \! 1)}^m \, \calI_q\, ,
\]
 whence  $y' \in \calI_q$,  q.e.d.

\vskip5pt

   {\it (b)} \,  Assume that  $ \calC_q $  is a strict quantization;
by  Proposition \ref{reciprocity}{\it (2)},  it is enough to prove  $ \, \calC_q \supseteq {\big( \calC_q^{\,\triangledown} \big)}^{\!\Lsh}$.  To do that, we resume the argument used in \cite{Ga3}, Proposition 4.3, to show that  $F_q \supseteq {\big({F_q}^{\!\vee}\big)}' $.
                                                   \par
   We denote by  $L$  the closed, coisotropic, connected subgroup of  $ G^* $  such that
$ \calC_q^{\,\triangledown} = \gerC_q(L)$  and with $ \gerl $  its  Lie algebra.
                                                   \par
   Let  $c' \in {\big( \calC_q^{\,\triangledown} \big)}^{\!\Lsh}$.  Then there exist  $n \in \N$  and  $c^\vee\in \calC_q^{\,\triangledown} \setminus (q-1) \, \calC_q^{\triangledown} $  such that
$c' = {(q-1)}^n c^\vee$.  Note that  strictness of $\calC_q$ implies strictness of
$\calC_q^{\, \triangledown}$;  hence
$ c^\vee \not\in (q-1) \, {F_q}^{\!\vee} $, so that for  $\overline{c^\vee} := c^\vee \! \mod (q-1) \, {F_q}^{\!\vee}$  we have  $ \overline{c^\vee} \not= 0 \in {F_q}^{\!\vee}{\big|}_{q=1} \!\! = \, U\big(\gerg^*\big)$.  Moreover,  $ \; \overline{c^\vee} \in \calC_q^{\,\triangledown}{\big|}_{q=1} \! = \gerC(L) = U(\gerl) \subseteq U\big(\gerg^*\big)$.
                                                   \par
   Since  $ {F_q}^{\!\vee} $  is a quantization of  $ U \big( \gerg^* \big)$, we can fix an ordered basis  $ \, {\{ b_\lambda \}}_{\lambda \in \Lambda} \, $  of  $ \gerg^* \, $,  \, and a subset
$ {\big\{ x^\vee_\lambda \big\}}_{\lambda \in \Lambda} $  of
$ {(q-1)}^{-1} J_{\scriptscriptstyle F_q} $  such that  \hbox{$ x^\vee_\lambda  \hskip-1,4pt  \mod (q-1) \, {F_q}^{\!\vee} = b_\lambda$}  for all  $\lambda \in \Lambda$;  so
$ x^\vee_\lambda = {(q-1)}^{-1} x_\lambda$  for some  $ x_\lambda \in J_{\scriptscriptstyle F_q}$,  for all  $ \lambda$  (as  in the proof of  \cite{Ga3} Proposition 4.3). We can choose both the basis and its lift so that a subset  $ {\{b_\mu\}}_{\mu \in M} $      is a basis of  $ \gerl  $ (here $M \subseteq \Lambda $),  and, correspondingly,
$ {\big\{ x^\vee_\mu \big\}}_{\mu \in M} \subseteq {(q \! - \! 1)}^{-1} J_{\scriptscriptstyle F_q} \bigcap \, \calC_q^{\,\triangledown}$.  Since  $ \overline{c^\vee} \not= 0 \in {F_q}^{\!\vee}{\Big|}_{q=1} \!\! = \, U\big(\gerg^*\big) \, $, by the Poincar\'e-Birkhoff-Witt theorem there exists a non-zero polynomial  $ P \big( {\{ b_\mu \}}_{\mu \in M} \big)$  in variables
$ b_\mu $'s  such that $\overline{c^\vee} = P \big( {\{b_\mu\}}_{\mu \in M} \big)$, hence:
  \[
c^\vee - P \big( {\big\{ x^\vee_\mu \big\}}_{\mu \in M} \big)
\; \in \;  \calC_q^{\,\triangledown} \,{\textstyle \bigcap}\;
(q-1) \, {F_q}^{\!\vee}  \, = \,  (q-1) \;
\calC_q^{\,\triangledown}  \, .
 \]
Therefore,  $ c^\vee = P \big( {\big\{ x^\vee_\mu \big\}}_{\mu \in M} \big) + {(q-1)}^\nu \, c^\vee_1$  for some  $\nu \in \N_+$  where  $c^\vee_1 \in \calC_q^{\,\triangledown} \setminus (q-1) \, \calC_q^{\,\triangledown}$.
                                                   \par
   Now, we can see   --- like in \cite{EK}, Lemma 4.12 ---   that the degree of $P$  is not greater than $ n \, $.  Then
 \[
 c' = {(q-1)}^n c^\vee = {(q-1)}^n P \big( {\big\{ x^\vee_\mu \big\}}_{\mu \in M} \big) + {(q-1)}^{n+\nu} c^\vee_1
 \]
 with  $ {(q-1)}^n P \big( \big\{ x^\vee_\mu \big\}_{\mu \in M} \big) = {(q-1)}^n P \big( {\big\{ {(q-1)}^{-1} \, x_\mu \big\}}_{\mu \in M} \big) \in \calC_q$  because  $ P $  has degree boun\-ded (from above) by  $ n$.  As  $\calC_q \subseteq {\big( \calC_q^{\,\triangledown} \big)}^{\!\Lsh}$,  by  Proposition \ref{reciprocity},  we get
  \[
 c'_1 :=c' - {(q \! - \! 1)}^n P \big( {\big\{ x^\vee_\mu \big\}}_{\mu \in M} \big) \in {\big( \calC_q^{\,\triangledown} \big)}^{\!\Lsh}  \quad  \text{and}  \quad  c'_1 = {(q \! - \! 1)}^{n+\nu} c^\vee_1 = {(q \! - \! 1)}^{n_1} c^\vee_1
\]
with  $ n_1 := n + \nu > n$,  \, and  $c^\vee_1 \in \calC_q^{\,\triangledown} \setminus (q-1) \, \calC_q^{\,\triangledown}$.  We can repeat this construction with  $ c'_1 $  in place of  $ c' $,
$ n_1 $ in place of  $ n $,  etc.. Iterating, we get an increasing
sequence  of numbers ${\big\{ n_s \big\}}_{s \in \N}$  ($ n_0 \! := \! n \, $)  and a sequence of polynomials
 ${\big\{ P_s \big( {\{ X_\mu \}}_{\mu \in M} \big) \big\}}_{s \in \N}$  ($ P_0 \! := \! P$)  such that the degree of
$ P_s \big( {\{ X_\mu \}}_{\mu \in M} \big) $  is at most  $ n_s$,  and
$c' = \sum_{s \in \N} {(q-1)}^{n_s} P_s \big( {\big\{ x^\vee_\mu \big\}}_{\mu \in M} \big)$.
                                                   \par
   Consider
  $$  I_{\calC_q} := \Ker\,\Big( \calC_q \, \stackrel{\epsilon}{\relbar\joinrel\twoheadrightarrow} \, \Cqqm \,\stackrel{ev_1}{\relbar\joinrel\twoheadrightarrow }\, \C \Big) = \Ker\,\Big( \calC_q \,\stackrel{ev_1}{\relbar\joinrel\twoheadrightarrow}\, \calC_q \big/ (q-1) \, \calC_q \,\stackrel{\bar{\epsilon}}{\relbar\joinrel\twoheadrightarrow}\, \C \Big)  $$
   By construction, we have  $ \, {(q-1)}^{n_s} P_s \big( {\big\{ x^\vee_\mu \big\}}_{\mu \in M} \big) \in {I_{\calC_q}}^{\hskip-3pt n_s} \, $  for all  $ \, s \in \N \, $;  in turn, this means that  $ \; \sum_{s \in \N} {(q-1)}^{n_s} P_s \big( {\big\{ x^\vee_\mu \big\}}_{\mu \in M} \big) \in \widehat{\calC}_q \, $,  the latter being the  $ I_{\calC_q} $--adic  completion of  $ \calC_q \, $,  and  the formal expression  $ \, c' = \sum_{s \in \N} {(q-1)}^{n_s} P_s \big( {\big\{ x^\vee_\mu \big\}}_{\mu \in M} \big) \, $  is an identity in  $ \widehat{\calC}_q \, $:  therefore  $c' \in {\big( \calC_q^{\,\triangledown} \big)}^{\!\Lsh} \bigcap \, \widehat{\calC}_q $.  Acting as in \cite{Ga3}, again, we see that
$ \calC_q \bigcap {(q-1)}^\ell \, \widehat{\calC}_q = {(q-1)}^\ell \, \calC_q \, $  for all  $ \, \ell \in \N$.  Getting back
to  $c' \in {\big( \calC_q^{\,\triangledown} \big)}^{\!\Lsh} \bigcap \, \widehat{\calC}_q$, we have  $c' = {(q-1)}^{-m} \kappa$  for some  $m \in \N$  and  $\kappa \in \calC_q \; $; thus  $\kappa = {(q \! - \! 1)}^m \, c' \! \in \calC_q \bigcap {(q \! - \! 1)}^m \, \widehat{\calC}_q = {(q \! - \! 1)}^m \, \calC_q$, whence  $c' \in \calC_q$,  q.e.d.

\vskip5pt

   {\it (c)}  Let $\gerI_q$ be a strict quantization:   by Proposition \ref{reciprocity}{\it (2)\/}  it is enough to prove  $\gerI_q \subseteq {\big( \gerI_q^{\;!} \big)}^{\!\curlyvee}$; so given  $y \in \gerI_q $,  we must prove that  $y \in {\big( \gerI_q^{\;!} \big)}^{\!\curlyvee}$.  Recall that  $\gerI_q \subseteq U_q = \big( {U_q}^{\!\prime} \big)^\vee$, the last identity following from  Theorem \ref{GQDP}.  By construction,
  \[
\big( {U_q}^{\!\prime} \big)^\vee  \, = \;  {\textstyle \sum_{n \geq 0}} \, {(q-1)}^{-n} \, I_{\scriptscriptstyle {U_q}^{\!\prime}}^{\;n} \;\; ,  \qquad  I_{\scriptscriptstyle {U_q}^{\!\prime}} \, := \, \big( {U_q}^{\!\prime} \big)^+ + (q-1) \, {U_q}^{\!\prime}
\]
so for  $y \in \gerI_q \subseteq U_q = \big( {U_q}^{\!\prime} \big)^\vee$  there exists  $N \in \N$  such that
\begin{equation}\label{quarantasei}
 y_+  \; := \;  {(q-1)}^N \, y  \; \in \;  I_{\scriptscriptstyle {U_q}^{\!\prime}}^{\;N}  \; \subseteq \;  {U_q}^{\!\prime}
\end{equation}

Strictness of $ \gerI_q \,$,  i.e.~$ \, \gerI_q \bigcap \, (q-1) \, U_q = (q-1) \, \gerI_q \, $,  implies
  $$  \Big(\, {\textstyle \sum\limits_{s=1}^n} {U_q}^{\!\otimes (s-1)} \otimes \gerI_q \otimes {U_q}^{\!\otimes (n-s)} \Big) \,{\textstyle \bigcap}\, \big( {(q \! - \! 1)}^n \, {U_q}^{\!\otimes n} \big) = {(q \! - \! 1)}^n \, \Big(\, {\textstyle \sum\limits_{s=1}^n} {U_q}^{\!\otimes (s-1)} \otimes \gerI_q \otimes {U_q}^{\!\otimes (n-s)} \Big)  $$
for all  $n \in \N_+$; then, by the very definitions, the latter yields  $ \; \gerI_q^{\;!} = \gerI_q \bigcap {U_q}^{\!\prime} \; $.
                                                                             \par
\noindent   If   in (\ref{quarantasei}) $ N = 1$, then  $y_+ = y \in {U_q}^{\!\prime}$,  thus  $ y \in \gerI_q \bigcap {U_q}^{\!\prime} = \gerI_q^{\;!} \, $,  q.e.d. If, instead,  $N > 1$ , then formula (\ref{quarantasei}), along with
$\gerI_q \,\coideal\, U_q $,  yields:
  \begin{equation}\label{quarantasette}
  \delta_n(y_+) \in \Big( {(q-1)}^N \cdot {\textstyle \sum\limits_{s=1}^n} {U_q}^{\!\otimes (s-1)} \otimes \gerI_q \otimes {U_q}^{\!\otimes (n-s)} \Big) \,{\textstyle \bigcap}\, \big( {(q-1)}^n \, {U_q}^{\!\otimes n} \big) \, ,  \quad  \forall \,n \in \N_+
\end{equation}
and since  $ \gerI_q $  is strict, from (\ref{quarantasette}) one gets
\[
\delta_n(y_+) \in {(q \! - \! 1)}^n \, {\textstyle \sum\limits_{s=1}^n} {U_q}^{\!\otimes (s-1)} \otimes \gerI_q \otimes {U_q}^{\!\otimes (n-s)}\qquad \forall\, n\in\N
\]
%
%
which means  $y_+ \in \gerI_q^{\;!}$.  Eventually, we have found  $ \; y_+ \in \gerI_q^{\;!} \bigcap I_{\scriptscriptstyle {U_q}^{\!\prime}}^{\;N} \; $.
                                                                             \par
   Now look at  $ \, I_{\scriptscriptstyle \gerI_{\!q}^{\,!}} := I_{\scriptscriptstyle {U_q}^{\!\prime}} \,{\textstyle \bigcap}\, \gerI_q^{\;!} \, $.  Using the fact that  $ \, {U_q}^{\!\prime} = {U_q(\gerg)}' = F\big[G^*\big] \, $   --- from Theorem \ref{GQDP}  ---   and  $ \, \gerI_q^{\;!} = {\gerI_q(K)}^! = \calI_q(L) \, $  for some  coisotropic subgroup $L$ in $G^*$  --- as granted by Proposition \ref{esclamativo}  ---   and still taking into account strictness, by an easy geometrical argument (via specialization at  $q = 1$)  we see that
  $$  I_{\scriptscriptstyle {U_q}^{\!\prime}}^{\;n} \;{\textstyle \bigcap}\; \gerI_q^{\;!}  \, \equiv \,  I_{\scriptscriptstyle \gerI_{\!q}^{\,!}}^{\;n}  \mod(q-1) \, {U_q}^{\!\prime}  \qquad  \forall \, n \in \N_+ \; .  $$
This, together with  $ \, \gerI_q \bigcap \, (q-1) \, U_q = (q-1) \, \gerI_q \, $,  yields also
\[
  I_{\scriptscriptstyle {U_q}^{\!\prime}}^{\;n} \;{\textstyle \bigcap}\; \gerI_q^{\;!}  \, \equiv \,  I_{\scriptscriptstyle \gerI_{\!q}^{\,!}}^{\;n}  \mod(q-1) \, \gerI_q^{\;!}  \qquad  \forall \, n \in \N_+
\]
Finally, by suitable, iterated cancelation of factors  $ (q-1) $,  which is possible because of the condition  $ \, \gerI_q \bigcap \, (q-1) U_q = (q-1) \, \gerI_q \, $,   we eventually obtain
  \[
I_{\scriptscriptstyle {U_q}^{\!\prime}}^{\;n}\; {\textstyle \bigcap}\; \gerI_q^{\;!}  \, \equiv \,  I_{\scriptscriptstyle \gerI_{\!q}^{\,!}}^{\;n}  \mod {(q-1)}^n \, \gerI_q^{\;!}  \qquad  \forall \, n \in \N_+ \; .
\]
   \indent   To sum up, we have  $y_+ \in I_{\scriptscriptstyle {U_q}^{\!\prime}}^{\;N} \;{\textstyle \bigcap}\; \gerI_q^{\;!} = I_{\scriptscriptstyle \gerI_{\!q}^{\,!}}^{\;N} \; $; therefore, by definitions,
  \[
y = {(q-1)}^{-N} \, y_+  \; \in \;  {(q-1)}^{-N} \, I_{\scriptscriptstyle \gerI_q^{\;!}}^{\;N}  \; \subseteq \;  \big( \gerI_q^{\;!} \big)^{\!\curlyvee}  \, .
\]
 \vskip5pt
   {\it (d)} \,  Let  $\gerC_q$ be a strict quantization: by Proposition \ref{reciprocity}{\it (2)\/}  it is enough to prove  $ \gerC_q \subseteq {\big( \gerC_q^{\,\Lsh\,} \big)}^{\!\triangledown}$. We follow the same arguments  used for claim  {\it (c)}.   Let
$c \in \gerC_q$,  since  $\gerC_q \subseteq U_q = \big( {U_q}^{\!\prime} \big)^\vee$   ---  from  Theorem \ref{GQDP}  ---   and
 $\big( {U_q}^{\!\prime} \big)^\vee  = \,  {\textstyle \sum_{n \geq 0}} \, {(q-1)}^{-n} \, I_{\scriptscriptstyle {U_q}^{\!\prime}}^{\;n}$,
%
%
%
(notation as above) for  $c \in \gerC_q \subseteq U_q = \big( {U_q}^{\!\prime} \big)^\vee$  there exists  $N \in \N$  such that
 $ \; c_+ := {(q-1)}^N \, c \, \in \, I_{\scriptscriptstyle {U_q}^{\!\prime}}^{\;N} \, \subseteq \, {U_q}^{\!\prime} \; $.
%
%
                                                          \par
   Now, strictness of $ \gerC_q $   implies
\[
\big( {U_q}^{\!\otimes (n-1)} \otimes \gerC_q \big) \,{\textstyle \bigcap}\, {(q-1)}^n \, {U_q}^{\!\otimes n} = {(q-1)}^n \, \big( {U_q}^{\!\otimes (n-1)} \otimes \gerC_q \big)   \qquad \forall\; n\in\N_+
\]
hence  $ \, \gerC_q^{\,\Lsh\,} = \gerC_q \bigcap {U_q}^{\!\prime} \, $.  If  the above $N$ is $1$, then  $ \, c_+ = c \in {U_q}^{\!\prime} \, $,  thus  $ \, c \in \gerC_q \bigcap {U_q}^{\!\prime} = \gerC_q^{\,\Lsh\,} \, $,  q.e.d.  If instead  $ \, N > 1 \, $,  then
  \[
  \delta_n(c_+) \, \in \, \big( {(q-1)}^N \cdot {U_q}^{\!\otimes {n-1}} \otimes \gerC_q \big) \,{\textstyle \bigcap}\, \big( {(q-1)}^n \, {U_q}^{\!\otimes n} \big)  \qquad  \forall \, n \in \N_+
\]
and, since  $ \gerC_q $  is strict,  $ \; \delta_n(c_+)\in \, {(q-1)}^n \cdot {U_q}^{\!\otimes {n-1}} \otimes \gerC_q \; $  for all  $ \, n \in \N_+ \, $,  which means  $ \, c_+ \in \gerC_q^{\,\Lsh\,} \, $.  Thus, eventually, we have  $ \; c_+ \in \gerC_q^{\,\Lsh\,} \bigcap I_{\scriptscriptstyle {U_q}^{\!\prime}}^{\;N} \; $.
                                                                             \par
   Let us look, now, at  $ \; I_{\scriptscriptstyle \gerC_q^{\,\Lsh\,}} := I_{\scriptscriptstyle {U_q}^{\!\prime}} \,{\textstyle \bigcap}\, \gerC_q^{\,\Lsh\,} \; $.  Again in force of  strictness of  $\gerC_q\,$,  a geometrical argument (at  $ q = 1$)  as before leads us to
\[
    I_{\scriptscriptstyle {U_q}^{\!\prime}}^{\;n} \;{\textstyle \bigcap}\; \gerC_q^{\,\Lsh\,}  \, \equiv \,  I_{\scriptscriptstyle \gerC_q^{\,\Lsh\,}}^{\;n}  \mod {(q-1)}^n \, \gerC_q^{\,\Lsh\,} \;\; ,  \qquad  \forall \; n \in \N_+
\]
from which we conclude that  $ \; c_+ \in I_{\scriptscriptstyle {U_q}^{\!\prime}}^{\;N} \;{\textstyle \bigcap}\; \gerC_q^{\,\Lsh\,} = I_{\scriptscriptstyle \gerC_q^{\,\Lsh\,}}^{\;N} \; $.  Therefore, by the very definitions,
  \[
  c  \; = \;  {(q-1)}^{-N} \, c_+  \; \in \;  {(q-1)}^{-N} \, I_{\scriptscriptstyle \gerC_q^{\,\Lsh\,}}^{\;N}  \; \subseteq \;  \big( \gerC_q^{\,\Lsh\,} \big)^{\!\triangledown} \;\; ,  \qquad  \text{q.e.d.}
\]
 \vskip5pt
   {\it (e)} \,  This is a direct consequence of claims from {\it (a)\/}  through  {\it (d)}.
 \vskip5pt
   {\it (f)} \,  Once again, this is true because the whole construction is independent of the existence of real structures.
\end{proof}

\smallskip

   It is now time to clarify how the coisotropic subgroup $L$ of $G^*$ is linked to the coisotropic subgroup $K$ of $G\,$. We will give this relation in the weak quantization case first, and show how it improves under stronger hypothesis.

 \eject

\begin{theorem} \label{complimentary} Let  $K$ be a subgroup of $G$, and let
$\calI_q(K)\,$,  $\calC_q(K)\,$,  $\gerI_q(K)$  and  $\gerC_q(K)$  be weak quantizations as in Definition \ref{weakq}.  Then (with notation of Proposition \ref{bot-pro})
\begin{itemize}
  \item[(a)]  \quad  $ {\calI_q(K)}^{\!\curlyvee} = \, \gerI_q \big( K^{\langle \perp \rangle} \big) \; $;
  \item[(b)]  \quad  $ {\calC_q(K)}^{\!\triangledown} = \, \gerC_q \big( K^{\langle \perp \rangle} \big) \; $;
  \item[(c)]  \quad  if  $ \; \gerI_q(K) = {\big( {\gerI_q(K)}^! \,\big)}^{\!\curlyvee} \, $, then
$ {\gerI_q(K)}^{\,!} = \, \calI_q \big( K^{\langle \perp \rangle} \big)$;  in particular, this holds if the quantization  $\gerI_q(K)$  is  {\sl strict};
  \item[(d)]  \quad  if  $ \; \gerC_q =  {\big( {\gerC_q(K)}^\Lsh \,\big)}^{\!\triangledown}$,  then  ${\gerC_q(K)}^{\!\Lsh} = \, \calC_q \big( K^{\langle \perp \rangle} \big)$;  in particular, this holds if the quantization  $\gerC_q(K) $  is  {\sl strict};
   \item[(e)]  \quad  claims (a--d) hold as well in the framework of  {\sl real}  quantum subgroups.
\end{itemize}
\end{theorem}

\begin{proof}
 {\it (a)}  By  Proposition \ref{curlyvee}  we already have  ${\calI_q(K)}^{\!\curlyvee} \! = \gerI_q(L)$  for some  subgroup
$ L\subseteq G^*$.  In order to show that $ L = K^{\langle \perp\rangle} $,   we will proceed much like in the proof of ${F_q}^{\!\vee} \Big/ (q-1) \, {F_q}^{\!\vee} \cong U(\gerg^*)$,  as given in \cite{Ga3}, Theorem 4.7.
                                          \par
   Let us fix a subset  $\{ j_1, \dots, j_n \}$  of  $ J $ adapted to $K$  as in the proof of  Proposition \ref{curlyvee}.  Let $J^\vee := {(q\!-\!1)}^{-1} J \subset {F_q}^{\!\vee}$  and
$ j^{\,\vee} := {(q-1)}^{-1} j$  for all  $ j \in J $.  From the discussion in that proof, we argue also that  $\big\{ {(q-1)}^{-|\underline{e}|} j^{\,\underline{e}} \mod (q-1) \, {F_q}^{\!\vee} \;\big|\; \underline{e} \in \N^{\,n} \,\big\}$,  where  $j^{\,\underline{e}} = \prod_{s=1}^n j_s^{\,\underline{e}(i)}$, is a  $ \C $--basis  of  $ {F_1}^{\!\vee} $, and  $\big\{ j_1^{\,\vee}, \dots, j_n^{\,\vee} \big\}$  is a  $ \C $--basis  of  $\gert = J^\vee \mod (q-1) \, {F_q}^{\!\vee}$.
                                              \par
   Now,  $j_\mu \, j_\nu - j_\nu \, j_\mu \in (q \! - \! 1) \, J$  (for  $\mu, \nu \in \{1,\dots,n\}$)  implies that:
  \[
  j_\mu \, j_\nu - j_\nu \, j_\mu \; = \; (q - 1) \, {\textstyle \sum_{s=1}^n} \, c_s \, j_s \, + \, {(q - 1)}^2 \, \gamma_1 \, + \, (q - 1) \, \gamma_2
\]
for some  $c_s \in \Cqqm$,  $\gamma_1 \in J$  and  $\gamma_2 \in J^2 $.  Therefore
\[
  \big[ j_\mu^\vee, j_\nu^\vee \,\big] := j_\mu^\vee \, j_\nu^\vee - j_\nu^\vee \, j_\mu^\vee = {\textstyle \sum_{s=1}^n} \, c_s \, j_s^\vee + \gamma_1 + {(q\!-\!1)} \, \gamma_2^\vee \equiv {\textstyle \sum_{s=1}^n} \, c_s \, j_s^\vee \mod \, (q\!-\!1) \, {F_q}^{\!\vee}
\]
(where we posed  $\gamma_2^\vee := {(q-1)}^{-2} \gamma_2 \in {(q-1)}^{-2} {\big( J^\vee \big)}^2 \subseteq {F_q}^{\!\vee}$)  thus  the subspace $\gert := J^\vee \! \mod (q-1) \, {F_q}^{\!\vee}$  is a Lie subalgebra of  $ {F_1}^{\!\vee}$.  But then it should be  ${F_1}^{\!\vee} \cong U(\gert)$  as Hopf algebras, by the above description of  $ {F_1}^{\!\vee} $  and PBW theorem.
                                             \par
   Now for the second step.  The specialization map  $\pi^\vee \colon \, {F_q}^{\!\vee} \relbar\joinrel\twoheadrightarrow {F_1}^{\!\vee} = U(\gert)$  actually restricts to  $\eta \, \colon\, J^\vee \twoheadrightarrow  \gert = J^\vee \Big/ J^\vee \bigcap {\big( (q\!-\!1) \, {F_q}^{\!\vee} \big)} = J^\vee \Big/ \big( J + J^\vee J \big)$,  because  $ J^\vee \bigcap {\big( (q\!-\!1) \, {F_q}^{\!\vee} \big)} = J^\vee \bigcap {(q\!-\!1)}^{-1} {I_{\scriptscriptstyle F_q}}^{\hskip-3pt 2} = J + J^\vee J $.  Also, multiplication by  $ {(q\!-\!1)}^{-1} $  yields a  $ \Cqqm $--module  isomorphism  $ \mu \, \colon \, J \, {\buildrel \cong \over {\lhook\joinrel\relbar\joinrel\relbar\joinrel\twoheadrightarrow}} \, J^\vee $.
Let  $\rho \, \colon \, \mathfrak{m}_e \relbar\joinrel\twoheadrightarrow \mathfrak{m}_e \big/ {\mathfrak{m}_e}^{\!2} = \gerg^* $  be the natural projection map, and  $\nu \, \colon \, \gerg^* \lhook\joinrel\longrightarrow \mathfrak{m}_e $  a section of  $ \rho$.  The specialization map  $\pi \, \colon \, F_q \relbar\joinrel\twoheadrightarrow F_1$  restricts to  a map $\pi' \colon \, J \relbar\joinrel\relbar\joinrel\twoheadrightarrow J \big/ (J \bigcap \, (q-1) \, F_q) = \mathfrak{m}_e$.  Let's fix a section  $\gamma \, \colon \, \mathfrak{m}_e \! \lhook\joinrel\relbar\joinrel\rightarrow J$  of  $ \pi'$ and consider the composition   $\sigma := \eta \circ \mu \circ \gamma \circ \nu \, \colon \, \gerg^* \longrightarrow \gert$:  this is a well-defined Lie bialgebra morphism, independent of the choice of  $ \nu $  and  $ \gamma$.
                                                                   \par
   In the proof of  Proposition \ref{curlyvee}  we made a particular choice for the subset  $\{ j_1, \dots, j_n \}$.  As a consequence, the above analysis to prove that  $ \, \sigma \, \colon \, \gerg^* \cong \gert$  shows also that the left ideal  ${\calI_1}^{\!\curlyvee} := {\calI_q}^{\! \curlyvee} \! \mod (q\!-\!1) \, {F_q}^{\!\vee}$  of  $ U(\gert) $  is generated by
\[
\eta \big( {\calI_q}^{\!\curlyvee} \big) = (\eta \circ \mu) \big( \calI_q \big) = (\sigma \circ \rho \circ \pi) \big( \calI_q \big) = \sigma \big( \rho(\calI\,) \big) = \sigma \big( \gerk^\perp \big) \, .
\]
  So  ${\calI_1}^{\!\curlyvee} = U(\gerg^*) \cdot \gerk^\perp = U(\gerg^*) \cdot \langle \gerk^\perp \rangle = \gerI\big(K^{\langle \perp \rangle}\big)$   --- where we are identifying  $ \gerg^* $  with its image via  $ \sigma $  ---   which eventually means  $\gerl = \langle \gerk^\perp \rangle$.
 \vskip5pt
   {\it (b)} By  Proposition \ref{triangledown}  we have  ${\calC_q(K)}^{\!\triangledown} \! = \gerC_q(L)$  for some  coisotropic subgroup $ L$ in $G^*$. We must prove that  $L = K^{\langle \perp \rangle}$.  Once again, we mimic the procedure of the proof of Proposition \ref{triangledown}, and we fix a subset  $\big\{ j_1, \dots, j_n \big\}$  of  $ J $  as in the proof of  such Proposition.  Then, tracking the analysis we did there to prove that  $\sigma \, \colon \, \gerg^* \cong \gert$,  we see also that the unital subalgebra  ${\calC_1}^{\!\!\triangledown} := {\calC_q}^{\!\!\triangledown} \! \mod (q\!-\!1) \, {F_q}^{\!\vee}$  of  $ U\big(\gerg^*\big) $  is generated by  $\eta \big( {\calC_q}^{\!\!\triangledown} \big) = (\mu \circ \eta) \big( \calC_q \big) = (\sigma \circ \rho \circ \pi) \big( \calC_q \big) = \sigma \big( \rho(\calC\,) \big) = \sigma \big( \gerk^\perp \big)$.  Thus  $ {\calC_1}^{\!\!\triangledown} $  is the subalgebra of  $U\big(\gerg^*\big) $  generated by  $ \gerk^\perp $,  hence  ${\calC_1}^{\!\!\triangledown} = {\big\langle \gerk^\perp \big\rangle}_{\!\text{\it Alg}} = U \big( {\big\langle \gerk^\perp \big\rangle}_{\! \text{\it Lie}} \big) = U \big( \gerk^{\langle \perp \rangle} \big) = \gerC \big( K^{\langle \perp \rangle} \big) \, $,  which means  $ \, \gerl = \langle \gerk^\perp \rangle \, $,  q.e.d.
\vskip5pt
   {\it (c)}  Thanks to  Proposition \ref{esclamativo}  we already know that  ${\gerI_q(K)}^! = \, \calI_q(L)$  for some  coisotropic subgroup $L$ in $G^*$. Again, we must prove that  $L = K^{\langle \perp \rangle}$.  Note that we can assume  $ K $  to be connected, as its relationship with  $ \gerI_q(K) $  passes through  $ \gerk $  alone; thus in the end we simply have to prove that  $\mathfrak{l} := \text{\it Lie}\,(L) = \gerk^{\langle \perp \rangle} = \gerk^\perp$,  \, taking into account that  $\gerk^{\langle \perp \rangle} = \gerk^\perp$  because  $ \gerk $  is coisotropic, by a remark following Proposition \ref{reciprocity}.
                                                                     \par
   By assumption  $\gerI_q(K) = {\big( {\gerI_q(K)}^! \,\big)}^{\!\curlyvee}$;  this and  {\it (a)\/}  together give
  \[
  \gerI_q(K)  \; = \;  {\big( {\gerI_q(K)}^! \,\big)}^{\!\curlyvee}  = \;  {\calI_q(L)}^\curlyvee  \, = \; \gerI_q\big(L^{\langle \perp \rangle}\big)  \; = \;  \gerI_q\big(L^\perp\big)
\]
where  $ L^{\langle \perp \rangle} = L^\perp$  because  $ L $  is coisotropic as well: at  $q = 1 $, this implies  $\gerk = \gerl^\perp$,  q.e.d.
 \vskip5pt
   {\it (d)} We must prove that  $L = K^{\langle \perp \rangle} $: as above we can assume  $K$  to be connected, so we only have to prove that  $\mathfrak{l} := \text{\it Lie}\,(L) = \gerk^{\langle \perp \rangle} = \gerk^\perp$  (as  $ \gerk $  is coisotropic, by Proposition \ref{converse}.
                                                                     \par
   By assumption  $\gerC_q =  {\big( {\gerC_q(K)}^\Lsh \,\big)}^{\!\triangledown}$;  this along with  {\it (c)\/}  gives
  \[
 \gerC_q(K)  \; = \;  {\big( {\gerC_q(K)}^\Lsh \,\big)}^{\!\triangledown}  = \;  {\calC_q(L)}^\triangledown  \, = \; \gerC_q\big(L^{\langle \perp \rangle}\big)  \; = \;  \gerC_q\big(L^\perp\big)
\]
with  $L^{\langle \perp \rangle} = L^\perp$  since  $ L $  is coisotropic too: specializing at  $q = 1$,  this yields  $\gerk = \gerl^\perp$.
 \vskip5pt
   {\it (e)} This is clear again since all arguments pass through unchanged in the real setup.
\end{proof}

\smallskip

\begin{corollary}\label{coro}
 Let  $\calI_q(K) $ and $\calC_q(K)$ be weak quantizations of a (not necessarily) coisotro\-pic subgroup $K$ of $G$, of type $\calI$ and $\calC$ respectively.  Then, with notation of  Definition \ref{coisotropic duality},  we have
  $$  {\big( {\calI_q(K)}^\curlyvee \big)}^! =\calI_q \big(\Kk \big) \;\; ,  \qquad  {\big( {\calC_q(K)}^\triangledown \big)}^{\!\!\Lsh} = \, \calC_q \big(\Kk \big) \;\; .  $$
\end{corollary}

\begin{proof}
 Theorem \ref{complimentary}{\it(a)\/}  gives  $ \, {\calI_q(K)}^{\!\curlyvee} \! = \gerI_q \big( K^{\langle \perp \rangle} \big) \, $, and  Proposition \ref{reciprocity}  yields
  $$  \Big( \gerI_q \big( K^{\langle \perp \rangle} \big)^! \Big)^{\!\curlyvee} \! = \, \Big( {\big( {\calI_q(K)}^{\!\curlyvee} \big)}^! \Big)^{\!\curlyvee} \! = \; {\calI_q(K)}^{\!\curlyvee} = \; \gerI_q \big( K^{\langle \perp \rangle} \big)  $$
so that  $\Big( \gerI_q \big( K^{\langle \perp \rangle} \big)^! \Big)^{\!\curlyvee} \! = \; \gerI_q \big( K^{\langle \perp \rangle} \big)$.  Then  Theorem \ref{complimentary}  gives
  $$  \gerI_q \big( K^{\langle \perp \rangle} \big)^! = \, \calI_q\big( (K^{\langle \perp \rangle})^{\langle \perp \rangle} \big) \, = \, \calI_q \big(\!\Kk \big)  $$
by  Proposition \ref{bot-pro}.  Therefore  $ {\big( {\calI_q(K)}^\curlyvee \big)}^! = \, \gerI_q \big( K^{\langle \perp \rangle} \big)^! = \, \calI_q \big(\!\Kk \big)$ as claimed.
                                                                       \par
   Similarly,  Theorem \ref{complimentary}{\it (b)\/}  gives  $ \, {\calC_q(K)}^{\!\triangledown} \! = \gerC_q \big( K^{\langle \perp \rangle} \big) \, $,  and the first remark after Proposition \ref{reciprocity}  yields
  $$  \Big( \gerC_q \big( K^{\langle \perp \rangle} \big)^{\!\!\Lsh} \,\Big)^{\!\triangledown} \! = \, \Big( {\big( {\calC_q(K)}^\triangledown \big)}^{\!\!\Lsh} \,\Big)^{\!\triangledown} \! = \; {\calC_q(K)}^\triangledown = \; \gerC_q \big( K^{\langle \perp \rangle} \big)  $$
so that  $ \; \Big( \gerC_q \big( K^{\langle \perp \rangle} \big)^{\!\!\Lsh} \,\Big)^{\!\triangledown} \! = \; \gerC_q \big( K^{\langle \perp \rangle} \big) \; $.  Then again by  Theorem \ref{complimentary}{\it (d)\/}  we get
  $$  \gerC_q \big( K^{\langle \perp \rangle} \big)^{\!\!\Lsh} = \, \calC_q\big( (K^{\langle \perp \rangle})^{\langle \perp \rangle} \big) \, = \, \calC_q \big(\! \Kk \big)  $$
still by  Proposition \ref{bot-pro}.  Thus  ${\big( {\calC_q(K)}^\triangledown \big)}^{\!\!\Lsh} = \, \gerC_q \big( K^{\langle \perp \rangle} \big)^{\!\!\Lsh} = \, \calC_q \big(\! \Kk \big)$  as claimed.
\end{proof}

\medskip

\begin{remark}
 One might guess that the analogue to this Corollary holds true for weak quantizations of type $\gerI$ and $\gerC$ as well: actually, we have no clue about that, in either sense.
\end{remark}

\medskip

  We now consider the ``compatibility''      among different Drinfeld-like maps acting on quantizations of different types over a single pair  {\sl (subgroup, space)}.  Indeed, we show that Drinfeld's functors preserve the subgroup-space correspondence   --- Proposition \ref{correspondence} ---   and the orthogonality correspondence   --- Proposition \ref{orthogonality} ---   (if either occurs at the beginning) between different quantizations as mentioned.

\medskip

\begin{proposition}\label{correspondence}
 Let  $K$  be a closed subgroup of  $ G $,  and let  $\Psi $ and $ \Phi $  be the map mentioned in  \S \ref{subgrps-homspaces}.  Then the following holds:
\begin{itemize}
   \item[(a)] \,  Let  $ \, \calC_q $  and  $ \calI_q $  be as in  Section \ref{quantizations}.  If  $ \; \Psi(\calC_q) = \calI_q \, $,  then  $ \; \Psi \big( {\calC_q}^{\!\!\triangledown} \big) = {\calI_q}^{\!\curlyvee} \; $.
  \item[(b)] \,  Let  $\calI_q $  and  $ \calC_q $  be as in  Section \ref{quantizations}.  If  $ \; \Phi(\calI_q) = \calC_q \, $, then  $ \; \Phi \big({\calI_q}^{\!\curlyvee} \big) = {\calC_q}^{\!\!\triangledown} \; $.
  \item[(c)] \,  Let  $\gerC_q $  and  $ \gerI_q $  be as in  Section \ref{quantizations}.  If  $ \; \Psi(\gerC_q) = \gerI_q \, $,  then  $ \; \Psi \big({\gerC_q}^{\!\!\Lsh} \big) \subseteq {\gerI_q}^{\! !} \; $.
  \item[(d)] \,  Let  $\gerI_q $  and  $ \gerC_q $  be as in  Section \ref{quantizations}.  If  $ \; \Phi(\gerI_q) = \gerC_q \, $,  then  $ \; \Phi \big({\gerI_q}^{\! !} \big) = {\gerC_q}^{\!\!\Lsh} \; $.
\end{itemize}
\end{proposition}

\begin{proof}
 Claims  {\it (a)\/}  and  {\it (c)\/}  both follow trivially from  definitions.
                                                 \par
   As to claim  {\it (b)},  let  $ \, \eta \in {\calC_q}^{\!\! +} ={\Phi(\calI_q)}^+ \, $, so that  $ \, \Delta(\eta) \in \eta \otimes 1 + F_q \otimes \calI_q \, $.  Then  $ \, \eta^\vee := {(q-1)}^{-1} \eta \, $  enjoys
  $$  \Delta\big(\eta^\vee\big) \in \eta^\vee \! \otimes 1+ F_q \otimes {(q-1)}^{-1} \, \calI_q \subseteq \eta^\vee \! \otimes 1 + {F_q}^{\!\vee} \! \otimes {\calI_q}^{\!\curlyvee}  $$
whence  $ \, \eta^\vee \! \in {\big({F_q}^{\!\vee}\big)}^{\text{\it co}
{\calI_q}^{\!\!\curlyvee}} \! =: \Phi\big({\calI_q}^{\!\curlyvee}\big)
\, $.  Since  $ {\calC_q}^{\!\!\triangledown} $  is generated (as a
subalgebra) by  $ \, {(q-1)}^{-1} \, {\calC_q}^{\!\! +} \, $,  we conclude that  $ \, {\calC_q}^{\!\!\triangledown}
\subseteq \Phi\big({\calI_q}^{\!\curlyvee}\big) \, $.
                                                 \par
   Conversely, let  $ \, \varphi \in \Phi\big({\calI_q}^{\!\curlyvee}\big) \, $.  Then  $ \, \Delta(\varphi) \in \varphi \otimes 1 + {F_q}^{\!\vee}\! \otimes {\calI_q}^{\!\curlyvee} \, $,  and there exists
$ \, n \in \N \, $  such that  $ \, \varphi_+ := {(q-1)}^n \varphi \in \calI_q \, $,  so that
$ \, \Delta(\varphi_+) \in F_q \otimes \calI_q + \calI_q
\otimes F_q \, $  (since  $ \, \calI_q \,\coideal\, F_q \, $).  Then
  $$  \Delta(\varphi_+) \in \big( \varphi_+ \otimes 1 + {(q-1)}^n
{F_q}^{\!\vee} \! \otimes {\calI_q}^{\!\curlyvee} \big) \,{\textstyle \bigcap}\,
\big( F_q \otimes \calI_q + \calI_q \otimes F_q \big)  $$
or equivalently
\begin{equation}
\label{quarantotto}  \Delta(\varphi_+) \, - \, \varphi_+ \otimes 1 \; \in \big(
{(q-1)}^n {F_q}^{\!\vee} \! \otimes {\calI_q}^{\!\curlyvee} \big)
\, {\textstyle \bigcap} \, \big( F_q \otimes \calI_q + \calI_q
\otimes F_q \big)
\end{equation}
Now, the description of  $ {\calI_q}^{\!\curlyvee} $  given in the proof of Proposition \ref{curlyvee} implies that
  $$  \big( {(q-1)}^n {F_q}^{\!\vee} \! \otimes {\calI_q}^{\!\curlyvee}
\big) \, {\textstyle \bigcap} \, \big( F_q \otimes \calI_q + \calI_q
\otimes F_q \big)  \; = \;  F_q \otimes \calI_q  $$
this together with (\ref{quarantotto}) yields  $ \, \Delta(\varphi_+) \in \varphi_+ \otimes 1 + F_q \otimes \calI_q \, $,  hence  $ \, \varphi_+ \in {F_q}^{\text{\it co}\calI_q} =: \Phi(\calI_q) = \calC_q \, $  and so  $ \, \varphi \in {(q-1)}^n \, \calC_q \bigcap {F_q}^{\!\vee} \, $.  On the other hand, the description of  $ {\calC_q}^{\!\!\triangledown} $  in the proof of Proposition \ref{triangledown} implies that  $ \, {(q-1)}^{-n} \, \calC_q \, {\textstyle \bigcap} \, {F_q}^{\!\vee} \subseteq \, {\calC_q}^{\!\!\triangledown} \, $,  hence we get  $ \, \varphi \in {\calC_q}^{\!\!\triangledown} \, $,  q.e.d.
                                          \par
   We finish with claim  {\it (d)}.  For the inclusion
$ \, \Phi \big({\gerI_q}^{\! !} \big) \supseteq {\gerC_q}^{\!\!\Lsh} \, $,  let  $ \, \kappa \in {\gerC_q}^{\!\!\Lsh} \, $.  Since  $ \Phi \big({\gerI_q}^{\! !} \big) $  contains the scalars, we may assume that
$ \, \kappa \in \Ker\,(\epsilon)$, thus  $\Delta(\kappa) = \kappa \otimes 1 + 1 \otimes \kappa + \delta_2(\kappa) \, $.  By  Proposition \ref{Lsh},  we have  $ \, {\gerC_q}^{\!\!\Lsh} \coideal_\ell \, {U_q}^{\!\prime} \, $;  thus  $ \, \Delta(\kappa) - \kappa \otimes 1 = 1 \otimes \kappa + \delta_2(\kappa) \in {U_q}^{\!\prime} \otimes {\gerC_q}^{\!\!\Lsh} \, $,  and more precisely
  $$  \Delta(\kappa) - \kappa \otimes 1 = 1 \otimes \kappa + \delta_2(\kappa) \in {U_q}^{\!\prime} \otimes {\big( {\gerC_q}^{\!\!\Lsh} \,\big)}^+  \;\; .  $$
Since  $ \, {\gerC_q}^{\!\!\Lsh} \subseteq  \Psi \big( {\gerC_q}^{\!\!\Lsh} \big) \subseteq {\gerI_q}^{\! !} \, $,  \, by claim  {\it (c)},  we get
$ \, \Delta(\kappa) - \kappa \otimes 1 \in {U_q}^{\!\prime} \otimes {\gerI_q}^{\! !} \, $,  so  $ \, \kappa \in {\big( {U_q}^{\!\prime} \big)}^{\text{\it co}{\gerI_q}^{\! !}} =: \Phi \big( {\gerI_q}^{\! !} \big) \, $.  Thus  $ \, {\gerC_q}^{\!\!\Lsh} \subseteq \Phi \big( {\gerI_q}^{\! !} \big) \, $.  For the converse inclusion, let  $ \, \eta \in \Phi\big({\gerI_q}^{\! !}\big) \, $;  again, we can assume  $ \, \eta \in \Ker\,(\epsilon) \, $  too.  As  $ \, {\gerI_q}^{\! !} \subseteq \gerI_q \, $,   we get  $ \, \eta \in \Phi\big({\gerI_q}^{\! !}\big) \subseteq \Phi\big(\gerI_q\big) = \gerC_q \, $.  Then  $ \, \delta_n(\eta) \in {U_q}^{\! \otimes n} \! \otimes \gerC_q \, $  for all  $ \, n \in \N_+ \, $,  so
  $$  \delta_n(\eta) \in {(q \! - \! 1)}^n \! \left( {\textstyle
\sum\limits_{s=1}^{n-1}} \, {U_q}^{\!\! \otimes (\!s-1\!)} \! \otimes \gerI_q
\otimes {U_q}^{\!\! \otimes (n-s\!)} \!\right) \! {\textstyle \bigcap} \big( {U_q}^{\!\! \otimes (n-1\!)} \! \otimes \gerC_q \big) \subseteq {(q \! - \! 1)}^n \, {U_q}^{\!\! \otimes (n-1\!)} \! \otimes \gerC_q  $$
hence  $ \, \delta_n(\eta) \in {(q \! - \! 1)}^n \, {U_q}^{\! \otimes (n-1)} \! \otimes \gerC_q \, $  ($ n \! \in \! \N_+ $)  and  $ \, \eta \in \gerC_q \, $,  which means $\eta \in {\gerC_q}^{\!\!\Lsh} \, $.
\end{proof}

\smallskip

\begin{remark}
 The inclusion  $ \, \Psi \big( {\gerC_q}^{\!\!\Lsh} \big) \subseteq {\gerI_q}^{\! !} \, $  of  Proposition \ref{correspondence}{\it (c)\/}  is  {\sl not\/}  an identity in general   --- indeed,  {\sl counterexamples do exist}.
\end{remark}

\medskip

   Finally, we look at what happens when our Drinfeld-like recipes are applied to a pair of quantizations associated with a same subgroup / homogeneous spaces with respect to some fixed double quantization (in the sense of  Section \ref{quantizations}).  The result reads as follows:

\medskip

\begin{proposition}\label{orthogonality}
Let  $ \big( F_q[G] \, , U_q(\gerg)\big)$  be a double quantization of  $(G,\gerg)\,$.  Then:
\begin{itemize}
  \item[(a)] \,  Let  $\calC_q$ and $\gerI_q$  be weak quantizations and assume that  $ \, \calC_q = {\gerI_q}^{\!\perp} \, $  and  $ \, \gerI_q = {\calC_q}^{\!\perp} \, $.  Then  $ \, {\gerI_q}^{\! !} = {\big( {\calC_q}^{\!\!\triangledown} \big)}^\perp \, $  and  $ \, {\calC_q}^{\!\!\triangledown} \subseteq {\big( {\gerI_q}^{\! ! \,} \big)}^\perp \, $.  If, in addition, either one of  $\,\calC_q$  or  $\,\gerI_q$  is  {\sl strict},  then also  $ \, {\calC_q}^{\!\!\triangledown} = {\big( {\gerI_q}^{\! ! \,} \big)}^\perp \, $.
  \item[(b)] \,  Let  $\gerC_q$ and  $\calI_q$ be weak quantizations and assume that  $ \, \calI_q = {\gerC_q}^{\!\perp} \, $  and  $ \, \gerC_q = {\calI_q}^{\!\perp} \, $.  Then  $ \, {\gerC_q}^{\!\!\Lsh} = {\big( {\calI_q}^{\!\curlyvee} \big)}^\perp \, $  and  $ \, {\calI_q}^{\!\curlyvee} \subseteq {\big( {\gerC_q}^{\!\!\Lsh\,} \big)}^\perp \, $.  If, in addition, either one of  $\,\gerC_q$  or   $\,\calI_q$  is  {\sl strict},  then also  $ \, {\calI_q}^{\!\curlyvee} = {\big( {\gerC_q}^{\!\!\Lsh\,} \big)}^\perp \, $.
\end{itemize}
\end{proposition}

\begin{proof}
 Both in claim  {\it (a)\/}  and in claim  {\it (b)\/}  the
orthogonality relations between  $ \gerC_q $  and  $ \calI_q $  and
between  $ \calC_q $  and  $ \gerI_q $  are considered w.r.t.~the
pairing between  $ F_q[G] $  and  $ U_q(\gerg) $,  and the subsequent
orthogonality relations are meant w.r.t.~the pairing between
$ {F_q[G]}^\vee $  and  $ {U_q(\gerg)}' $.  Indeed, by  Theorem \ref{GQDP},  $ \Big( {U_q(\gerg)}' \, , {F_q[G]}^\vee \Big) $  is a double quantization of  $ \big( G^*, \, \gerg^* \big)\,$.
 \vskip5pt
   {\it (a)}  First,  $ \, \epsilon(\gerI_q) = 0 \, $  because  $ \gerI_q $  is a coideal.  Then  $ \, x = \delta_1(x) \in (q-1) \, U_q \, $  for all  $ \, x \in {\gerI_q}^{\! !} \, $,  hence  $ \, {\gerI_q}^{\! !} \subseteq (q-1) \, U_q \, $.  Thus we have
  $$  \big\langle \calC_q, {\gerI_q}^{\! !\,} \big\rangle \subseteq (q-1) \, \Cqqm \;\; .  $$
   \indent   Now let  $ \, J = J_{\scriptscriptstyle F_q} \, $  be the ideal
of  $ F_q $, and take  $ \, c_i\in \calC_q \cap J$ ($i=1,\ldots,n$)\,;  then  $ \, \langle c_i, 1 \rangle = \epsilon(c_i) = 0 \, $  ($ i= 1, \dots, n$)\,.  Given  $ \, y \in {\gerI_q}^{\! !} \, $,  look at
  $$  \left\langle \, {\textstyle \prod\limits_{i=1}^n} c_i \, ,
\, y \right\rangle  = \left\langle \, \mathop{\otimes}\limits_{i=1}^n c_i \, , \Delta^n(y) \right\rangle
= \bigg\langle \, \mathop{\otimes}\limits_{i=1}^n c_i \, , {\textstyle \sum\limits_{\Psi \subseteq \{1,\dots,n\}}} \hskip-3pt \delta_\Psi(y) \bigg\rangle  = {\textstyle \sum\limits_{\Psi \subseteq \{1,\dots,n\}}} \! \bigg\langle \, \mathop{\otimes}\limits_{i=1}^n c_i \, , \delta_\Psi(y) \bigg\rangle  $$
   \indent   Consider the summands in the last term of the above formula.  Let  $ \, \vert\Psi\vert = t$  ($t \leq n$)\,,  then
\[
\Big\langle \mathop{\otimes}\limits_{i=1}^n c_i \, , \, \delta_\Psi(y) \Big\rangle  \; = \;  \Big\langle \mathop{\otimes}\limits_{i \in \Psi} c_i \, , \, \delta_t(y) \Big\rangle
\cdot {\textstyle \prod\limits_{j \not\in \Psi}} \big\langle\, c_j \, , 1 \,\big\rangle
\]
by definition of  $\delta_\Psi\,$.  Thanks to the previous analysis, we have  $ \, \prod_{j \not\in \Psi} \langle c_j \, , 1 \rangle = 0 \, $  unless
$ \, \Psi = \{1,\dots,n\} \, $,  and in the latter case
\[
\delta_\Psi(y) \, = \, \delta_n(y)  \; \in \;  {(q-1)}^n \, {\textstyle \sum\limits_{s=1}^n} \, {U_q}^{\!\otimes (s-1)} \! \otimes \gerI_q \otimes {U_q}^{\!\otimes (n-s)}  \quad .
\]
The outcome is
\[
  \left\langle\, \mathop{\otimes}\limits_{i=1}^n c_i \, ,
\, y \right\rangle  =  \bigg\langle \, \mathop{\otimes}\limits_{i=1}^n c_i \, , \, \delta_n(y) \bigg\rangle  \in
\left\langle \, \mathop{\otimes}\limits_{i=1}^n c_i \, , \,
{(q-1)}^n {\textstyle \sum\limits_{s=1}^n} \, {U_q}^{\otimes (s-1)}
\! \otimes \gerI_q \otimes {U_q}^{\otimes (n-s)} \right\rangle = 0
\]
because  $ \, y \in {\gerI_q}^{\! !}$  and  $\gerI_q = {\calC_q\phantom{)}}^{\!\!\!\perp} \, $  by assumption.  Thus $ \; \big\langle \! {(q\!-\!1)}^{-n} {\big( \calC_q \bigcap J \,\big)}^n , \, {\gerI_q}^{\! !} \,\big\rangle = 0 \, $,  for all  $ \, n \in \N_+ \, $.  In addition,  $ \, \big\langle 1 \, , \, {\gerI_q}^{\! !} \,\big\rangle = \epsilon\big({\gerI_q}^{\! !}\big) = 0 \, $.  The outcome is  $ \, \big\langle {\calC_q}^{\!\!\triangledown}, \, {\gerI_q}^{\! !} \,\big\rangle = 0 \, $,  whence  $ \, {\gerI_q}^{\! !} \subseteq {\big( {\calC_q}^{\!\!\triangledown} \big)}^\perp \, $ ,  and  $ \, {\calC_q}^{\!\!\triangledown} \subseteq {\big( {\gerI_q}^{\! ! \,} \big)}^\perp \, $.
                                                                          \par
   Now we prove also  $ \, {\big( {\calC_q}^{\!\!\triangledown} \big)}^\perp \subseteq {\gerI_q}^{\! !} \, $.  Notice that  $ \, {\calC_q}^{\!\!\triangledown} \supseteq \calC_q \, $,  whence  $ \, {\big( {\calC_q}^{\!\!\triangledown} \big)}^\perp \subseteq {\calC_q \phantom{)}}^{\!\!\!\!\perp} = \gerI_q \, $;  therefore  $ \, {\big( {\calC_q}^{\!\!\triangledown} \big)}^\perp \subseteq \gerI_q \, $.  Pick now  $ \, \eta \in {\big( {\calC_q}^{\!\!\triangledown} \big)}^\perp \, $  (inside  $ {U_q}'$).  Since  $ \, \eta \in {U_q}' \, $,  for all  $ \, n \in \N_+ \, $  we have  $ \, \delta_n(\eta) \in {(q-1)}^n {U_q}^{\otimes n} \, $,  and from  $ \, \eta \in {\big( {\calC_q}^{\!\!\triangledown} \big)}^\perp \, $  we get also that  $ \, \eta_+ := {(q-1)}^{-n} \delta_n(\eta) \, $  enjoys  $ \, \Big\langle {\big(\, \calC_q \bigcap J_{F_q} \,\big)}^{\otimes n} , \, \eta_+ \Big\rangle = 0 \, $   --- acting as before ---   so that
  \[
  \eta_+ \in {\Big( {\big(\, \calC_q \,{\textstyle \bigcap}\, J_{F_q} \,\big)}^{\otimes n} \Big)}^{\!\perp} \; = \, {\textstyle \sum\limits_{r+s=n-1}} {U_q}^{\otimes r} \otimes {\big(\, \calC_q
\,{\textstyle \bigcap}\, J_{F_q} \,\big)}^{\!\perp} \otimes {U_q}^{\otimes s}  \quad .
\]
\indent   Moreover  $ \, \delta_n(\eta) \in {J_{U_q}}^{\!\otimes n} \, $,  hence  $ \, \delta_n(\eta) \in \big( {(q \! - \! 1)}^n {U_q}^{\otimes n} \big) \bigcap {J_{U_q}}^{\!\otimes n} \! = \, {(q \! - \! 1)}^n {J_{U_q}}^{\!\otimes n} \, $,  so
  \[
\displaylines{
   \eta_+ \in {\Big( {\big(\, \calC_q \,{\textstyle \bigcap}\, J_{F_q} \big)}^{\otimes n} \Big)}^{\!\perp} \,{\textstyle \bigcap}\, {J_{U_q}}^{\!\otimes n}  =  \Big( {\textstyle \sum\limits_{r+s=n-1}} \, {U_q}^{\otimes r} \otimes {\big(\, \calC_q \,{\textstyle \bigcap}\, J_{F_q} \,\big)}^\perp \otimes {U_q}^{\otimes s} \Big) \,{\textstyle \bigcap}\, {J_{U_q}}^{\!\otimes n} \, =  \cr
   \hfill   = \, {\textstyle \sum_{r+s=n-1}} \, {J_{U_q}}^{\!\otimes r} \otimes \Big( \! {\big(\, \calC_q \,{\textstyle \bigcap}\, J_{F_q} \,\big)}^\perp \,{\textstyle \bigcap}\, J_{U_q} \Big) \otimes {J_{U_q}}^{\!\otimes s}  \quad .  \cr }
\]
Since  $ \, {\big(\, \calC_q \bigcap J_{F_q} \,\big)}^\perp \bigcap \, J_{U_q} = \, {\calC_q}^{\!\perp} \bigcap J_{U_q} = \, \gerI_q \bigcap J_{U_q} = \, \gerI_q \, $,  we have
\[
\eta_+ \, \in {\textstyle \sum\limits_{r+s=n-1}} \hskip-3pt {J_{U_q}}^{\!\otimes r} \! \otimes \gerI_q \otimes {J_{U_q}}^{\!\otimes s}
\]
whence
\[
\delta_n(\eta) \, \in \, {(q \! - \! 1)}^n {\textstyle \sum\limits_{r+s=n-1}} \hskip-3pt {U_q}^{\otimes r} \! \otimes \gerI_q \otimes {U_q}^{\otimes s}   \qquad  \forall \; n \in \N_+ \;\; .
\]
Being, in addition,  $ \, \eta \in \gerI_q \, $,  for we proved that
$ \, {\big( {\calC_q}^{\!\!\triangledown} \big)}^\perp \subseteq \gerI_q \, $,  we get  $ \, \eta \in {\gerI_q}^{\! !} \, $.   Therefore  $ \, {\big( {\calC_q}^{\!\!\triangledown} \big)}^\perp \subseteq {\gerI_q}^{\! !} \, $,  q.e.d.
                                        \par
   Finally, assume that  $\calC_q$ or $\gerI_q$ are strict quantizations.  Then we must still prove that  $ \, {\calC_q}^{\!\!\triangledown} = {\big( \gerI_q^{\;\,!\,} \big)}^\perp \, $.  Since  $ \, \calC_q = {\gerI_q}^{\!\perp} \, $  and  $ \, \gerI_q = {\calC_q}^{\!\perp} \, $,  it is easy to check that  $ \calC_q $  is strict if and only if  $\gerI_q $  is; therefore, we can assume that  $\gerI_q$ is strict.
                                        \par
   The assumptions and  Theorem \ref{converse} {\it (b)\/}  give  $ \, \gerI_q = \big( \gerI_q^{\;\,!\,} \big)^{\!\curlyvee} \, $;  moreover,  $ \, \calI_q := {\gerI_q}^{\! !} \, $  is strict.  Then we can apply the first part of claim  {\it (b)}   --- which is proved, later on, in a way independent of the present proof of claim  {\it (a)\/}  itself ---   and get  $ \, \big( {\calI_q}^{\!\curlyvee} \big)^\perp = \, \big( {\calI_q}^{\!\perp} \big)^{\!\Lsh} \, $.  Therefore
\begin{equation}\label{quarantanove}
 {\calC_q}^{\!\!\triangledown}  \, =  \; \big( {\gerI_q}^{\!\perp} \big)^{\!\triangledown}  \, = \;  \Big( \Big( \big( \gerI_q^{\;\,!\,} \big)^{\!\curlyvee} \Big)^{\!\perp\,} \Big)^{\!\triangledown}  \, = \;  \Big( \big( {\calI_q}^{\!\curlyvee} \big)^\perp \Big)^{\!\triangledown}  \, = \;  \Big( \big( {\calI_q}^{\!\perp} \big)^{\!\Lsh\,} \Big)^{\!\triangledown}  \quad .
\end{equation}
   \indent   Now, it is straightforward to prove that  $\calI_q$ strict implies that  ${\calI_q}^{\!\perp}$ is strict as well.  Then  Proposition \ref{converse}{\it (d)\/}  ensures  $ \, \Big(\! \big( {\calI_q}^{\!\perp} \big)^{\!\Lsh\,} \Big)^{\!\triangledown} \! = \, {\calI_q}^{\!\perp} \, $.  This along with (\ref{quarantanove}) yields  $ \, {\calC_q}^{\!\!\triangledown}  =  \Big( \big( {\calI_q}^{\!\perp} \big)^{\!\Lsh\,} \Big)^{\!\triangledown}  = \,  {\calI_q}^{\!\perp}  =  {\big( \gerI_q^{\;\,!\,} \big)}^\perp \, $,  ending the proof of  {\it (a)}.
 \vskip5pt
   {\it (b)}  With much the same arguments as for  {\it (a)},  we find as well that
  $$  \big\langle {\calI_q}^{\!\curlyvee}, \, {\gerC_q}^{\!\!\Lsh} \, \big\rangle  \; \in \;  \big\langle J^{\otimes (n-1)} \! \otimes \calI_q \, , \, {U_q}^{\otimes (n-1)} \! \otimes \gerC_q \big\rangle  \; \subseteq \;  \big\langle \calI_q \, , \gerC_q \big\rangle  \; = \;  0  $$
because  $ \, \calI_q = {\gerC_q}^{\!\perp} \, $;  this means that
  \begin{equation}\label{quarantadieci}
{\calI_q}^{\!\curlyvee} \subseteq {\big( {\gerC_q}^{\!\!\Lsh\,}
\big)}^\perp \quad ,  \qquad \qquad  {\gerC_q}^{\!\!\Lsh} \subseteq
{\big( {\calI_q}^{\!\curlyvee} \big)}^\perp \quad .
\end{equation}
   \indent   Let now  $ \, \kappa \in {\big( {\calI_q}^{\!\curlyvee} \big)}^\perp_q \; \big( \! \subseteq \! {U_q}' \, \big) \, $.  Since  $ \, \kappa \in {U_q}' \, $,  we have  $ \, \delta_n(\kappa) \in {(q-1)}^n {U_q}^{\otimes n} \, $  for all  $ \, n \in \N \, $;   moreover, from  $ \, \kappa \in {\big( {\calI_q}^{\!\curlyvee} \big)}^\perp \, $  it follows that  $ \, \kappa_+ := {(q-1)}^{-n} \delta_n(\kappa) \in {U_q}^{\otimes n} \, $  enjoys  $ \, \Big\langle J^{\otimes (n-1)} \otimes \calI_q \, , \, \kappa_+ \Big\rangle = 0 \, $,  so that
  \[
  \kappa_+ \in {\Big( J^{\otimes (n-1)} \otimes \calI_q \Big)}^{\!\perp} = \, {\textstyle \sum_{r+s=n-2}} \, {U_q}^{\otimes r} \otimes J^\perp \otimes {U_q}^{\otimes s} \otimes U_q \, + \,{U_q}^{\otimes (n-1)} \otimes {\calI_q \phantom{)}}^{\!\!\!\perp} \;\; .
\]
In addition,  $ \, \delta_n(\kappa) \in {J_{U_q}}^{\!\!\otimes n} \, $,  where  $ \, J_{U_q} := \Ker\,\big( \, \epsilon \! : U_q \longrightarrow \Cqqm \big) \, $;  therefore  $ \, \delta_n(\kappa) \in \big( {(q-1)}^n {U_q}^{\otimes n} \big) \bigcap {J_{U_q}}^{\!\!\otimes n} = {(q-1)}^n {J_{U_q}}^{\!\!\otimes n} \, $,  which together with the above formula yields
  \[
  \displaylines{
   \kappa_+ \, \in \, {\Big( J^{\otimes (n-1)} \otimes \calI_q \Big)}^{\!\perp} {\textstyle \bigcap} \; {J_{U_q}}^{\!\!\otimes n} \, =   \hfill  \cr
   = \, \bigg(\, {\textstyle \sum\limits_{r+s=n-2}} {U_q}^{\otimes r} \otimes J^\perp \otimes {U_q}^{\otimes s} \otimes U_q \bigg) \, {\textstyle \bigcap} \; {J_{U_q}}^{\!\!\otimes n} \, + \, \Big( {U_q}^{\otimes (n-1)} \otimes {\calI_q \phantom{(}}^{\!\!\!\perp} \Big) \, {\textstyle \bigcap} \; {J_{U_q}}^{\!\!\otimes n} \, =   \quad  \cr
   \qquad   = \, {\textstyle \sum\limits_{r+s=n-2}} {J_{U_q}}^{\!\!\otimes r} \otimes \Big( J^\perp \, {\textstyle \bigcap} \; J_{U_q} \Big) \otimes {J_{U_q}}^{\!\!\otimes s} \otimes J_{U_q} \, + \, {J_{U_q}}^{\!\!\otimes (n-1)} \otimes \Big( \calI_q \phantom{(}^{\!\!\!\perp} \, {\textstyle \bigcap} \; J_{U_q} \Big) \, =  \cr
   \hfill   = \, {J_{U_q}}^{\!\!\otimes (n-1)} \otimes \Big( \calI_q \phantom{(}^{\!\!\!\perp} \, {\textstyle \bigcap} \; J_{U_q} \Big) \, = \, {J_{U_q}}^{\!\!\otimes (n-1)} \otimes \Big( \gerC_q \; {\textstyle \bigcap} \; J_{U_q} \Big) \, \subseteq \, {U_q}^{\otimes (n-1)} \otimes \gerC_q
\cr }
\]
where in the third equality we used the fact that  $ \, J^\perp \bigcap J_{U_q} = \{0\} \, $.  So
$ \, \kappa_+ \in {U_q}^{\!\otimes (n-1)} \otimes \, \gerC_q \, $,  hence  $ \, \delta_n (\kappa) \in {(q-1)}^n {U_q}^{\otimes (n-1)} \otimes \gerC_q \, $  for all  $ \, n \in \N_+ \, $:  thus  $ \, \kappa \in {\gerC_q}^{\!\!\Lsh} \, $.  Therefore  $ \, {\big( {\calI_q}^{\!\curlyvee} \big)}^\perp \subseteq {\gerC_q}^{\!\!\Lsh} \, $, which together with the right-hand side inequality in (\ref{quarantadieci})  gives  $ \, {\gerC_q}^{\!\!\Lsh} = {\big( {\calI_q}^{\!\curlyvee} \big)}^\perp \, $.
                                        \par
   In the end, suppose also that  one between $\gerC_q$ and  $\calI_q$ is strict.  As  $ \, \calI_q = {\gerC_q}^{\!\perp} \, $  and  $ \, \gerC_q = {\calI_q}^{\!\perp} \, $,  one sees easily that  $ \calI_q $  is strict if and only if  $ \gerC_q $  is; then we can assume that  $\gerC_q$ is strict.  We want to show that  $ \, {\calI_q}^{\!\curlyvee} = {\big( {\gerC_q}^{\!\!\Lsh\,} \big)}^\perp \, $.
                                        \par
   The assumptions and Theorem \ref{converse}{\it (d)\/}  give  $ \, \gerC_q = \big( \gerC_q^{\,\Lsh\,} \big)^{\!\triangledown} \, $.  Moreover, we have  that $\calC_q$ is strict  by  Proposition \ref{triangledown}{\it (3)\/}  and  Proposition \ref{Lsh} {\it (3)}.  Then we can apply the first part of claim  {\it (a)},  thus getting  $ \, \big( {\calC_q}^{\!\!\triangledown} \big)^\perp = \, \big( {\calC_q}^{\!\perp} \big)^! \, $.  Therefore
\begin{equation}\label{quarantundici}
  {\calI_q}^{\!\curlyvee}  \, = \;  \big( {\gerC_q}^{\!\perp} \big)^{\!\curlyvee}  \, = \;  \Big( \Big(\! \big( \gerC_q^{\,\Lsh\,} \big)^{\!\triangledown} \Big)^{\!\perp\,} \Big)^{\!\!\curlyvee}  \, = \;  \Big(\! \big( {\calC_q}^{\!\!\triangledown} \big)^\perp \Big)^{\!\!\curlyvee}  \, = \;  \Big(\! \big( {\calC_q}^{\!\perp} \big)^{!\,} \Big)^{\!\!\curlyvee}
\end{equation}
   \indent   Now, one proves easily that  $\calC_q$ strict  implies  ${\calC_q}^{\!\perp}$ strict.  Then  Theorem \ref{converse}{\it (c)\/}  yields  $ \, \Big(\! \big( {\calC_q}^{\!\perp} \big)^{!\,} \Big)^{\!\!\curlyvee} \! = \, {\calC_q}^{\!\perp} \, $.  This and \eqref{quarantundici} give  $ \, {\calI_q}^{\!\curlyvee} \! = \Big(\! \big( {\calC_q}^{\!\perp} \big)^{!\,} \Big)^{\!\!\!\curlyvee} \!\! = {\calC_q}^{\!\perp} \! = \! {\big( {\gerC_q}^{\!\Lsh\,} \big)}^\perp \, $,  which eventually ends the proof of  {\it (b)}.
\end{proof}

\bigskip

\section{Examples}  \label{examples}

In this last section we will give some examples showing how our general constructions may be explicitly implemented. Some of the examples may look rather singular, but our aim here is mainly to draw the reader's attention on how even badly behaved cases can produce reasonable results. It has to be remarked that a wealth of new examples of coisotropic subgroups of Poisson groups have been recently produced (\cite{Zam}), to which our recipes could be interestedly applied.

\medskip

  {\sl N.B.: for the last two examples   --- Subsections  \ref{param-fam_cs}  and  \ref{non-coiso}  ---   one can perform the explicit computations (that we just sketch) using definitions, formulas and notations as in  \cite{CiGa}, \S 6,  and in  \cite{Ga2}, \S 7}.

\medskip

\subsection{Quantization of Stokes matrices as a  $ \text{\it GL}_{\,n}^{\;*} $--space}  \label{Stokes}

 As a first example, we mention the following.  A well-known structure of Poisson group, typically known as the  {\sl standard\/}  one, is defined on  $ \text{\it SL}_{\,n} \, $;  then one can consider its (connected) dual Poisson group  $ \text{\it SL}_{\,n}^{\;*} \, $,  which in turn is a  Poisson group as well.  The set of  {\sl Stokes matrices}   --- i.e.~upper triangular, unipotent matrices ---   of size  $ n $  bears a natural structure of Poisson homogeneous space, and even Poisson quotient, for  $ \text{\it SL}_{\,n}^{\;*} \, $.  In  \cite{CiGa},  Section 6, it was shown that one can find an explicit quantization, of formal type, of this Poisson quotient by a suitable application of the QDP procedure for formal quantizations developed in that paper.
                                                                           \par
    Now, let us look at the explicit presentation of the formal quantization  $ U_\hbar(\mathfrak{sl}_n) $  considered in  [{\it loc.~cit.}].  One sees easily that this  can be turned into a presentation of a  {\sl global\/} quantization (of  $ \mathfrak{sl}_n $  again), i.e.~a QUEA  $ U_q(\mathfrak{sl}_n) $  in the sense of  Section \ref{quantizations}.  Similarly, Drinfeld's QDP (for quantum groups) applied to  $ U_\hbar(\mathfrak{sl}_n) $  provides a formal quantization  $ \, F_\hbar[[\text{\it SL}_{\,n}^{\;*}]] := {U_\hbar(\mathfrak{sl}_n)}' \, $  of the function algebra over the formal group  $ \text{\it SL}_{\,n}^{\;*} \, $;  but then the analogous functor for the global version of QDP yields  (cf.~Theorem \ref{GQDP})  a global quantization  $ \, F_q[\text{\it SL}_{\,n}^{\;*}] := {U_q(\mathfrak{sl}_n)}' \, $  of the function algebra over  $ \text{\it SL}_{\,n}^{\;*} \, $.  In a nutshell,  $ F_q[\text{\it SL}_{\,n}^{\;*}] $  is nothing but (a suitable renormalization of) an obvious  $ \Cqqm $--integral  form of  $ F_\hbar[[\text{\it SL}_{\,n}^{\;*}]] \, $.
                                                                           \par
    Carrying further on this comparison, one can easily see that the whole analysis performed in  \cite{CiGa}  can be converted into a similar analysis for the global context, yielding parallel results; in particular,  {\it one ends up with a global quantization   --- of type  $ \calC $,  in the sense of  Section \ref{quantizations}  ---   of the space of Stokes matrices}.  More in detail, this quantization is a strict one, as such is the quantum subobject one starts with.
                                                                           \par
    Since all this does not require more than a word by word translation, we refrain from filling in details.

\medskip

\subsection{A parametrized family of real coisotropic subgroups}  \label{param-fam_cs}

Coisotropic subgroups may come in families, in some cases inside the same conjugacy class (which is responsible for different Poisson homogeneous bivectors on the same underlying manifold). An example in the real case was described in detail in \cite{BCGST}. The setting is the one of standard Poisson $SL_2(\mathbb R)\,$, which contains a two parameter family of $1$--\,dimensional coisotropic subgroups described, globally, by the right ideal and two-sided ideal
\begin{equation}
\label{coiso-generators}
\calI_{\mu,\nu}  \; := \;  \left\{\, a-d+2\,q^{\frac 1 2}\mu b \; , \; q\,\nu b+c
\,\right\} \cdot F_q\big[SL_2({\mathbb R})\big]
\end{equation}
where $a,b,c,d$ are the usual matrix elements generating $\,F_q\big[SL_2(\mathbb R)\big]\,$, with $*$-- structure in which they are all real (thus $\,q^*=q^{-1}\,$) and $\,\mu,\nu\in {\mathbb R} \, $.  The corresponding family of coisotropic subgroups of classical  $ SL_2({\mathbb R}) $  may be described as
\[
  K_{\mu,\nu}  \; := \;  \bigg\{\! \left(\begin{array}{cc}d-2\mu b & b\\ -\nu b&d\end{array}\right) \,\bigg|\, b,d \in {\mathbb R} \, , \, d^2+\nu b^2=1 \,\bigg\}
\]
(adapting our main text arguments to the case of  {\sl right\/}  quantum coisotropic subgroups, this is quite trivial and we will do it without further comments).  The corresponding $SL_2({\mathbb R})$--quantum  homogeneous spaces have local description given as follows: $\calC_{\mu,\nu}$ is the subalgebra generated by
\begin{equation}  \label{mu-homo}
  \begin{aligned}
     &  z_1 \, = \, q^{-\frac 1 2}(ac+\nu bd)+2\mu bc \; ,  \quad  z_2 \, = \, c^2+\nu d^2+2\mu q^{-\frac 1 2}cd \; ,  \\
     &  z_3 \, = \, a^2+\nu b^2+2\mu q^{-\frac 1 2}ab \; .
  \end{aligned}
\end{equation}

Using commutation relations   --- see (12) in \cite{BCGT}  ---   it is easily seen that $\calC_{\mu,\nu}$ has a linear basis given by  $ \; \big\{\, z_1^p z_2^q \, , \, z_1^p z_3^r \,\big|\, p, q, r \in \N \,\big\} \; $.

\medskip

\begin{proposition}
The subalgebra  $\calC_{\mu,\nu}$ is a right coideal of  $ F_q\big[SL_2({\mathbb R})\big] $  and is a strict quantization   --- of type  $ \calC $  ---   of  $ K_{\mu,\nu} \, $.
\end{proposition}

\begin{proof}
 The first statement is proven in \cite{BCGT}. As for the second we will first show that $ \, z_1^p z_2^q\,, z_1^p z_3^r \not\in (q-1) F_q\big[SL_2({\mathbb R})\big] \, $  for any  $ \, p,q,r\in\N \, $. This may done by considering their expression in terms of the usual basis  $ \, \big\{\, a^p b^r c^s \, , \, b^h c^k d^i \,\big\} \, $  of $F_q\big[SL_2({\mathbb R})\big]\,$. In fact we do not need a full expression of monomials  $ z_1^p z_2^r $  or  $ z_1^p z_3^r $  in terms of this basis, which would lead to quite heavy computations. It is enough to remark that, for example, since
\[
z_1^p z_2^r  \; = \;  \left(q^{-\frac 1 2}ac+b(\nu d+2\mu c)\right)^p \left(c^2+(\nu d+2\mu q^{-\frac 1 2}c)d\right)^r
\]
we can get an element multiple of  $ \, a^p c^{p+2r} \, $  only from  $ \, (ac)\cdot\cdots (ac)\cdot c\cdots \cdot c \, $,  which is of the form  $ \, q^h a^p c^{p+2r} \not\in F_q\big[SL_2({\mathbb R})\big] \, $. Since no other elements may add up with this one, we have  $ \, z_1^p z_2^r \not\in (q-1) F_q\big[SL_2({\mathbb R})\big] \, $.  A similar argument works for  $ z_1^p z_3^r \, $.
                                                                        \par
   In a similar way we prove that any  $ \Cqqm $--linear  combination of the  $ z_1^p z_2^q $'s  and the  $ z_1^s z_3^r $'s  is in  $ \, (q-1) F_q\big[SL_2({\mathbb R})\big] \, $  if and only if all coefficients are in  $ \, (q-1) \Cqqm \, $.  Therefore  $ \calC_q $  is strict, q.e.d.
\end{proof}

\smallskip

   It makes therefore sense to compute  $ \, \calC_{\mu, \nu}^{\,\,\triangledown} \, $;  to this end, we can resume a detailed description of  $ \, U_q(\gersl_2^{\,*}) := {F_q\big[SL_2({\mathbb R})\big]}^\vee \, $   --- apart for the real structure, which is not really relevant here ---   from  \cite{Ga2},  \S 7.7.  From our PBW-type basis we have that  $ \calC_{\mu,\nu}^{\,\,\triangledown} $  is the subalgebra of  $ {F_q\big[SL_2({\mathbb R})\big]}^\vee $  generated by the elements  $ \, \zeta_i := {\frac{1}{\,q-1\,}} \big( z_i - \varepsilon(z_i) \big) \in {F_q\big[SL_2(\mathbb R)\big]}^\vee \, (i=1,2,3) \, $.  Since we know that
  $$  H_+ \, := \, \frac{\,a-1\,}{\,q-1\,}  \; ,  \qquad  E \, := \, \frac{\,b\,}{\,q-1\,}  \; ,  \qquad
  F  \, := \,  \frac{\,c\,}{\,q-1\,} \; ,  \qquad  H_-  \, := \,  \frac{\,d-1\,}{\,q-1\,}  $$
are algebra generators of  $ \, U_q(\gersl_2^{\,*}) := {F_q\big[SL_2({\mathbb R})\big]}^\vee \, $,  we deduce that
\begin{equation}\label{dualmunu}
  \begin{aligned}
\frac{\,\zeta_1\,}{\,q-1\,}  \,\; =  &  \;\;\, q^{-\frac 1 2}(F+\nu E) + (q-1) \left(q^{-\frac 1 2} H_+F+q^{-\frac 1 2} \nu E H_-+2\mu E F \right)  \\
\frac{\,\zeta_2-\nu\,}{\,q-1\,}  \,\; =  &  \;\;\, 2 \, (\nu H_- +\mu q^{-\frac 1 2}F) + (q-1) \left( F^2+\nu H_-^2+2\mu q^{-\frac 1 2} F H_- \right)  \\
\frac{\,\zeta_3-1\,}{\,q-1\,}  \,\; =  &  \;\;\, 2 \, (H_++\mu q^{-\frac 1 2}E) + (q-1) \left( H_+^2+\nu E^2+2\mu q^{-\frac 1 2} H_+ E \right)
  \end{aligned}
\end{equation}
In the semiclassical specialization  $ \; U_q(\gersl_2^{\,*}) \,{\buildrel {q \longrightarrow 1} \over {\, \relbar\joinrel\longrightarrow\,}}\, U_q(\gersl_2^{\,*}) \Big/ (q-1) U_q(\gersl_2^{\,*}) \; $  one has that  $ \, E \mapsto \text{\rm e} \, $,  $ \, F \mapsto \text{\rm f} \, $,  $ \, H_\pm \mapsto \pm \text{\rm h} \, $,  where  $ \, \text{\rm h}, \text{\rm e}, \text{\rm f} \, $  are Lie algebra generators of  $ \gersl_2^* \, $;  therefore the semiclassical limit of the right hand side of \eqref{dualmunu} is the Lie subalgebra generated by  $ \, \text{\rm f} + \nu \, \text{\rm e} \, $,  $ \, -\nu \, \text{\rm h} + \mu \, \text{\rm e} \, $,  $ \, \text{\rm h} + \mu \, \text{\rm e} \, $,  or, equivalently, the 2--dimensional Lie subalgebra generated by  $ \, \text{\rm f} + \nu \, \text{\rm e} \, $  and  $ \, \text{\rm h} + \mu \, \text{\rm e} \, $  (the three elements above being linearly dependent) with relation  $ \, [\, \text{\rm h} + \mu \, \text{\rm e} \, , \text{\rm f} + \nu \, \text{\rm e} \,] = \text{\rm f} + \nu \, \text{\rm e} \; $.
The quantization of this coisotropic subalgebra of $\gersl_2^{\,*}$ is therefore the subalgebra generated inside  $ U_q(\gersl_2^{\,*}) $  by the quadratic elements \eqref{dualmunu}.
                                                                          \par
   Similar computations can be performed starting from $\calI_{\mu,\nu}\,$. The transformed  $\calI_{\mu,\nu}^\curlyvee$ is the right ideal generated by the image of  $ \, a-d+2\,q^{\frac 1 2}\mu b \, $  and  $ \, q\nu b+c \, $,  i.e.~the right ideal generated by  $ \, H_+-H_- + 2 \, q^{\frac 1 2} \mu E \, $  and  $ \, q \, \nu E + F \, $;  also, from its semiclassical limit it is easily seen that this again corresponds to the same coisotropic subgroup of the dual Poisson group  $ {\text{\it SL}_{\,2}(\mathbb R)}^* \, $.
                                                                          \par
   All this gives a local   --- i.e., infinitesimal ---   description of the ($2$--dimensional) coisotropic subgroups  $ K^{\,\,\bot}_{\mu,\nu} $  in  $ {\text{\it SL}_{\,2}(\mathbb R)}^* \, $.

\medskip

\subsection{The non coisotropic case}  \label{non-coiso}

Let us finally consider the case of a non coisotropic subgroup. We will consider the embedding of $SL_2(\C)$ into $SL_3(\C)$ corresponding to a non simple root, which easily generalizes to higher dimensions. Computations will only be sketched.

Let $\gerh$ be the subalgebra of $\gersl_3(\C)$ spanned by $E_{1,3}\,$, $F_{1,3}\,$,  $ \, H_{1,3}=H_1+H_2 \, $. Easy computations show that the standard cobracket values are
\begin{equation}\label{cobracket non simple}
  \begin{aligned}
     \delta(E_{13})  \; =  &  \;\; E_{13}\wedge (H_1+H_2)+2E_{23}\wedge E_{12}\\
     \delta(F_{13})  \; =  &  \;\; F_{13}\wedge (H_1+H_2)-2F_{23}\wedge F_{12}\\
     \delta(H_1+H_2)  \; =  &  \;\; 0
  \end{aligned}
\end{equation}
and, therefore, the corresponding embedding  $ \; SL_2(\C) \lhook\joinrel\longrightarrow SL_3(\C) \; $  is  {\sl not\/}  coisotropic. To compute the coisotropic interior $\gerho$ of $\gerh\,$, consider that  $ \, \langle H_1 + H_2 \rangle \, $  is, trivially, a subbialgebra of $\gerh\,$, thus contained in $\gerho\,$.  Let  $ \, X := (H_1+H_2)+\alpha E_{13}+\beta F_{13} \; $:
 then
\[
  \delta(X)  \; = \;  X\wedge(H_1+H_2)+2\left(\alpha E_{23}\wedge E_{12}-\beta F_{23}\wedge F_{12}\right)
\]
shows that no such $X$ is in $\gerho\,$, unless $\,\alpha=0=\beta\,$.
 The outcome is that we have
  $$  \stackrel{\circ}{H} \,\; = \; \left(\begin{array}{ccc}
          \gamma  &  0  &     0    \\
             0    &  1  &     0    \\
             0    &  0  &  \gamma^{-1}
                                          \end{array}\right)
\; \subseteq \; SL_3(\C)  $$
with  $ \, \gamma \in \C^* \, $.  Correspondingly
\[
\gerh^{\langle\bot\rangle} \, = \, \Big(\gerho\Big)^\bot \, = \; \big\langle\, \text{\rm e}_{1,2} \, , \, \text{\rm e}_{1,3} \, , \, \text{\rm e}_{2,3} \, , \, \text{\rm f}_{1,2} \, , \, \text{\rm f}_{1,3} \, , \, \text{\rm f}_{2,3} \, , \, \text{\rm h}_{2.2} \,\big\rangle  \qquad \Big(\, \subseteq \, {\mathfrak{sl}_3(\C)}^* \,\Big)
\]
and, thus  $ \, {SL_3(\C)}^* \Big/ \! H^{\langle\bot\rangle} \, $  is a $1$--\,dimensional Poisson homogeneous space   --- with, of course, zero Poisson bracket.
                                                               \par
   Let us consider now any weak quantization $\gerC_q(H)$ of $H$. It should certainly contain the subalgebra of $U_q(\gersl_3)$ generated by the root vectors $E_{1,3}\,$, $F_{1,3}\,$, together with  $ \, K_1 K_3^{-1} \, $  and  $ \, \widehat{H}_{1,3} := \big( K_1 K_3^{-1} - 1 \big) \Big/ (q-1) \; $.  The equality
\[
\Delta(E_{1,3})  \; = \;  E_{1,3} \otimes K_1 K_3^{-1} + 1 \otimes E_{1,3} + (q-1) E_{1,2} \otimes E_{2,3}
\]
tells us that, in order to be a left coideal, such a quantization should also contain either  $ \, (q-1) E_{1,2} \, $  or  $ \, (q-1) E_{2,3} \, $  (and thus, as expected, it cannot be strict). Let us try to compute some elements in ${\gerC_q(H)}^{\!\Lsh}\,$.  Certainly, since
\[
  \delta_2\big(\widehat{H}_{1,3}\big)  \; = \;  \widehat{H}_{1,3} \otimes \big( K_1 K_3^{-1} - 1 \big)  \; = \;  (q-1) \, \widehat{H}_{1,3} \otimes \widehat{H}_{1,3}
\]
we can conclude that  $ \, (q-1) \widehat{H}_{1,3} \in {\gerC_q(H)}^{\!\Lsh} \, $.  On the other hand,
\[
  \delta_2(E_{1,3})  \; = \;  (q-1) E_{1,3} \otimes \widehat{H}_{1,3} + (q-1) E_{1,2} \otimes E_{2,3}
\]
implies that  $ \, (q-1) E_{1,3} \not\in {\gerC_q(H)}^{\!\Lsh} \, $,  while  $ \, (q-1)^2 E_{1,3} \in {\gerC_q(H)}^{\!\Lsh} \, $.

\vskip11pt

   All this means the following.

\vskip7pt

   Within  $ {\gerC_q(H)}^{\!\Lsh} $  we find a non-diagonal matrix element of the form  $ \, (q\!-\!1) \, t_{1,3} \, $:  it belong to  $ \, (q\!-\!1) {U_q(\gersl_3)}' \, $  but not to  $ \, (q\!-\!1) {\gerC_q(H)}^{\!\Lsh} $,  so that
  $$  {\gerC_q(H)}^{\!\Lsh} \,{\textstyle \bigcap}\, (q\!-\!1) \, {U_q(\gersl_3)}' \,\; \supsetneqq \;\, (q\!-\!1) \, {\gerC_q(H)}^{\!\Lsh}  $$
which means that the quantization  $ {\gerC_q(H)}^{\!\Lsh} $  {\it is not strict}.  On the other hand, we know by  Proposition \ref{Lsh}{\it (3)\/}  that  $ {\gerC_q(H)}^{\!\Lsh} $  is  {\sl proper}.  Therefore,  {\it we have an example of a quantization (of type  $ \calC_q \, $,  still by  Proposition \ref{Lsh}{\it (3)})  which is  {\sl proper},  yet it is  {\sl not strict}}.
 \vskip9pt
   In addition, in the specialization map  $ \, \pi : {U_q(\gersl_3)}' \relbar\joinrel\relbar\joinrel\twoheadrightarrow\, {U_q(\gersl_3)}' \Big/ (q-1) {U_q(\gersl_3)}' \, $  the element  $ \, (q\!-\!1) \, t_{1,3} \, $  is mapped to zero, i.e.~it yields a trivial contribution to the semiclassical limit of  $ {\gerC_q(H)}^{\!\Lsh} $   --- which here is meant as being  $ \; \pi\big( {\gerC_q(H)}^{\!\Lsh} \,\big) = {\gerC_q(H)}^{\!\Lsh} \! \Big/ {\gerC_q(H)}^{\!\Lsh} \bigcap \, (q-1) \, {U_q(\gersl_3)}' \; $.  With similar computations it is possible to prove, in fact, that the only generating element in  $ {\gerC(H)}^{\!\Lsh}$  having a non-trivial semiclassical limit is  $ \, (q-1) \widehat{H}_{1,3} \, $.  Therefore, through specialization at  $ \, q = 1 \, $,  from  ${\gerC(H)}^{\!\Lsh}$  one gets only  $ \, \pi\big({\gerC_q(H)}^{\!\Lsh}\big) = \C\big[t_{2,2}\big] \, $:  indeed, this in turn tells us exactly that  $ {\gerC_q(H)}^{\!\Lsh} $  is a quantization, of  {\sl proper\/}  type, of the homogeneous  $ {SL_3(\C)}^* $--space $ \, {SL_3(\C)}^* \!\Big/ H^{\langle\bot\rangle} \, $  (whose Poisson bracket is trivial).

\vskip23pt

   {\bf Remark.}  It is worth stressing that this example   --- no matter how rephrased ---   could not be developed in the language of formal quantizations as a direct application of the construction in  \cite{CiGa},  for only strict quantizations were taken into account there.

\end{document}